\documentclass{amsart}
\usepackage{anysize}

\marginsize{2,5 cm}{2,0 cm}{2,0 cm}{2,0 cm}

\usepackage{amsmath, amsfonts,amsbsy,amssymb,amscd}
\usepackage{t1enc}
\usepackage[cp1250]{inputenc}
\usepackage[british]{babel}
\usepackage[all]{xy}
\usepackage{color}
\usepackage{amsfonts}
\usepackage{latexsym}
\usepackage{mathrsfs}
\usepackage{hyperref}
\usepackage{graphicx}

\usepackage{color}
\usepackage{enumerate}
\usepackage{bbm}
\definecolor{orange}{rgb}{1,0.5,0}

\newcommand{\bes}{\begin{equation*}}
\newcommand{\ees}{\end{equation*}}

\usepackage{mathrsfs} 

\DeclareMathAlphabet{\mathpzc}{OT1}{pzc}{L}{it} 

\newtheorem{thm}{Theorem}[section] 
\newtheorem{lemma}[thm]{Lemma}     
\newtheorem{cor}[thm]{Corollary}
\newtheorem{prop}[thm]{Proposition}

\numberwithin{equation}{section}

\def\R{\mathrm{Re\,}}

\def\C{\mathbb{C}}
\def\geq{\geqslant}
\def\leq{\leqslant}
\def\R{\mathbb{R}}

\def\Z{\mathcal{Z}}
\def\Cb{B}
\def\N{\mathbb{N}}

\def\cB{\mathcal{B}}

\def\cC{\mathcal{C}}

\def\V{\mathcal V}
\def\pSymLog{\mathcal PSymLog}

\newcommand{\bea}{\begin{eqnarray}}
  \newcommand{\eea}{\end{eqnarray}}
  \newcommand{\beab}{\begin{eqnarray*}}
  \newcommand{\eeab}{\end{eqnarray*}}

  \newcommand{\be}{\begin{equation}}
  \newcommand{\ee}{\end{equation}}

\newcommand{\la}{\lambda}
\newtheorem{defn}[thm]{Definition}
\newtheorem{remark}[thm]{Remark}
\title[Disjointness of rescalings]{Disjointness of rescalings of smooth area preserving flows on surfaces}

\author[P.\ Berk]{Przemyslaw Berk}
\address{Institut f\"ur Mathematik, Universit\"at Z\"urich, Winterthurerstrasse 190,
CH-8057 Z\"urich, Switzerland}
\email{prze.berk@gmail.com}

\author[C.\ Ulcigrai]{Corinna Ulcigrai}
\address{Institut f\"ur Mathematik, Universit\"at Z\"urich, Winterthurerstrasse 190,
CH-8057 Z\"urich, Switzerland}
\email{corinna.ulcigrai@math.uzh.ch}

\begin{document}
\baselineskip=14pt

\begin{abstract} 
We consider the problem of \emph{disjointness of rescalings} $(\varphi_{\kappa t})_{t\in \mathbb{R}}$, $\kappa\in\R$ 
of a flow $(\varphi_{t})_{t\in \R}$ in the context of   smooth flows preserving 
a smooth invariant measure, or, equivalently, locally Hamiltonian flows on 
compact orientable surfaces. We show that, when the genus of the surface is 
$g\ge 2$, almost every  locally Hamiltonian flow with 2g-2 non-degenerate 
simple saddles is such that any distinct two rational rescalings 
$(\varphi_{\kappa t})_{t\in \mathbb{R}}$ and $(\varphi_{\kappa' t})_{t\in 
\mathbb{R}}$, with $\kappa=p/q$ and  $\kappa'=p'/q'$ of different absolute values, are  disjoint. 
Previous results on disjointness of rescalings were available only for 
locally Hamiltonian flows and their special flow representations in genus one. 


 The result is proved using a criterion for disjointness based on the study of the distribution  of Birkhoff sums of a special representation and in particular estimates on their exponential tails decay.

A key novel geometric ingredient in the proof is the existence of a sequence of rigidity times which display what we call bounded-type rigidity, so that a large set of points comes back in time $q$ with distance $O(1/q)$. To produce such bounded-type rigidity times we exploit  a particular way of degeneration of a translation surface to a flat torus for which the vertical flow has bounded-type rotation number.
\end{abstract}

\maketitle


\section{Introduction}
{
We push in this paper the investigation of \emph{disjointness} in the context of smooth area-preserving flows on surfaces. The definition of disjointness and the question of disjointness among flows and within their time-changes goes back to seminar work by Furstenberg \cite{Fur}, see e.g.~\cite{Glas} and the references therein. 
The notion of \emph{disjointness of rescalings} more specifically has seen a revival of interest in recent years, in particular in connection with Sarnak's Moebius orthogonality conjecture and in the context of parabolic dynamics (see \S~\ref{sec:Moebius}). We pursue in this paper the investigation of disjointness of rescalings as a fundamental property of parabolic flows, by proving its typicality among a class of smooth-area preserving flows on surfaces. Our main result is stated in \S~\ref{sec:main}.

\subsection{Disjointness in parabolic dynamics}
Let us first recall the notion of disjointness of rescalings and its relevance. }
\subsubsection{Disjointness of rescalings}
Given a (measure-preserving) flow $\varphi_\R:= (\varphi_t)_{t\in \R}$ on a measure space $(X,\mu)$, we can investigate what happens when we \emph{rescale} time linearly, namely consider, for a given $\kappa\neq 0$, the flow  $\varphi^\kappa_\R:= (\varphi^{\kappa}_t)_{t\in \R}$ where $\varphi^{\kappa}_t:= \varphi_{\kappa t}$.  Even such a simple time-reparametrization can drastically change the ergodic theoretical properties of the flow. In particular, a non-trivial rescaling $\varphi^\kappa_\R$ (where by non-trivial we mean that $\kappa\neq 1$) can fail to be \emph{isomorphic} to the original flow. \emph{Disjointness} can be seen as a strong form of lack of isomorphism. There are several notions of disjointness, for example \emph{spectral disjointness} (see   \S~\ref{sec:disjointness}).  
We will focus in this paper on \emph{disjointness in the sense of Furstenberg}, or absence of common joinings (see \S~\ref{sec:disjointness} for definitions), a notion introduced by Furstenberg in \cite{Fur}. 

\subsubsection{Disjointness to study Moebius orthogonality}\label{sec:Moebius}
A surge of activity in the investigation of \emph{disjointness} in several classes of entropy zero flows, was initially motivated by a famous conjecture by Sarnak \cite{Sa, Sa:Af} on orthogonality of the Moebius functions with systems of zero entropy, see for example the survey \cite{surveyMoebius} and the references therein. Disjointness of rescalings has been exploited in several recent results (starting from the seminal work in this direction \cite{BSZ} by Bourgain, Sarnak and Ziegler on horocycle flows) as a key tool which implies orthogonality with the Moebius (and other multiplicative) functions. 
 
This recent activity on disjointness  has at the same time led to the realization that 
the property of \emph{disjointness of rescalings} may be a distinctive phenomenon for  \emph{parabolic} flows (see for example \cite{KLU} and the references therein).  It is in this spirit that we push the investigation of this property within the fundamental class of smooth area-preserving flows on surfaces. 

As a disclaimer, although several results on disjointness of rescalings were also proved for entropy zero \emph{maps}  (see e.g.~\cite{CE} and \cite{BK} for reducible IETs or \cite{ALL} for rank one transformations)
 we focus in this paper on \emph{flows}, and hence provide only a brief survey of the literature 
on disjointness in entropy zero \emph{flows}, known as \emph{parabolic flows}.

\subsubsection{Parabolic flows}
Parabolic flows (and more in general parabolic dynamical systems) are flows which display  \emph{slow} sensitive dependence on initial conditions, in the sense that the speed of divergence of nearby orbits is subexponential, typically polynomial or subpolynomial (in contrast with \emph{hyperbolic flows}, which exhibit exponential divergence), see for example the surveys  \cite{Fo:ICM, Ul:ECM}.   Classical examples of parabolic flows include the \emph{horocycle} flow on a surface with constant negative curvature and its time-changes, time-changes of nilflows on nilmanifolds and smooth area-preserving flows on surfaces (for which we refer the interested reader to the surveys  \cite{Ul:ICM, Fr-Ul:Enc} and the references therein). 

{ 
One of the key discoveries by Marina Ratner in her seminal work in homogeneous dynamics was the phenomenon of \emph{joining rigidity} of (time-changes of) horocycle and unipotent flows: informally, two such  flows (or rescalings) are disjoint unless there is an algebraic reason for them to be isomorphic, see e.g.~\cite{Rat:Act, Rat:Inv}. The investigation of joining rigidity of time-changes of horocycle and unipotent flow is still an active area of research, as demonstrated by the several results recently proved, see  for example \cite{FF, AFR, DKW, T1, T2}. 



If we consider disjointness of \emph{rescalings},
the  standard horocycle flow $h_\R$ on $SL(2,\R)/SL(2,\mathbb{Z})$ \emph{fails} to have such property: indeed  $h_\R$  is \emph{conjugated} to (and hence in particular \emph{not disjoint} from)  all its rescalings (via the geodesic flow $g_t$, in view of the relation $g_t\circ h_{s}= h_{e^t}s\circ g_t$). 
  However, for many non-trivial \emph{time-changes} (also called~time \emph{reparametrizations}) of horocycle flows, different \emph{rescalings}  (i.e. $ p\neq q$) are disjoint, see \cite{KLU, FF}. This can be interpreted by saying that the (standard) horocycle flow (in view of its homogeneity and hence additional symmetries) is rather an \emph{exception} among reparametrizations of horocycle flows: \emph{typical} {time-changes} \emph{do} display disjointness of rescalings. 
  

Results on disjointness and disjointness of rescalings for other classes of parabolic flows are not many. In the context of nilflows on nilmanifolds, Furstenberg disjointness of rescalings was only proved for typical time-changes of special (bounded type) Heisenberg nilflows by Forni and Kanigowski  \cite{FoKa}. The problem for higher rank Heisenberg flows, and even for typical (Diophantine) Heisenberg nilflows, remains currently widely open.

\subsubsection{Disjoint rescalings for surface flows}
In the context of smooth area-preserving flows, some results on disjointness of rescalings were proved, but up to now (to the best of our knowledge), typicality of this property  was only shown to hold in some  genus one settings, that we now recall.}
  
In  \cite{KLU}, 
 Kanigowski, Lemanczyk and the second author  proved  that Furstenberg disjointness of rescalings holds for typical \emph{Arnold flows} on genus one surfaces: these are given by (the restriction to the unique minimal component of) a flow on a torus with one center and one saddle.  
 Arnold flows admit a representation 
  as special flows over a rotation of the circle, under a roof function with so called \emph{logarithmic asymmetric singularities} (see \S~\ref{sec:reduction}). 

 The same type of special representation, when the base is an interval exchange transformation, provides a description of minimal components of locally Hamiltonian flows on higher genus surfaces (in presence of a saddle loop homologous to zero). A typical flow on a higher genus surface (i.e.~a surface with genus $g\geq 2$) which is \emph{minimal}, on the other hand, admits a representation
as a special flows over a IET 
under a roof function with \emph{symmetric} logarithmic singularities (see the definitions in \S~\ref{sec:roofs}). 
The special case in which the base is a rotation (which does not correspond to any locally Hamiltonian flow, but provides the prototype for the study of minimal locally Hamiltonian flows in higher genus) was studied by Kanigowski and the first author in \cite{BK}, where they showed that distinct \emph{rational} rescalings (i.e.~rescalings of the form $\varphi^\kappa_\R$ where $\kappa=p/q\in\mathbb{Q}$) of a typical such flow are disjoint.

Another class of special flows 
 for which spectral disjointness of 
 powers was proved, namely so-called  {\it von Neumann flows}. This class (introduced in \cite{vN} to produce examples of systems with continuous spectrum) 
contains special flows over rotations under a piecewise linear roof; 
 one can show that they arise from the study of linear flows on surfaces with boundary, see \cite{Fr-Le3,Co-Fr}.  
In \cite{DK}, Kanigowski and 
 Dong 
 showed that \emph{any} two different rescalings of a Von Neumann flow are disjoint, unless they are isomorphic. In \cite{BK}, Kanigowski and the first author also consider more generally \emph{von Neumann flows} in higher genus (special flows over interval exchanges under a piece-wise linear roof). They prove that for a typical IET, distinct rational rescalings of any special flow with a piece-wise linear roof function, are disjoint.

\subsection{Main result and open questions}\label{sec:main}
The  results summarized above indicate that disjointness of rescalings may be a typical feature of parabolic flows and lead to the natural program of pushing the investigation of this property further in the realm of parabolic dynamics. The first natural question is whether disjointness of rescalings (in Furstenberg or spectral sense, for rational or almost every rescaling) can be proved for smooth area preserving flows on surfaces of higher genus, i.e.~of genus $g\geq 2$. 
 Our main result is the following:
\begin{thm}[Disjointness of rational rescalings]\label{thm:main}
For a \emph{typical}  locally Hamiltonian flow $\varphi_\R$ on a surface $M$ of 
any genus $g\geq 2$ with only simple saddles, any two rational distinct rescalings $\varphi^{\kappa}_\R$ and $\varphi^{\kappa'}_\R $ with $\kappa,\kappa'\in \mathbb{Q}$, such that $|\kappa|\neq|\kappa'|$, are \emph{disjoint} in the sense of Furstenberg. 
\end{thm}
We will provide below the definition of \emph{disjoint} in the sense of Furstenberg. The notion of \emph{typical}, used here in a \emph{measure theoretic sense}, i.e.~gives a full measure set with respect to a natural measure class on locally Hamiltonian flows with given singularity types, sometimes referred to as the Katok fundamental class;   
see \S\ref{sec:measureclass} for definitions. 

\smallskip
%

{Let us mention a different disjointness result (from some point of view \emph{weaker}, although  of spectral nature) which was recently proved in higher genus.  
Instead of asking whether \emph{time-changes} of a typical flows in a certain class are disjoint, or have disjoint rescalings, one can first ask,  whether two typical flows from a given (parabolic) \emph{class} are disjoint.  
Under the assumption that all singularities are  \emph{non-degenerate}, i.e.~that all saddles are \emph{simple},  within the open set\footnote{Here \emph{open} refers to the topology given by small perturbations by a closed $1$-form. Locally Hamiltonian flows 
with only \emph{non-degenerate} (or Morse-type) singularities are indeed divided 
in two open sets, often denoted in the literature by $\mathcal{U}_{min}$  and $\mathcal{U}_{\neg min}$: in  $\mathcal{U}_{min}$ the typical flow is \emph{minimal} (as well as \emph{weak mixing} \cite{Ul:wea},\label{ft:wm} but not \emph{mixing} \cite{Ul:abs}).   
In  the open set  $\mathcal{U}_{\neg min}$ on the other hand, the typical flow decomposes into (open) union of closed orbits and   
 invariant subsurfaces (with boundary) on which the restriction of the flow is minimal (the \emph{minimal components}, which are at most $g$, where $g$ is the genus), see e.g.~\cite{Ma:tra}, on which the flow is 
\emph{mixing} \cite{Ul:mix, Rav:mix}, as well as mixing of all orders \cite{KKU}.

\label{footnote:opensets}} 
  of  minimal locally Hamiltonian flows 
  (or equivalently special flows under roofs with logarithmic \emph{symmetric} singularities), 
it was very recently proved in \cite{FKU} by Fr\k{a}czek, Kanigowski and the second author that   
 the typical \emph{pair} of such flows is  spectrally disjoint.  

It is also natural to ask whether \emph{spectral} (rather than Furstenberg) disjointness holds,  at least  in the genus two setting. Although our proof is inspired by the proof by Kanigowski and the first author of \emph{spectral} disjointness of rescalings (i.e.~we use the same type of disjointness criterion) for their analogous genus one setting (namely special flows over rotations under a symmetric logarithm), we pursue a different strategy (as explained in \S~\ref{sec:strategy}) because of the difficulties of controlling the tails in the higher genus case\footnote{One of the crucial differences when passing from rotations (that can be seen as interval exchange transformations with $d=2$ intervals) to \emph{interval exchange transformations} (IETs) with $d\geq 2$ intervals is that the latter fail \emph{Denjoy-Koksma} inequalities and the cohomological equations cannot in general be solved. This leads to a \emph{centering} problem of special Birkhoff sums, see \S~\ref{sec:strategy}).}, which leads only to the proof of Furstenberg disjointness. 


Finally, the question of disjointness of rescalings should also be investigated for the generalization of \emph{Arnold flows} (i.e.~special flows under asymmetric logarithmic roofs, see \S~\ref{sec:roofs}) to higher genus $g\geq 2$, i.e.~\emph{minimal components} of 
 locally Hamiltonian flows with saddle loops homologous to zero\footnote{This is the open set often denoted in the literature as $\mathcal{U}_{\neg min}$, see footnote~\ref{footnote:opensets}.}. In this setting, Furstenberg disjointness of rescalings 
 is an open problem as soon as $g>1$ (it was only proved in
 \cite{KLU} for Arnold flows, i.e.~for $g=1$). The stronger notion of \emph{spectral} disjointness of rescalings is open even for Arnold flows with genus one (i.e.~special flows over rotations). 
This seems a difficult problem, since already the nature of the spectrum of Arnold flows is an open question (see e.g.~\cite{FKU} and the spectral questions therein).

\subsection{Strategy of the proof and bounded-type rigidity}\label{sec:strategy}
{We now conclude the introduction by highlighting some of the novel ideas in 
the proof. In particular, we introduce the use of what we call \emph{bounded type rigidity} 
times, which we build for an IET using a geometrically based  mechanism of  
\emph{degeneration} to a two-tower rigidity (produced constructing 
renormalization times where a surface, of which the IET is a Poincar{\'e} 
section, degenerates to a torus with a bounded type irrational flow, see 
\S~\ref{sec:degeneration} for details). }
 
\subsubsection{Reduction to Birkhoff sums distributions}
To study ergodic and spectral properties of locally Hamiltonian flows, it is 
standard to exploit their representation as special flows over an IET (or a 
rotation when $g=1$) that we will denote by $T$. The growth of the 
\emph{Birkhoff sums} $S_n f = \sum_{k=0}^{n-1} f\circ T^k$ of the roof function 
$f$ and its derivatives plays a crucial role in the proof of properties such as 
mixing, weak mixing, multiple mixing, as well as spectral and disjointness 
properties  (see e.g.~\cite{Ul:wea, SK:mix, Ul:mix, Rav:mix, KKU, FFK, KLU, 
FKU}). 
Similarly, our approach to disjointness of rescalings is also based on the study of Birkhoff sums $S_n f $, but rather than of estimates on the growth of their derivatives (which are crucial to study mixing, as well as Ratner forms of shearing, as in \cite{Rav:mix, KKU}),  
 we are interested in the \emph{distribution} of $S_n f-n\int\! f $ (seen as a random variable where the initial point $x$ is chosen uniformly with respect to Lebesgue measure). 

\subsubsection{The disjointness criterion}\label{sec:intro_disj_crit}
The \emph{criterion} that we use for proving \emph{disjointness} of special flows (stated in \S\ref{sec:disj_sf}) comes from \cite{Fr-Le1} as well as from \cite{BFr1}
and is based on the study of the \emph{tails} of Birkhoff sums distributions at times where they display a strong form of \emph{tightness}: 
heuristically, if one can find a sequence of \emph{rigidity times}  $(r_n)_{n\in \N}$ such that the Birkhoff sums $S_{r_n}(f)$ are \emph{tight} (in the sense of random variables) and their tail \emph{decays exponentially} and, furthermore, multiples $S_{2 r_n}(f)$, $S_{3r_n}(f), \dots, S_{k r_n}(f)$ have \emph{tails} of the same order,  
one can show that the flow $\varphi_\R$ and its rescaling $\varphi^k_\R$ are disjoint (Proposition~\ref{prof:disjointnesssf} provides the formal statement). 

We remark that 
 \emph{tightness} of Birkhoff sums is a feature of 
flows which display \emph{absence of mixing}.  In a weaker form (tightness along \emph{partial} rigidity sets), it suffices indeed to prevent mixing\footnote{An important early criterion for absence of mixing appears in Katok's work \cite{Ka:int}, which shows 
 that special flows over IETs under roof functions of bounded variation are never mixing, and by Kochergin's, which shows the absence of mixing for special flows over rotations under a roof with a symmetric logarithmic singularity (see \cite{Ko:abs, Ko:07}). Both criteria require as input  \emph{tightness} of Birkhoff sums along some subsequences of \emph{rigidity} (or \emph{partial rigidity}) \emph{times}, i.e.\ one has to show that there exists a sequence $(q_n)$ of times such that $T^{q_n}$ converges to identity on subsets $E_n $ of measure tending to one (if there is rigidity, or measure bounded below in the case of partial rigidity) and at the same time, for some centralizing sequence $(a_n)$ and uniform constant $C$,
 $Leb\{ x\in E_n | \ |S_{q_n}(f)(x) - a_n|<C\}/Leb(E_n) \to 1 $.}
see e.g. \cite{Ka:int, Ko:abs, Ko:07, Sch:abs, Ul:abs}. An idea that goes back to 
work of  Fr\k{a}czek and Lema\'nczyk 
 \cite{Fr-Le2},  is that if in addition  
to \emph{tightness}, one can also control the \emph{tails} of the distribution of the centralized Birkhoff sums $S_{q_n}(f)(x) - a_n$, one can prove much stronger results (using joinings and Markov operators), in particular spectral disjointness from mixing flows and singularity of the spectrum. This approach lead to the proof of singularity of the spectrum in genus two \cite{Ch-Fr-Ka-Ul} and recently in any genus \cite{FKU}.

{
\subsubsection{Bounded type degeneration} 
The main difficulty that we had to overcome is that these by now \emph{classical} rigidity times are \emph{not suitable} to study and \emph{differentiate} the (exponentially small) tails in higher genus 
to verify the assumptions for disjointness.  Roughly speaking, if $(q_n)_{n\in\N}$ is such a sequence of rigidity times, then the values of Birkhoff sums $S_{q_n}(f)(T^{kq_n}x)$ are essentially the same for a typical point from the interval. 
To be able to  distinguish the tails of distributions Birkhoff sums for different rescalings, we introduce and exploit   a new mechanism for tightness: 
instead than tightness times 
  produced essentially using one large Rohlin tower (or, at the surface level, a large \emph{almost cylinder}), we exploit a new type of rigidity times, which we call \emph{bounded-type} rigidity times (or for short, BT rigid times). The name refers to the type of  rigidity displayed by a \emph{bounded-type} rotation, which can be modelled as a  \emph{two} Rohlin towers picture for IET (or,  geometrically, two rectangles with suitable glueings), see \S~\ref{sec:degeneration}.

  These times are constructed controlling the degeneration of a surface suspending the IET along the Teichm{\"u}ller geodesic flow (using ergodicity of  
  Rauzy-Veech induction) and using \emph{rectangle representations} (a slight generalization of the zippered rectangles construction, see \S~\ref{sec:rectangles}) to then infer the existence of suitable Rohlin towers for the IET (see \S~\ref{sec:towers} for details). 

The geometric picture given by these two towers degeneration allows at the same time to produce rigidity and \emph{match} pieces of orbits which travel together (a mechanism that we call \emph{matching} of orbits and of Birkhoff sums, which is explained in Section~\ref{sec:matchingBS}), but also to distinguish two rescalings through the distribution of their tails, since when considering successive \emph{multiples} of this type of rigidity time, since the \emph{closest visit} to the singularity (which is responsible for the tail of the distribution) \emph{moves away} by a definite amount. 

We stress that the strategy exploited to study rescalings  and their disjointness by Kanigowski and the first author in \cite{BK} and here are quite different (and the first yields spectral disjointness, while we can only prove Furstenberg disjointness). 
 The main difference comes from the fact that, unlike in \cite{BK},  here we do not have control on the distributions of Birkhoff sums of the roof function over the whole interval. Furthermore, it turns out that the measure of the tails (which we use to differentiate between the rescalings) on one hand, and the measure of the set where we do not have the control over the values of Birkhoff sums on the other, are closely tied and related to the bounded rigidity scale. Thus, while in \cite{BK} the authors could choose to work with arbitrarily large tails (since showing positivity of the measure of the differentiating set was sufficient), in our setting we need to make much more careful choices, since, to guarantee that we can use the criterion for disjointness (see Proposition \ref{prof:disjointnesssf}), the measure of tails has to be substantially larger than the measure of the non-controlled set.

\subsection{Structure of the paper}
In Section~\ref{sec:background} we gather background definitions and material. We first recall in \S~\ref{sec:localHam} some background on locally Hamiltonian flows and their reduction to special flows in the setting we are studying. 
We then recall in \S\ref{sec:disjointness} basic joinings and disjointness notions and then state the criterion that we will use to prove disjointness of rescalings in special flows (Proposition~\ref{prof:disjointnesssf}). Some basics  on   translation surfaces appear in 
 \S~\ref{sec:surfaces}. 

In Section~\ref{sec:BTrigid} we motivate and define the new type of rigidity times we are going to use, that we call BT rigid (short for \emph{bounded type} rigid) (see Definition~\ref{def:BTrigid}).  These times are used to \emph{match} Birkhoff sums, an idea which is first motivated and then explained in \S~\ref{sec:matchingBS}.
The classical 
 estimates on (Birkhoff sums of) functions with symmetric logarithmic singularities that are needed for the proof of exponential tails are collected in Section~\ref{sec:BS}. Assuming the existence of a suitable sequence of \emph{good} BT rigidity times (see Definition~\ref{def:goodMn}), the proof that exponential tails follow from these estimates, is also provided in this section (see in particular \S~\ref{sec:exptails_via_match}).  

In Section~\ref{sec:Eproof} we prove the existence of good BT rigidity times, in which, using results by the last author from \cite{Ul:abs} (see also \cite{FKU}) one also has a good control on (trimmed) Birkhoff sums of the derivatives. This section requires the use of Rauzy-Veech induction and its natural extension. The needed facts on Rauzy-Veech induction used are stated in \S~\ref{sec:RV}. 
Finally, in the final Section~\ref{spectral disjointness}, we conclude the proof of disjointness by estimating the measures of the tails of Birkhoff sums distributions for different rescalings at good BT rigid times.}

\section{Background material}\label{sec:background}
{This background section consists of three subsections, dealing respectively 
with locally Hamiltonian flows and their representation as special flows (see 
\S~\ref{sec:localHam}), joinings and disjointness notions and the disjointness 
criterion (see \S~\ref{sec:disjointness}) and  basics on translation surfaces 
in \S~\ref{sec:surfaces}.} 

\subsection{Locally Hamiltonian flows and their special flow representations}\label{sec:localHam} 
{In this section we first recall some basic notions on locally Hamiltonian 
flows (from \S\ref{sec:locHam} to \S\ref{sec:measureclass}) and the definition 
of  special flows with logarithmic singularities over IETs (see 
\S~\ref{sec:iets} to \S~\ref{sec:roofs}). Then we state the properties of the 
special flow representations of the flows were are interested (see 
\S~\ref{sec:reduction}). }

\subsubsection{Locally Hamiltonian flows}\label{sec:locHam}
Throughout the paper let $S$ denote a  smooth, compact, connected, orientable surface with a fixed smooth area form (namely a non-degenerate two form that can be represented in local coordinates on a chart $U\subset S$ as $V(x,y)\mathrm{d} x \wedge \mathrm{d}y$ positive real valued function  $V:U \to \mathbb{R}$). These flows are in one-to-one correspondence with closed smooth $1$-forms on $S$, the form $\eta:=\imath_X\omega=\omega( \eta, \,\cdot \,)$, where $\imath_X$ denotes the contraction operator, being associated to the flow $\varphi_\R$ obtained integrating the vector field $X$. The assumption that $\eta $ is closed guarantees that $\varphi_\R$ preserves the area form and vice versa. The flow $(\varphi_t)_{t\in\mathbb{R}}$  is  also known in the literature as the \emph{multi-valued Hamiltonian} flow (see \cite{No:the, Ar:top} or \emph{locally Hamiltonian flow} (in \cite{Ul:abs, Rav:mix, KKU, FKU} among other more recent papers)  associated to $\eta$. Indeed, 
on each coordinate chart $U\subset S$ (but not generally globally) one can find a \emph{local Hamiltonian} $H: U \to \R$ and coordinates $(x,y)$ so that $(\varphi_t)_{t\in\mathbb{R}}$ is given by the solution to the  equations $\dot{x}={\partial H}/{\partial y}$, $\dot{y}=-{\partial H}/{\partial x}$ for some smooth  real-valued Hamiltonian function $H$.  
When $g\geq 2$, the  (finite) set of fixed points of $(\varphi_t)_{t\in\mathbb{R}}$ is always non-empty. We say that fixed points are \emph{non-degenerate} if the $1$-form $\eta$  is \emph{Morse}, i.e.\ it is locally the differential of a Morse function. 

It is  well known that  any  minimal (or minimal component of) locally Hamiltonian flow can be represented as \emph{the special  flow} over an  \emph{interval exchange transformations} or, for short, IET. 	We recall this representation, which is the way we will concretely work with these flows in this paper, 
 in  \S~\ref{sec:reduction}.
 
 We will assume in this paper that the surface under consideration has genus $g\ge 2$, that $(\varphi_t)$ is \emph{minimal} and that it has \emph{non-degenerate} fixed points. This implies in particular that there are \emph{2g-2} fixed points, all of which are \emph{simple saddles} (i.e.\ four-pronged saddles, with two incoming and two outgoing separatrices).

\subsubsection{Katok fundamental class and full measure}\label{sec:measureclass}
Let us denote by $\mathcal{L}_g$ the set or \emph{locus} ($\mathcal{L}$ stands for {locus}) of minimal locally Hamiltonian flows on a surface of genus {$g\ge 2$}. {When all the singularities of the flow are simple saddles (or equivalently have index $-1$), one can check (e.g.~in view of the Poincar{\'e}-Bendixon theorem) that their number is exactly $ 2g-2$.} 
The ({measure-theoretical}) notion of \emph{typical} on $\mathcal{L}_g$ (which is the notion used in the statement of Theorem~\ref{thm:main})  is defined as follows and coincide with the notion of typical induced by the  \emph{Katok fundamental class}   (introduced by Katok in \cite{Ka:inv}, see also \cite{NZ:flo}). 

{Let $\Sigma:= \{p_1,\ldots,p_{2g-2}\}$ be the set of saddle points of $\varphi_\R$ and let $H_1(S, \Sigma, \mathbb{R})$ be the relative homology of $S$ with respect to $\Sigma:=\{p_1,\ldots,p_{2g-2}\}$. Notice that the dimension $d$  of $H_1(S, \Sigma, \mathbb{R})$ is $2=2g+(2g-2)-1=4g-3$.} 
Consider the period map $Per: \mathcal{L}_g\to \mathbb{R}^{4g-3}$  defined as follows.  Let    $\gamma_1, \dots, \gamma_{4g-3}$ be a basis of $H_1(S, \Sigma, \mathbb{R})$. 
 The image of  $\eta$ by the period map $Per $ is {$Per(\eta) = (\int_{\gamma_1} \eta, \dots, \int_{\gamma_{4g-3}} \eta) \in \mathbb{R}^{4g-3}$.}
We say that a property holds for a \emph{typical} flow in $\mathcal{L}_g$ if it \emph{fails} on a set of measure zero  with respect to the pull back of the Lebesgue measure class via $Per: \mathcal{L}_g\to \mathbb{R}^{4g-3}$, namely it fails on the preimage $Per^{-1}(Z)$ of a set $Z\subset \mathbb{R}^{4g-3}$ with $Leb(Z)=0$. In \S~\ref{sec:reduction}, we give  a reformulation of this notion of typical in terms of special flows representations (see Lemma~\ref{lemma:reduction}).

\medskip
\subsubsection{Interval exchange transformations}\label{sec:iets}
We now recall the definitions of interval exchange transformations, which appear as Poincar{\'e} sections of locally Hamiltonian flows (as well as translation flows and billiards in rational polygons). The locally Hamiltonian flow will then be recovered as a 
 special flow (\S~\ref{sec:sfdef}) over an IET  (see \S~\ref{sec:reduction}).
 
 \smallskip
Let $I:=[0,|I|)$.  An interval exchange transformation (IET) of $d$ intervals $T:I\to I$ with \emph{ combinatorics} $\pi$, where  $\pi: \{{1},\dots, d\}\to \{{1},\dots, d\}$ is a permutation of $d$ symbols)  and \emph{endpoints} (of the continuity intervals) 
$0=:\beta_0 < \beta_1 < \dots \beta_{d-1}< \beta_d:= |I|$ 
is a piecewise isometry which sends the interval $I_i:=[\beta_i,\beta_{i+1})$, for $0\leq i<d$,
by a translation, explicitly given by
\[
T(x) = x-\beta_i +\beta_{\pi(i)}, \qquad \text{if }\ x \in [\beta_i,\beta_{i+1}).
\]
We say that $\pi:\{1,\dots, d\} \to \{1,\dots, d\}$ is \emph{standard} if $\pi(1) = d$ and $\pi(d)=1$. It turns out that by suitably choosing Poincar{\'e} sections of translation flows, in every stratum the first return map to this section is an IET whose combinatorial data is given by the standard permutation (see \cite{Rau}, as well as Lemma~\ref{lemma:reduction}).

Let us denote by $Disc(T)$ the set of endpoints of continuity intervals of $T$ (that we call \emph{discontinuity} set), namely
\begin{equation}\label{def:Disc}
Disc (T):= \{\beta_0:=0, \beta_1, \dots, \beta_{d-1}, \beta_d:=|I|\}
\end{equation}
We say that $T$ satisfies \emph{Keane condition} if $T^m(\beta_i)\neq \beta_j$, for all $m\ge 1$ and for all $i,j\in\{0,\ldots,d-1\}$, except $j=0$ with $m=1$. Keane \cite{Keane} showed that an  IET with an irreducible permutation that satisfies the Keane condition is \emph{minimal}.

We say that a result holds \emph{for almost every IET with combinatorics} $\pi$ if it holds for almost every choice of the lengths $|I_i|=\beta_{i+1}-\beta_i$ of the exchanged intervals, with respect to the restriction of the Lebesgue measure on $\mathbb{R}^d$ to the simplex $\Delta_{d-1} = \{(\lambda_1,\dots  \lambda_d)$ such that $ \lambda_i\geq 0$ and $\sum_{i=0}^{d-1} \lambda_i=1\}$.

{
\smallskip
Given $x\in I$ and $r\in\mathbb{N}$ we denote by $\mathcal{O}_T(x,r)$  the \emph{orbit segment} of length $r$ starting from $x$, namely
\begin{equation}\label{eq:orbitsegment}
\mathcal{O}_T(x,r):=\{ x, T(x), T^2(x),\dots, T^{r-1}(x)\}.
\end{equation}
Finally, for any $\beta\in Disc(T)$ we  denote by $m(x,\beta,r)$ the \emph{{minimum distance} of $\mathcal{O}_T(x,r)$ to the singularity $\beta$} of $T$ given  by
\begin{equation}\label{eq:mindistdef}
m(x,\beta,r):= \min_{j=0,\ldots,r-1}|T^jx-\beta|. 
\end{equation}}

\subsubsection{Special flows}\label{sec:sfdef}
Let us now recall the definition of \emph{special flow}.  Let $T$ be an automorphism of a standard (Borel)  probability space $(X,\mathcal{B},\mu)$. Let $f:X\to\R_{>0}$ be an integrable function  so that $\inf_{x\in X} f(x)>0$. Let us denote by $S_n(f)(x)$ the  \emph{Birkhoff sum} defined by
\[S_n(f)(x)=\left\{
\begin{array}{rcl}
\sum_{0\leq i<n}f(T^ix)&\text{if}& n\geq 0\\
-\sum_{n\leq i<0}f(T^ix)&\text{if}& n< 0.
\end{array}
\right.\]
The \emph{special flow} $(T^f_t)_{t\in\R}$ built \emph{over} the automorphism $T$ and \emph{under} the \emph{roof function} $f$  acts
on
\[X^f:=\{(x,r)\in X\times\R: 0\leq r<f(x)\}\]
and is given by
\[T^f_t(x,r)=(T^nx,r+t-S_n(f)(x)),\]
where $n=n(t,x) \in\Z$ is a unique integer number with $S_n(f)(x)\leq r+t<S_{n+1}(f)(x)$.  Under the action of $(T^f_t)_{t>0}$,  a point $(x,y) \in X_f$ moves  with unit velocity along the vertical line up to the point $(x,f(x))$, then jumps instantly to the point $\left( T(x),0 \right)$, according to the base transformation and afterward it continues its motion along the vertical line until the next jump and so on.  The integer $n(t,x)$ (for $t>0$) is the number of discrete iterations of the map $T$ undergone by the orbit of $x$ up to time $t$.

The flow $(T^f_t)_{t\in\R}$ preserves the finite measure $\mu^f$ which is the restriction of $\mu\times\ Leb_\R$ to $X^f$. The $\sigma$-algebra $\mathcal B^f$ on $X^f$ is $\cB\otimes \cB(\R)$ restricted to $X^f$.
If $T$ is ergodic with respect to $\mu$, it is easy to see then  $(T^f_t)_{t\in\R}$ is also ergodic (with respect to  $\mu^f$), see e.g.~\cite{CFS:erg}.  


\subsubsection{Roofs with logarithmic singularities}\label{sec:roofs}
 We now define the class of roof functions we will work with, namely roofs with (symmetric) logarithmic singularities, which  arise as roof functions in the special flow representation of locally Hamiltonian flows with non-degenerate (\emph{isomorphic}) saddles.
%
\smallskip
Let $T$ be an IET with continuity intervals $I_0,\dots , I_{d-1}$ and endpoints $0:=\beta_0<\beta_1< \dots \beta_{d-1}<\beta_d:=|I|$ (see ~\S\ref{sec:iets}). 


\begin{defn}[logarithmic singularities]\label{def:SymLog}
We say that a function $f $ has \emph{pure logarithmic singularities} (at the endpoints of $T$) and write  $f \in \ \mathcal{P}Log \left(T\right)$ if it is of the form
\begin{equation}\label{eq-1}
f(x)=\sum_{0\leq i<d}\chi_{(\beta_i,\beta_{i+1})}(x) \big(-C_i^+\ \log(x-\beta_i)-C_{i+1}^-\ \log(\beta_{i+1}-x)\big),
\end{equation}
for some constants $C_i^{\pm}\geq 0$, not all simultaneously zero.  

\smallskip
\noindent We say that $f$ has \emph{pure symmetric logarithmic singularities} (at the endpoints of $T$) and write $f \in {\mathcal{P} SymLog \left(T\right)}$  if in addition  we have that $\sum_{i=0}^{d-1}C_i^+=\sum_{i=1}^dC_{i}^-$. 

\smallskip
\noindent
We say that $f$ has \emph{logarithmic singularities} (resp. \emph{symmetric logarithmic singularities}) and write $f \in \ Log^2{\left( \sqcup_{i=0}^{d-1} I_i\right)}$ (resp.\  $f \in  SymLog^2 \left({\sqcup_{i=0}^{d-1} I_i}\right)$) if and only if $f$ can be written as $f=f^p+g$ where $f^p \in \mathcal{P} Log{\left( T\right)}$ (resp.\  $f \in \mathcal{P}SymLog{\left( T\right)}$) has pure logarithmic (symmetric) singularities and $g:I\to\R$ is {such that the restrictions $g_i:=g|I_i: I_i\to \R$ to each continuity interval $I_i$, for $1\leq i\leq d$, extend to a twice differentiable function on the closure $\overline{I}_i$, i.e.~are restriction of $g_i\in \mathcal{C}^2(\overline{I_i})$ for each $1\leq i\leq d$}.
\end{defn}
{
\begin{remark}\label{rk:regularity_g}
In the literature, when one says that $f$ has (symmetric) logarithmic
singularities, and  writes $f\in SymLog(T)$, it is often only assumed that 
$f=f^p+g$ where $f^p$ as here is in $\mathcal{P} Log{\left(T\right)}$ but $g$ 
is only assumed to have \emph{bounded variation}. The apex $2$ in the notation 
$SymLog^2(T)$ denotes this additional assumption\footnote{{In \cite{FKU} the 
notation $SymLog^{2+}(T)$ is used to recall that $g$ \emph{extends} to a 
$\mathcal{C}^2$ function on each closure $I_i$, $1\leq i \leq d$; we choose to 
leave the $+$ not to make the notation too heavy.}}. 
We need the (twice) differentiability of $f$ on each $I_i$ essentially in the proofs in this paper, but one can show that the roofs which arise in representations of locally Hamiltonian flows on surfaces automatically do satisfy this additional regularity assumption (see the reduction explained in  \S~\ref{sec:reduction}, in particular Lemma~\ref{lemma:reduction} and its proof,  
or directly \cite{FKU} for the result used in the reduction).
\end{remark}}

 Notice that the signs in \eqref{eq-1} are chosen so that $f\geq 0$. We remark that we allow some of the $C_i^\pm$ to be zero; so $f$ could have a finite one-sided limit at some $\beta_i$ (but we assume that at least one of the singularities is indeed logarithmic).
Let us also introduce a notation for the constant obtained summing the constants $C_i^\pm$:
\begin{defn}\label{def:Cf}
Given $f\in \mathcal{P} Sym Log (T)$ of the form \eqref{eq-1}, we define $C_f$ to be the constant
$$C_f := \sum_{i=0}^{d-1}C_i^+ =  \sum_{i=1}^dC_i^-.$$ 
\end{defn}

\subsubsection{Reduction to symmetric special flows}\label{sec:reduction}
The following Lemma provides the reduction of a locally Hamiltonian flow in $\mathcal{L}_g$ to a  special flow 
over a symmetric IET, under a  roof with symmetric logarithmtic singularities. 

\smallskip

\begin{lemma}[Special flow reduction for flows in $\mathcal{L}_g$]\label{lemma:reduction}
If  $S$ has genus $g\ge 2$  and  $(\varphi_t)_{t \in \mathbb{R}}\in\mathcal{L}_g $ is a minimal locally Hamiltonian flow with  $2g-2$ non-degenerate simple saddles $p_1,\ldots,p_{2g-2}\in M$, then it is isomorphic to a special flow over an IET $T$ with a standard permutation $\pi$ with $d=4g-3$ and roof  $f\in SymLog^2{\left( T\right)}$. 
%
Furthermore, for a property to be \emph{typical} within   $\mathcal{L}_g$ (in the sense of \S\ref{sec:measureclass}), 
it is enough to show that for almost every IET  with  standard permutation $\pi$ (see  \S~\ref{sec:iets}) with $d=4g-3$, every special flow with symmetric logarithmic singularities  $f\in SymLog^2{\left(T \right)}$ has such property. 
\end{lemma}
\noindent The proof of Lemma~\ref{lemma:reduction} is essentially given in \cite{Ch-Fr-Ka-Ul}
 (see in particular Lemma 2.1 and Corollary 2.2 in Section 2.7 in 
 \cite{Ch-Fr-Ka-Ul}), {where the precise form of the constants $C^\pm_i$ is 
 computed, although it is only shown that  $f\in SymLog (T)$ (see 
 Remark~\ref{rk:regularity_g}), namely that  
$g$ in the decomposition $f=f_p+g$ has bounded variation. The proof that one can choose a suitable cross section of the flow for the special representation to ensure the additional regularity assumptions of $g$ (i.e.~to show that   $f\in SymLog^2 (T)$, see Definition~\ref{def:SymLog}) is given in \cite{FKU} (see the proof of Theorem 3 in \S~3.1 in \cite{FKU}).}

\subsection{Joinings and disjointness criteria}\label{sec:disjointness}
 In this section we  recall the definition of joinings and  
 the definition of \emph{disjointness in the sense of Furstenberg}, introduced by Hillel Furstenberg in \cite{Fur}. We then state a criterion for disjointness of two flows, which
will provide the key tool to prove spectral disjointness of rescalings in this paper.

\subsubsection{Disjointness}	
{
Let $\phi_\R=(\phi_t)_{t\in\R}$ and $\psi_\R=(\psi_t)_{t\in\R}$ be two  measure preserving flows acting respectively on standard { Borel} probability spaces $(X,\mathcal B,\mu)$ and $(Y,\cC,\nu)$. A {\em joining} $\rho$ between $\phi_\R$ and $\psi_\R$ is a $\phi_\R\times \psi_\R$ invariant probability measure such that $\rho(B\times Y)=\mu(B)$ and $\rho(X\times C)=\nu(C)$, for all $B\in \cB$ and $C\in \cC$. 
The set of joinings is denoted by $J(\phi_\R,\psi_\R)$. Two flows are {\em disjoint} in the sense of Furstenberg if  $J(\phi_\R,\psi_\R)=\{\mu\otimes \nu\}$. In particular, it implies that the flows are non-isomorphic. We denote by $J_2(\phi_\R):=J(\phi_\R,\phi_\R)$ the set of {\em self-joinings} of $\phi_\R$.
}
We now present a criterion of two flows being disjoint. Our final goal is to obtain a version of this criterion, useful for our purposes. However, we start from recalling the following classical definition.

\begin{defn}\label{def:expdec}
We say that a measure $P\in\mathcal P(\R)$ has \emph{exponential decay} if there exist constants $c,b\in\R_{>0}$ such that
	\[
	P\big((-\infty,-t)\cup(t,\infty) \big)<ce^{-b t}\ \text{ for every }\ t\in\R_{>0}.
	\]
\end{defn}	
Let us also recall that, for every $t\in\R$, the \emph{Koopman} \emph{operator} associated to the automorphism  $\phi_t$, which, abusing the notation,  we will denote also by $\phi_t$, is the operator
\[\phi_t:L^2(X,\mu)\to L^2(X,\mu)\quad \text{  given by \ \ }\phi_t(f)=f\circ \phi_t.\]
Moreover, for every $P$, the probability measure on $\R$, we consider also the \emph{integral operator} $P(\phi)$, given by the formula
\[
P(\phi):=\int_\R \phi_{t}\,dP(t).
\]

\subsubsection{A disjointness criterion}\label{sec:disj_crit} The following lemma gives a criterion for two flows to be disjoint in the sense of Furstenberg. 

	\begin{lemma}[Furstenberg disjointness criterion]\label{krytspek}
		Assume that $\phi_\R=(\phi_t)_{t\in\R}$ and $\psi_\R=(\psi_t)_{t\in\R}$ are weakly mixing\footnote{Note that assuming weak mixing is necessary: indeed, if a flow $(\phi_t)_{t\in\R}$ has a non-zero eigenvalue $a\in\R$ then $KLa$ is a common eigenvalue for $(\phi_{Kt})_{t\in\R}$ and $(\phi_{Lt})_{t\in\R}$ for every $K,L\in\N$. Thus the natural powers of a non-trivial ergodic flow which is not weakly mixing are never pairwise Furstenberg disjoint.} flows on probability spaces $(X,\mathcal B,\mu)$ and $(Y,\mathcal C,\nu)$ respectively. Suppose that there exists a sequence $(t_n)_{n\in\N}$ increasing to infinity, a number $\gamma>0$ and linear operators $U$ on $L^2(X,\mathcal B,\mu)$ and $V$ on $L^2(Y,\mathcal C,\nu)$ such that  
		\begin{equation}\label{zbie}
			\phi_{t_n}\to (1-\gamma)P(\phi)+\gamma\,  U\quad \text{and}\quad \psi_{t_n}\to (1-\gamma)Q(\psi)+\gamma\, V\quad \text{weakly},
			\end{equation}
		where $P,Q$ are probability measures with exponential decay on $\R$. If there exists a Borel set $A\in\R$ such that $P(A)- Q(A)>\frac{\gamma}{1-\gamma}$ then $\phi_\R$ and $\psi_\R$ are disjoint in the sense of Furstenberg.
		\end{lemma}
		\noindent		{This Lemma can be seen as a special 
		case of Proposition 3.2 in \cite{BFr1}.}\footnote{{In \cite{BFr1}, the language of \emph{Markov operators} 
		(a generalization of Koopman operators), as well as the language of 
		$3$-joinings is used (see the proof of  Proposition 3.2). However, to 
		prove Lemma~\ref{krytspek}, only $2$-joinings are needed, and the proof 
		(which we sketch here) can be obtained by considering a projection on 
		two coordinates in the proof of Proposition 3.2 in \cite{BFr1}.}} 

{We finish this section by explaining a simple argument, that will allow us to focus only on positive rescalings of considered flows.
	\begin{lemma}\label{lem:positive_enough}
		Assume that $\phi_\R=(\phi_t)_{t\in\R}$ is a flow for which there exists a sequence $(t_n)_{n\in\N}$ increasing to infinity, a number $\gamma>0$ and a linear operator $U$ on $L^2(X,\mathcal B,\mu)$ such that  
		\begin{equation*}
			\phi_{t_n}\to (1-\gamma)\int \phi_t\, dQ(t)+\gamma\,  U\quad \text{weakly},
		\end{equation*}
		where $Q$ is a probability measure on $\R$. Then
		\begin{equation*}
			\phi_{-t_n}\to (1-\gamma)\int \phi_{-t}\, dQ(t)+\gamma\,  U^*\quad \text{weakly},
		\end{equation*}
		where $U^*$ is the dual operator of $U$.
	\end{lemma}
	\begin{proof}
		Let $f,g\in L^2(X,\mathcal B,\mu)$ be arbitrary. By the assumption of the lemma, we have
		\[
		\begin{split}
		\lim_{n\to\infty}\langle f\circ \phi_{-t_n},g \rangle &=\lim_{n\to\infty}\langle f,g\circ \phi_{t_n} \rangle=\lim_{n\to\infty}\overline{\langle g\circ \phi_{t_n},f \rangle}
		=
		\overline{\Big\langle(1-\gamma)\int g\circ \phi_t\, dQ(t)+\gamma\,  U(g),f\Big\rangle}
		\\ & =
		(1-\gamma)\int\overline{\langle g\circ \phi_t,f\rangle\, dQ(t)}+\gamma\,  \overline{\langle U(g),f\rangle}
		=(1-\gamma)\int\overline{\langle g,f\circ \phi_{-t}\rangle\, dQ(t)}+\gamma\,  \overline{\langle g,U^*(f)\rangle}\\
		&=(1-\gamma)\int{\langle f\circ \phi_{-t},g\rangle\, dQ(t)}+\gamma\, {\langle U^*(f),g\rangle}
		=\Big\langle(1-\gamma)\int f\circ \phi_{-t}\, dQ(t)+\gamma\,  U^*(f),g\Big\rangle.
	\end{split}
		\]
		Since $f$ and $g$ were arbitrary, this means that $\varphi_{t_n}$ converges weakly to the desired limit and hence finishes the proof.
	\end{proof}
	\begin{cor}\label{cor: just positive}
		If $\phi_\R=(\phi_t)_{t\in\R}$ and $\psi_\R=(\psi_t)_{t\in\R}$ are two flows satisfying the assumptions of Lemma \ref{krytspek}, then $(\phi_t)_{t\in\R}$ is disjoint to both $(\psi_t)_{t\in\R}$ and $(\psi_{-t})_{t\in\R}$.
		\end{cor}
}

\subsubsection{Disjointness criterion for special flows}\label{sec:disj_sf}
In the setting of special flows, the  following criterion for Furstenberg disjointness of \emph{rescalings} was proven in \cite{Fr-Le2} and in more general version in \cite{BFr}.
Let us first recall the definition of \emph{partial rigidity}.

\begin{defn}
We say that $(q_n)_{n\in\N}$ is \emph{a partial rigidity sequence} for a dynamical system $(X,\mathcal B,\mu, T)$ if there exists a sequence  $(W_n)_{n\in\N}$ of  measurable subsets $W_n\in \mathcal B$, called 
 \emph{ partial rigidity sets}, and $0<\gamma<1$, such that $\lim_{n\to \infty} \mu(W_n)= 1-\gamma$
 and for every measurable $A\subset \mathcal B$ we have 
	\[
	\lim_{n\to\infty}\mu((A\triangle T^{q_n}A)\cap W_n)=0.
	\]
	\end{defn}
\begin{remark}\label{metr}
	Suppose that $(X,\mathcal B,\mu)$ is endowed with a metric $d$
	generating the $\sigma$-algebra~$\mathcal B$. If $$
\lim_{n\to \infty}\sup_{x\in
		W_n}d(T^{q_n}x,x)= 0,$$ then $(q_n)_{n\in\N}$ is a rigidity
	sequence for $T$ with rigidity sets $(W_n)_{n\in\N}$. A proof of this observation can be found for example in \cite{BFr}.
\end{remark}	
\noindent The following proposition will be used as a tool to proving Furstenberg disjointness of different rescalings of a special flow $(T_t^f)_{t\in\R}$ and to reduce it to the study of the distribution of Birkhoff sums $S_{q_n}f$ of the roof function $f$ along a sequence of rigidity times. More precisely, we consider the sequence $(S_{q_n}f-a_n)_{n\in \N}$ where $(a_n)_{n\in \N}$ will be chosen appropriately (to \emph{center} Birkhoff sums) and the associated sequence $(P_n)_{n\in \N}$ of measures on $\R$ obtained as pull-back of the invariant measure $\mu$ on the base via $(S_{q_n}f-a_n)_{n\in \N}$, namely the measures $P_n:=(S_{q_n}(f)-a_n)_\ast (\mu)$, $n\in \N$. 
Explicitly,
\begin{equation}\label{def:pullback}
P_n \left( (-\infty, t)\right):= \mu\left(\{x \in X: \; S_{q_n}f(x)-a_n< t \}\right), \qquad \text{ for every }\ t\in\R_{>0}.
\end{equation}
For every measure $\nu$ on $(X,\mathcal B)$ we denote by $\nu_B$, the conditional measure with respect to $B$, i.e. for every $A\in \mathcal B$ we have 
$$\nu_B:=\frac{\nu(A\cap B)}{\nu(B)}.$$

The following Proposition is the main criterion used to prove disjointness of rescalings of flows. It follows directly from Theorem 6 in \cite{Fr-Le2} and Lemma \ref{krytspek}. For convenience of the reader, we include here its proof, which shows how tail estimates of Birkhoff sums distributions can be used to prove disjointness of special flows. 

\begin{prop}[Disjointness of powers of special flows]\label{prof:disjointnesssf} Consider  a weakly mixing special flow $T^f_\R:=(T_t^f)_{t\in \R}$ on probability standard Borel space $(X^f,\mathcal B^f,\mu^f)$. Let $K,L\in\N$, with $K< L$
Suppose that there exists an increasing sequence of partial rigidity times  $(q_n)_{n\in\N}$  with corresponding partial rigidity sets $(W_n)_{n\in\N}$ with constant $\gamma>0$. Moreover, suppose that there exists a sequence $(a_n)_{n\in\N}$ of real numbers
	such that the sequences of measures $P_n:=(S_{Kq_n}(f)-Ka_n)_\ast (\mu_{W_n})$ and $Q_n:=(S_{Lq_n}(f)-La_n)_\ast (\mu_{W_n})$  on $\R$ 
	have \emph{uniform} exponential tails, 
		 i.e.~there exists constants $c,b\in\R_{>0}$ such that for every $n\in\N$ we have
\begin{equation}\label{eq:exptails}
\max\{\mu_{W_n}\left(\{x \in W_n: \; |S_{Kq_n}f(x)-Ka_n|> t \}\right),\mu_{W_n}\left(\{x \in W_n: \; |S_{Lq_n}f(x)-La_n|> t \}\right)\}<  ce^{-b t},
	\end{equation}
	for every $t>0$.
Then, if  for some  $t>0$
\begin{equation}\label{differenttail}
	\begin{split}
&\liminf_{n\to +\infty}\mu_{W_n}\left(\{x \in W_n \; |S_{Kq_n}f(x)-Ka_n|\geq 
Kt\}\right)\\
&- \limsup_{n\to +\infty}\mu_{W_n}\left(\{x \in W_n: \; |S_{L q_n}f(x)-L 
a_n|\geq L t\}\right)>\frac{\gamma}{1-\gamma},
\end{split}
\end{equation}
 then the $K$ and $L$ rescalings 
 $(T_{Kt}^f)_{t\in\R}$ and  $(T_{Lt}^f)_{t\in\R}$ of $(T_t^f)_{t\in\R}$ are disjoint in the sense of Furstenberg.
\end{prop}
\begin{remark}
	In the assumptions of Theorem 6 in \cite{Fr-Le1}, there is a weaker 
	condition than exponential tails, namely that the condition 
	\eqref{eq:exptails} is replaced by
	\[
		\Big(\int_{W_n}|S_{q_n}f(x)-a_n|^2d\, \mu(x)\Big)_{n\in\N}\quad\text{is a bounded sequence.}
	\]
	However, it is an easy exercise that the condition on uniform exponential tails implies the above.
\end{remark}


\begin{proof}[Proof of Proposition~\ref{prof:disjointnesssf}]
	Given a measure $P\in \mathcal{P}(\R) $ and $w\in \R$, let us denote by $Res_w(P)\in \mathcal{P}(\R)$ the $w$-\emph{rescaling}, namely the measure given by $[Res_w(P)](A)=P(w\cdot A)$. Note that if $P$ is an absolutely continuous measure with density $x\mapsto f(x)$ then $Res_w(P)$ is also an absolutely continuous measure with density $x\mapsto wf(w\cdot x)$. 
	
	By Theorem 6 in \cite{Fr-Le1}, the assumptions of Proposition \ref{prof:disjointnesssf} imply that in the weak operator topology, we have the following convergences:
	\begin{equation}\label{eq: operatorconveregence1}
		T_{Ka_n}^f\to (1-\gamma)[Res_K]_{\ast}P\left((T^f_{Kt})_{t\in\R}\right)+\gamma \tilde P\quad\text{and}\quad T_{La_n}^f\to (1-\gamma)[Res_L]_{\ast}Q\left((T^f_{Lt})_{t\in\R}\right)+\gamma\tilde Q,
		\end{equation}
		where $P:=\lim_{n\to\infty} P_n$ and $Q:=\lim_{n\to\infty} Q_n$ and $\tilde P,\tilde Q$ are linear operators on $L^2(X^f,\mathcal B^f,\mu^f)$. Note that $P$ and $Q$ exist up to taking a subsequence due to \eqref{eq:exptails} and Prokhorov Theorem.
		
		We now apply Lemma \ref{krytspek} to measures $[Res_K]_{\ast}P$ and $[Res_L]_{\ast}Q$, considering the \emph{tail} set $$A:=(-\infty,-t]\cup[t,\infty).$$
		 Then it remains to notice, that the assumption \eqref{differenttail} gives exactly that 
		  $$[Res_K]_{\ast}P(A)-[Res_L]_{\ast}Q(A)>\frac{\gamma}{1-\gamma},$$
		  so  that the conclusion of Proposition \ref{prof:disjointnesssf} follows from the Furstenberg disjointness criterion (Lemma~\ref{krytspek}). 
\end{proof}

\subsection{Translation surfaces and suspensions}\label{sec:surfaces}
 A translation surface $M$ is a closed surface with a Euclidean metric outside a set $\Sigma$ consisting of finitely many points  where the metric has conical singularities. On $M\backslash \Sigma$,  there exists a natural \emph{translation structure}, namely an atlas of charts on $M\setminus\Sigma$ such that each transition map is a translation. {At each conical singularity, the cone-angle is of the form $2\pi k$ for some $k\in\mathbb{N}$. We denote by $\mathcal{H}(k_1, \cdots, k_d)$ the \emph{stratum} consisting of all translation surfaces with $d$ singularities with conical angles $2\pi (k_1+1), \dots , 2\pi (k_1+d)$.} 
 
Translation surfaces are obtained by gluing polygons in the plane,  by identifying pairs of parallel, isometric sides, by translations.  We will recall here only few basic properties and the construction of a \emph{suspension} of an IET, namely a translation surface that has the given IET as a Poincar{\'e} section, which will be useful in the proof.  As a reference on background material on translation surfaces, we refer the 
 reader to one of the surveys \cite{Vi,Yo}. 

\subsubsection{Suspension polygons.}
Fix a \emph{standard} permutation $\pi$ of $d$ elements and take $\lambda\in\Lambda_{d-1}$. To build a translation surface, we  need to also choose a vector $\tau\in\R^{d}$, known as \emph{suspension datum}, such that for every $1\le j< d$ we have
\begin{equation}\label{eq: deftau}
	\sum_{1\le i< j}\tau_i>0\quad\text{and}\quad
	\sum_{i\mid 1\le \pi(i)<j}\tau_i<0.
\end{equation}
We denote the set of $\tau\in\R^{d}$ satisfying \eqref{eq: deftau} by $\Theta_{\pi}$.
We then consider a polygon in $\mathbb{R}^2$, which we identify with $\C$, we made by connecting the vertices given by
$$
v^+_i:= \sum_{i=1}^j(\lambda_{i}+ i\tau_{i}),\ \text{for} \ 1\leq i \leq d, \quad \text{ and} \quad v_i^-:=\sum_{i=1}^j(\lambda_{\pi(i)}+ i\tau_{\pi(i)}) \ \text{for} \ 1\leq i\leq d, $$ 
and the vertex $v_0:=0$. Due to \eqref{eq: deftau}, this polygon is made of two broken lines, one \emph{upper}, whose all non-extreme vertices are above the real axis, made connecting $v_0, v_1^+, \cdots, v_d^+$ on another, \emph{lower}, obtained connecting $v_0, v_1^-, \cdots, v_d^-$  whose all non-extreme vertices are below the real axis. Notice that $v_d^+=v_d^-$ and (because the permutation is symmetric) the two broken lines do not self-intersect and hence bound a polygon (see Figure~\ref{fig:suspension})  with $2d-2$ vertices.  By construction, for each edge $E^+_i:=\overline{v_i^+ v^+_{i+1}}$ of the top line, there exist an edge $E^-_i:=\overline{v_j^- v^-_{j+1}}$ which is parallel and isometric. 
 By idenfitfying all of such pairs by the unique translation which maps $E^+_i$  to $E^-_i$, we obtain a \emph{translation surface} that we denote $M=(\pi,\lambda,\tau)$. We denote by $\Sigma=\Sigma(\pi,\lambda,\tau)$ the set of vertices of $(\pi,\lambda,\tau)$ which form the singularities of $M$. 
\begin{figure}
	\includegraphics[width=.5\textwidth]{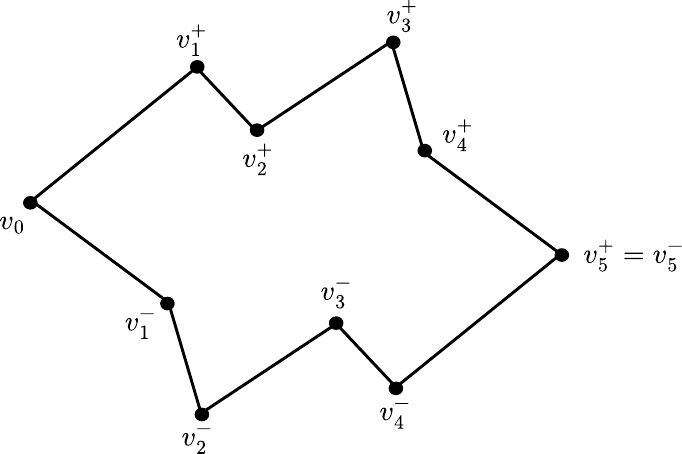}
	\caption{An example of a polygon representing a translation surface with $\pi=(54321)$. }\label{fig:suspension} 	
\end{figure}

Finally let us recall that any translation surface can be represented in more than one way as a polygon, for example \emph{cutting} the polygon along a straight line with endpoints in $\sigma$,  and \emph{glueing}  it back using the identifications of another pairs of sides. The (invertible) \emph{Rauzy-Veech algorithm}, described in Section \ref{sec:invertibleRV}, provides a way to produce a sequence of suspensions which all give different polygonal representations of the same translation surface.  

\subsubsection{The vertical flow.}\label{sec:vflow}
On any  translation  surface $M$ one can  define a global notion of direction induced by directions on $S^1\subset \R^2$ and from this one can define directional flows on $M$. We will consider one of such flows in particular, namely the \emph{vertical translation flow} $\varphi^v_{\R}=(\varphi_t^v)_{t\in\R}$ which moves points upwards with unit speed. This flow is defined on $M\backslash \Sigma$. We often refer to the points in $\Sigma$ as \emph{singularities} of $\varphi^v_{\R}$. 
 {Conical singularities of angle $4\pi$ give rise to simple saddle points of the linear flow. We write $\mathcal{H}(1^d)$ for the so-called \emph{principal} stratum, which consists of translation surfaces with $d$ conical angles $4\pi$ (and linear flows with only simple saddles). If the surface has genus $g$, $d=2g-2$.} 
 It is easy to see that the directional flows preserve the Lebesgue measure on $M$, which can be seen as a restriction of the Lebesgue measure on $ \mathbb{R}^2$. 

When $M=M(\pi, \lambda,\tau)$ is the suspension defined above, by considering the horizontal interval $I\subseteq M$ starting at $0$ and of length $|\lambda|$ as a Poincar\'e section for  $\varphi^v_{\R}$, one can see that the first return map to $I$ is exactly the  IET $T=(\pi,\lambda)$. 
We remark that in this setting the discontinuities of $T=(\pi,\lambda)$ are created by the singularities $\Sigma=\Sigma(\pi,\lambda,\tau)$ of $M=(\pi,\lambda,\tau)$: more precisely, the trajectories of  $\varphi^v_{\R}$ which start from a discontinuity of 
$T$ end in a singularity of $M$, and hence are  incoming \emph{separatrices} of a \emph{saddle} of   $\varphi^v_{\R}$.


\section{Bounded type rigid rectangles and towers}\label{sec:BTrigid}
{ We describe in this section the key geometric construction on which many of 
the analytic estimates rely upon. We start by considering translation surfaces 
obtained from glueing rectangles (through \emph{rectangles presentations}, 
which are a generalization of the notion of zippered rectangles, see 
\S~\ref{sec:rectangles}). We first describe the geometry of an irrational 
linear flow with a \emph{bounded-type} rotation number in terms of  
\emph{bounded-type rigidity} (see Def.~\ref{def:BTrigid}) presentations made by 
two rectangles of comparable areas, with certain identifications 
of their sides.
	We then describe how the geometry of a translation surface in higher genus can be forced to \emph{degenerate} to that of torus with bounded-type rigidity, by degenerating in a controlled way  rectangle presentations (see \S~\ref{sec:degeneration}). Rectangle presentations will be used to build \emph{Rokhlin towers} presentations for IETs, (defined in \S~\ref{sec:towers}).}

\subsection{Bounded type rigidity via rectangle presentations}
We will work with rectangle presentation of translation surfaces, which are a generalization of the zippered rectangles construction.
\subsubsection{Rectangle presentations}\label{sec:rectangles}
Let $S$ be a translation surface. When given by a rectangle presentation, $S$ is represented as a union of rectangles with glueings of the sides.
\smallskip

 A \emph{rectangle presentation} $\mathcal{R}(S)$, or simply a \emph{presentation} of $S$, is a collection   
 $$\mathcal{R}(S):= \{  R_1,  R_2, \cdots, R_d ;\,  \sim_{\mathcal{R}(S)} \} $$   
where, for each $1\leq i\leq d$, $R_i\subset \mathbb{R}^2$ is a (closed) rectangle with horizontal and vertical sides and $\sim_{\mathcal{R}(S)}$ is  equivalence relation identifying the boundaries of the rectangle by parallel translations  so that
$$S=  R_1\sqcup R_2\sqcup \dots R_d /\sim_{\mathcal{R}(S)}.$$
For the reader who is familiar with them,  \emph{zippered rectangles} (as defined by Veech in \cite{Ve:gau}, see also \cite{Yo}) are a  special type of rectangle presentations. We will work with rectangle presentations build similarly to zippered rectangles: the top sides of $R_i$ will be identified with points belonging to the union of the bottom sides of the rectangles, while each vertical side of a rectangle is union of two segments (one of which could be reduced to a point), so that each of the segments belonging to a left (respectively right) vertical side is identified by a (restriction of a planar) translation  to a (isometric)  segment contained in a right (respectively left) vertical side. 
An example of a rectangle presentation build in this way is shown in Figure \ref{fig: recpresent}.

\begin{figure}
	\includegraphics[width=.4\textwidth]{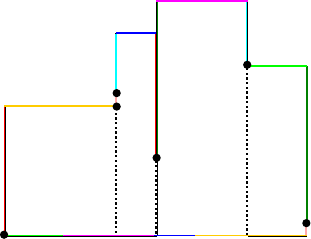}
	\caption{An example of a rectangle presentation. The dotted line are glued to the according to the adjacency as seen on the picture. The gluing of the remaining segments is glued with accordance to the color.}\label{fig: recpresent}
\end{figure}

{
\subsubsection{Bounded-type rigidity}\label{sec:degeneration}
We will be interested in building rectangle presentations of a translation surface $S$ which allow us to visualize $\epsilon$-(partial)\emph{rigidity times} for the vertical flow $\varphi_\R^v:= \varphi_\R^{v}$ on $S$, 
 i.e.~times $q$ such that, for a set $S_\epsilon\subset S$ of measure  $1-\epsilon$ of initial points  $x\in S_\epsilon$ the trajectory of  $\varphi_\R$ starting at $x$ returns after time $q$ $1/q$-close to $x$, i.e.~
 $$d(x,\varphi_q(x))<1/q, \quad \text{for\ all}\ x\in S_\epsilon$$ 
(where $d$ is the Euclidean metric induced on $S$, see \S~\ref{sec:surfaces}). 

 The classical example of a rectangle presentation which gives (partial) \emph{rigidity} times consists of one large rectangle, say $R_1$, which covers area $1-\epsilon$ and other $d-1$ rectangles whose total area is at most $\epsilon$, see Figure \ref{fig: onecylinder}. In this case,  the height $q_1$ of the rectangle $R_1$ is a rigidity time, as we can see  setting  $S_\epsilon:=R_1\cap \varphi_{-q_1}(R_1)$ (so that if $x\in S_\epsilon$, $x\in R_1$ it travels to the top of $R_1$ under the vertical flow and \emph{returns} and stays in $R_1$ up to time $q_1$).  Notice that in this case the displacement of $I_1$ (base interval of $R_1$) under the IET which identifies the top and bottom of the rectangles is at most $\epsilon$, so that for any $x \in S_\epsilon$, the distance $d(x, \varphi_{q_1}(x))$ is bounded above $\epsilon/{q_1}$.
 
 \begin{figure}
 	\includegraphics[width=.45\textwidth]{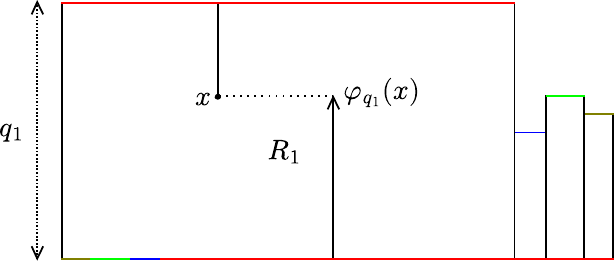}
 	\caption{One dominating cylinder $R_1$ of height $q_1$, yields rigid orbits of length $q_1$.}\label{fig: onecylinder}
 \end{figure}
 
 \subsubsection{Bounded-type rigidity times}\label{sec:BTrigidity}
We want to work with (a sequence of increasing) {partial} rigidity times $q$ such that $d(\varphi_q(x), x) \geq c/q$ for some $0<c<1$, i.e.~the \emph{relative displacement} $q \, d(x,\varphi_q x)$ is bounded below.  
The name \emph{bounded type rigidity} comes from rotations of \emph{bounded-type}, which display this type of controlled rigidity.  
In particular, if $S$ is a flat torus (i.e.~$S\in \mathcal{H}(0)$ is a genus one translation surface with one marked point), such that the Poincar{\'e} map of the vertical flow on a horizontal interval $I\subset S$ is a rotation $R_\alpha: I\to I$ of \emph{bounded-type}, i.e.~such that there exists $0<c<1$ so that 
$$
\left| {q\alpha}- p\right| \geq\frac{c}{q}, \qquad \text{for\ infinitely\ many}\ p \in \mathbb{Z}, \ q \in \mathbb{N}\backslash \{ 0\},
$$
  then we have times $q$ such that for all points $x$ on the torus $q d(x,\varphi_q x)\geq c$. 
We will also need to \emph{specify} the relative displacement range, i.e.~guarantee that, given an interval $C=(c_0,c_1)$, we have  $c_0\leq  q \, d(x,\varphi_q x)\leq c_1$. Let us hence introduce the following definition:

\begin{defn}[BT rigidity times]\label{def:BTrigid}
Given an interval $C=[c_0,c_1]$ where 
 $0<c_0<c_1\le 1$  and $\epsilon>0$, we say that $q$ is a $(C,\epsilon)$-BT rigidity time, or simply a $(C,\epsilon)$-rigidity time, if $q$ is an $\epsilon$-rigidity time and
$$
\frac{c_0}{q}\leq  d(x, \varphi_q x) \leq \frac{c_1}{q}, \quad  \text{for\ any}\ x\in S_\epsilon .
$$
\end{defn}
\noindent A sequence $(q_n)_n$ is 
 a sequence of $C$-\emph{BT rigidity times} (or simply BT rigidity times) if  for any $\epsilon>0$ there exists
 $n\in \N$ such that $q_n$ is a $(C,\epsilon)$-BT rigidity time.  
For example,  if $ \varphi_\R$ is as above the vertical flow on a flat torus with a  Poincar{\'e} map which is a rotation $R_\alpha$ with $\alpha$ of bounded type, 
the sequence of the denominators $(q_n)_n$ of the convergents of $\alpha$ form a sequence of $C$-BT rigidity times for some $C=[c,1]$ with $c>0$.

\subsubsection{BT rigidity using rectangles}\label{sec:BTrectangles}
Notice that sequences of $\epsilon_n$-rigid rectangle presentations with one rectangle $R_n^1$ of height $q_n$ and area $1-\epsilon_n$ increasing to $1$ \emph{cannot} be used to produce the type of BT rigidity in Definition~\ref{def:BTrigid}, since as remarked above $\Delta_n:=q_n d(x, \varphi_{q_n} x))<\epsilon_n\to 0$  as $n$ grow.}  
On the other hand, BT-rigity times can be built using \emph{two} towers presentations. 

\smallskip
\noindent {\bf Example} (BT rigidity with $d=2$). If $S$ is a  flat torus, we can represent it via a rectangle presentation $\mathcal{R}(S):= \{  R_1,  R_2 ;\,  \sim_{\mathcal{R}(S)} \} $ with $d=2$ rectangles, where the right vertical side of $R_2$ is glued to the top part of the left vertical side of $R_1$ (see Figure~\ref{fig:BTrigidityd2}).  
If the base of $R_1$ is larger than $c$, then the height $q_1$ of $R_1$
is a $c$-BT rigidity time: to see this, given $x\in S$, consider separately the two
cases when $x\in R_1$ or $x\in R_2$, as illustrated in Figure~\ref{fig:BTrigidityd2}). Notice that when $x\in R_2$, the glueings of the vertical sides are essential to get rigidity. Notice also that the  relative displacement $\Delta:= q\, d(x, \varphi_q(x))$ is exactly the $q \lambda_1$, where $\lambda_1$ is the width of the base of the first rectangle $R_1$, so the two rectangles should have comparable width to have a BT rigidity time.

\begin{figure}
	\includegraphics[width=.4\textwidth]{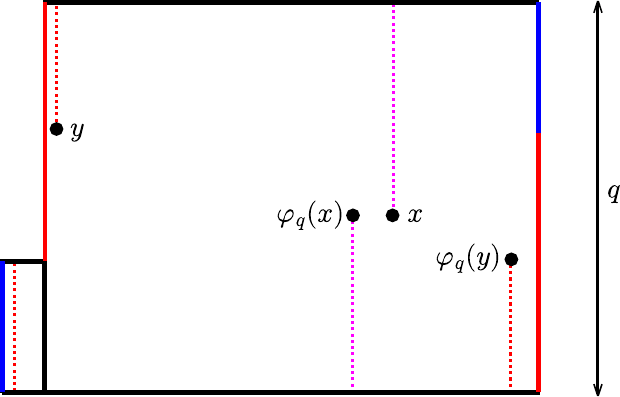}
	\caption{A $d=2$ rectangle presentation of a torus. Colors indicate vertical gluings. The height $q$ of the larger rectangle is a rigidity time. Picture on the right illustrates rigidity for $x$ and $y$ on the circle, exhibiting two different rigidity phenomena. }\label{fig:BTrigidityd2} 	
	\end{figure}




{
\subsubsection{BT rigid presentations in higher genus}
In higher genus/for higher number $d>2$ of exchanged intervals, we will work with BT rigidity times which are given by a rectangle presentation with three  \emph{large} rectangles (two with the same height) 
 and additional ones which have \emph{small} area. One can obtain this type of BT rigid rectangle presentations with $d>2$ rectangles by \emph{cutting}  in two  one of the two rectangles in $d=2$ presentation glued as in the above example, say $R_1$, and inserting additional rectangles in the middle (see Figure~\ref{fig:BTrigidrectangles}).  
 If the area of these additional rectangles is less than $0<\epsilon<1$, one can get in this way $\epsilon$-rigidity times which are also BT rigid.  
This is encoded in the following definition. Some comments on the meaning of the assumptions are given just below the definition.

\begin{figure}	\includegraphics[scale=1.6]{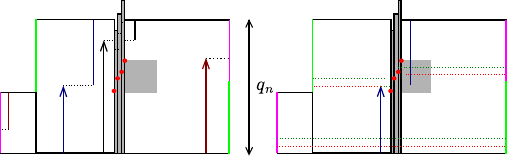}
	\caption{{\small A BT rigid rectangle presentation with $d=5$, illustrating the properties of Definition~\ref{def:BTrectangles}. The green and purple sides are identified. The \emph{bad zone} where rigidity fails is indicated in grey. On the left-hand side picture, we illustrate three rigid orbits with the short horizontal segments connecting the beginning and the end of the orbit. On the right-hand side picture, we illustrate absence of rigidity in the "bad zone", by drawing two dotted horizontal segments starting from the endpoint of the orbit. One can see that unlike in the first picture, these segments do not realign with the initial point of the orbit.}}\label{fig:BTrigidrectangles}
\end{figure}

\begin{defn}[BT rigid rectangles presentations for $d\geq 3$.]\label{def:BTrectangles}
 A rectangle presentation $\mathcal{R}(S)$ with $d\geq 3$ rectangles $R_1, 
 \cdots, R_d$ with heights $q_1,\cdots, q_d$ is a $(C, \epsilon, \rho)$-BT rigid 
 presentation, where $C=[c_0,c_1]\subset \mathbb{R}_{+}$, $\epsilon>0$ and $ \rho>0$, if 
\begin{itemize}
\item[(i)] the rectangles $R_3, \cdots, R_{d-1}$ have total area 
$\mathrm{Area}(R_3\cup \cdots \cup R_{d-1})<\epsilon$;
\item[(ii)] the two rectangles $R_2$ and $R_{d}$ have the same height $q:=q_2=q_{d}$;
\item[(iii)] the rescaled displacement $\Delta:= q \lambda_1 $ where $\lambda_1$ is the width of the rectangle $R_1$ belongs to $C$, i.e.~$c_0\leq \Delta \leq c_1$.
\item[(iv)] the left vertical side of $R_1$ is glued to the top part of the right vertical side of $R_d$ (in particular $q_d>q_1$).  
\item[(v)] the points of discontinuities of the glueings of the vertical sides 
between the rectangles $R_2,\cdots, R_{d-1}$ have (vertical) distance less than 
$ \rho \, q $.
	\item[(vi)] for every $z\in R_3, \cdots, R_{d-1}$, we have 
	\[
	\max\left\{\operatorname{dist}(z,R_2),\operatorname{dist}(z,R_d)\right\}\le \frac{\epsilon}{q},
	\]
{where $dist(p, q)$ is the Euclidean distance between two points $p,q$ on the translation surface $S$ and 
$dist (z, R)$ denotes the distance of the point $x$ from the rectangle $R$, i.e.~$\min \{ dist (z,p)$, $p\in R\}$.}
\end{itemize}
\end{defn}
In this rectangle presentation, in view of $(ii)$, the rectangles $R_2$ and $R_{d}$ (of equal heights) should be thought as obtained from the same rectangle, by cutting it vertically\footnote{The reason why we do not simply define a BT rigid rectangle presentation with two large rectangles and the other small ones is technical: the degenerations that we build for a typical translation surface to this type of BT rigid presentation lead us to this type of picture (obtained by cutting and pasting a related zippered rectangles picture obtained by Rauzy-Veech induction, see Section~\ref{sec:Eproof} for details).} to add the additional \emph{small} rectangles $R_3, \cdots, R_{d-1} $. Thus, by $(i)$, the rectangles $R_1, R_{2}$ and $R_d$ occupy a $1-\epsilon$ proportion of space; we will call them \emph{large}; the other rectangles will be called \emph{small}. 

{Given a  BT rigid presentation, one can show that $q:=q_2=q_d$,  the common height of $R_2$ and $R_{d}$, is a BT rigid time (see 
\S~\ref{sec:degenerationRV}, 
in particular the proof of Corollary~\ref{cor:E_BTrigidity}).} 
Notice that, as in the $d=2$ example, the \emph{vertical glueings}  between the vertical sides of the three \emph{large} rectangles $R_1$, $R_2$ and $R_{d}$ given by $(iv)$ play a crucial role, since it allows for rigidity.  Condition $(v)$ will be referred to as \emph{cramping of discontinuities} and guarantees   that all discontinuities of small rectangles have vertical distance  from each other which takes a proportion less than $\rho $ of the height $q$ of $R_2$ and $R_d$. This is needed because rigidity may not hold in  a rectangular region inside $R_{d}$ (i.e.~the region shaded in Figure~\ref{fig:BTrigidrectangles}, see \S~\ref{sec:Edegenerations_proof} for further details) and condition $(v)$ guarantees that this region has  area { of order $\rho$} (so that  the rigidity set which is obtained by removing both the union of small rectangles $R_3\cup \cdots \cup R_{d-1}$, their images by $\phi_{-q}$, which has  area less then $\epsilon$, and this rectangular region of area proportional to $\rho$). {Finally, condition $(vi)$ guarantees that the distance of a point to a singularity, is very well controlled by the distance of this point to the side of a zippered rectangle.}
 


\subsubsection{Existence of BT rigid degenerations}\label{sec:Edegenerations}
{
We now   show  that almost every translation surface in $\mathcal{H}(1^{2g-2})$ (namely with genus $g\geq 1$ and $2g-2$ conical singularities of angle $4\pi$, see \S~\ref{sec:surfaces})  admits a sequence of BT rigidity times. The same result could be proved for other strata, we only restrict to  $\mathcal{H}(1^{2g-2})$ since these are the translation surfaces underlying\footnote{{Locally Hamiltonian flows in $\mathcal{U}_{min}$, namely  locally Hamiltonian flow with only $2g-2$ simple saddles (see \S~\ref{sec:locHam})  are indeed time-reparametrizations of linear flows on translation surfaces in $\mathcal{H}(1^{2g-2})$. }} them are of this type.}
\begin{prop}[Existence of bounded-type rigid presentations]\label{prop:Edegenerations} 
For
any interval $C=[c_0,c_1]$ with $0<c_0<c_1<1$ and any choice of $0<\rho<1$ and $0<\epsilon<1$,  
 for almost every translation surface {$S\in \mathcal{H}(1^{2g-2})$}, 
 there exist a sequence of 
 $(C,\epsilon,\rho)$-BT rigid rectangle presentations 
 (as in 
 Definition~\ref{def:BTrectangles}) with $d=4g-2$ rectangles, {with heights  growing to infinity.} 
\end{prop} 
\noindent As a corollary, we also have the existence of BT rigidity times:{
\begin{cor}\label{cor:E_BTrigidity}
For  almost every ${S\in \mathcal{H}(1^{2g-2})}$, for any non-empty interval $C=[c_0,c_1]\subset{\mathcal{R}_+}$ and any $0<\epsilon<1$, there exists a sequence $(q_n)_n$ with $q_n\to \infty$ of $(C,\epsilon)-$BT rigidity times.
\end{cor}}
\noindent     The proof of both the Proposition and its Corollary will be presented in \S~\ref{sec:degenerationRV}, after having recalled in \S~\ref{sec:RV} basic notation from Rauzy-Veech induction.  Indeed, we  will prove the Proposition by defining  
 subsets of the parameter space (see Lemma~\ref{lemma_openset}) of zippered rectangles which yield BT-rigid rectangle presentations (in the sense of Definition~\ref{def:BTrectangles}) and exploiting ergodicity of Rauzy-Veech induction to visit them. 
 
 \begin{remark}\label{rk:degenerations} Geometrically, the surfaces given by rectangle presentations as in Definition~\ref{def:BTrectangles} can be thought of as obtained by glueing with a \emph{slit} construction\footnote{To glue, build a  translation surface which is the connected sum of two translation surfaces $S_1, S_2$, one can consider a \emph{slit} $\gamma$ on both $S_1, S_2$, i.e.~consider parallel, isometric (flat) segments $\gamma_i\subset S_i$ for $i=1,2$ (which one can consider as \emph{cuts}) and identifying the two sides of each with the opposite side of the other.} a higher genus flat surface $S_2$ with  $Area (S_2)<\epsilon$ to a torus (given by a BT rigid rectangle presentations with $d=2$ rectangles). The existence for $S$ of a sequence of BT rigidity times given by these rectangle presentations implies that the \emph{Teichmueller geodesic} $(g_{t} S)_{t\in \R}$ through $S$ admits a sequence $(t_n)_n$ of times such that $(g_{t_n} S)_{n\in \N}$ \emph{degenerates} to a genus one surface (with a $C$-BT rigid presentation). 
\end{remark}

\subsection{Rohlin towers with BT rigidity}\label{sec:towers}
From the existence of rectangle presentations with a certain geometry, one can deduce the existence of \emph{Rohlin towers} (see \S~\ref{sec:Rohlintowers} below) with similar properties (for the IET which describes the top/bottom glueings of the initial rectangle presentation). }

\subsubsection{Rohlin towers for IETs}\label{sec:Rohlintowers}
Let us first recall the definition of Rohlin tower (by intervals). 
\begin{defn}[Rohlin tower by intervals]
Let $T:I\to I$ be an IET with $|I|\leq 1$. Given an interval $J_n:=(a_n,b_n)\subset I$ and an integer $h_n\in \N$ we say that the union 
$\bigcup_{i=0}^{h_n-1}T^i J_n$ is a (Rohlin) \emph{tower by intervals} of \emph{base} $J_n$ of \emph{height} $h_n$ if and only if the images $T^i J_n, 0\leq i < h_n$ are pairwise disjoint intervals.
\end{defn}
Notice that since $T$ preserves the Lebesgue measure $\lambda$, all floors of $\mathcal{T}$ have the same  (Lebesgue) measure $|b-a|$ and the measure of $\mathcal{T}$ is then $\lambda(\mathcal{T})= h_n |b_n-a_n|$. 
{
\begin{defn}[Towers representation]
A \emph{tower representation}  (also called a \emph{skyscraper} representation) $\mathcal{T}=\mathcal{T}(T)$ of an IET $T: I\to I$ is a (finite) collection
$\mathcal{T}:= \{ \mathcal{T}_1, \dots, \mathcal{T}_k \}$ 
of Rohlin towers $\mathcal{T}_i$ for $T$, $0\leq i\leq k$, that \emph{partition} $I$, i.e.~such that $\bigcup_{i=1}^k\mathcal{T}_i =I$. 
\end{defn}
Given an IET $T: I\to I$ and a subinterval $J\subset I$, it is well known that the induced (or first return map) of $T$ to $J$ is also an IET $T_J$. A tower representation  of $T$ can then be obtained taking the  
Rohlin towers $\mathcal{T}_i$ each of which has as a base one of the continuity intervals $J_i\subset J$ of the induced map $T_J$,  and as height the (first) return time of  $J_i$ to $J$ under $T$. The action of $T$ can then be seen as \emph{moving upwards} by one every floor of a tower but the last one, which is identified to the base through $T_J$.

\subsubsection{BT rigid Rohlin towers presentations}\label{sec:BTtowers} 
Towers presentations can be thought of (and formally related to, see \S~\ref{sec:rectangles_vs_towers}) \emph{discretizations} of rectangle presentations. In this spirit, one can define BT rigid tower presentations in analogy to the definition of BT rigid rectangle presentations. 

\begin{figure}
	\includegraphics[width=.4\textwidth]{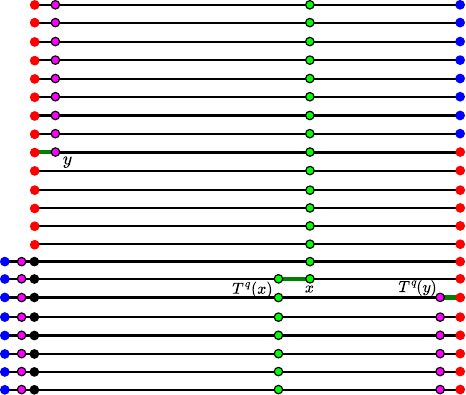}
	\caption{{\small The BT rigid tower presentation in genus 1 case is simply the two towers obtained by the classical Gauss map, when the partial quotient is bounded. Even though the measure of large tower is bounded away from 1, the heights of such towers still form a rigidity sequence for the rotation, as illustrated on the picture.} }\label{fig:BTrigiditygenus1} 	
\end{figure}

\begin{defn}[BT rigid towers presentations.]\label{def:BTtowers}
 A tower presentation $\mathcal{T} = \{ \mathcal{T}_1, \cdots, \mathcal{T}_d\}$ with $d\geq 3$ towers  with heights $q_1,\cdots, q_d$ is a  $(C, \epsilon, \rho)$-BT rigid presentation
for some $C=[c_0,c_1]\subset [0,1]$, $0<\epsilon,\rho<1$,  
  if 
\begin{itemize}
\item[(i)] the towers $\mathcal{T}_3, \cdots, \mathcal{T}_{d-1}$ have total measure $\lambda(\mathcal{T}_3\cup \cdots \cup \mathcal{T}_{d-1})<\epsilon$;
\item[(ii)] the two towers $\mathcal{T}_2$ and $\mathcal{T}_{d}$ have the same height $q:=q_2=q_{d}$;
\item[(iii)] the relative displacement $\Delta:= q\lambda_1$, where $\lambda_1 $
is the  width   of the tower $\mathcal{T}_1$, belongs to  $C$, i.e.~$c_0<\Delta<c_1$;
\item[(iv)] the tower $\mathcal{T}_d$  has height $q_d>q_1$ and the left endpoints of floors  of $\mathcal{T}_1$ are glued to the right endpoints of the top $q_1$ floors of $\mathcal{T}_d$, while top $q_2-q_1$ left endpoints of floors of $\mathcal{T}_2$ are glued to the right endpoints of the bottom $q_2-q_1=q_d-q_1$ floors of $\mathcal T_d$;  
\item[(v)] for $i=2,\dots, d-1$, the right endpoints of the bottom $k_i$
 floors of $\mathcal{T}_i$ are glued to the left endpoints of the corresponding floors of $\mathcal{T}_{i+1}$, while the right endpoints of the top $q_i-k_i$ floors are glued to 
to the left endpoints of the corresponding floors of $\mathcal{T}_{\pi^{-1}(\pi(i)+1)}$. Moreover, 
 $$
 \max_{3\leq i,j<d}| k_i-k_j|\leq {\rho\, q.}
 $$
 
 	\item[(vi)]  for every $x\in \mathcal T_2\cup\mathcal T_d$ and every $\beta\in Disc(T)$ if $x_{\beta}^+(x_{\beta}^-)$ is the closest visit to $\beta$ from the right (left) in time $q$, and $(a^+,b^+)$ (or $(a^-,b^-)$) is the level of $\mathcal T_2$ or $\mathcal T_{d}$ which contains $x_{\beta}^+(x_{\beta}^-)$, then
 	\[
 	\max\{\big||x_{\beta}^+-\beta|-|x_{\beta}^+-a^{+}|\big|, \big||x_{\beta}^--\beta|-|x_{\beta}^--b^{-}|\big| \}\le \frac{\epsilon}{q}.
 	\]
\end{itemize}
\end{defn}
\noindent Notice that the condition $(v)$ is a discretized version of condition $(v)$ for BT rigid rectangle presentations (see Definition~\ref{def:BTrectangles}). 

\subsubsection{Rohlin towers from rectangle presentations}\label{sec:rectangles_vs_towers}  
Let us finish this section by hinting at  the relation between rectangle representations of a translation surface $S$ and Rokhlin towers presentations of a IET $T$.  Given an IET $T$, let $S$ be a translation surface such that $T$ appears as  a Poincar{\'e} section of the vertical flow $\varphi_\R$ on an horizontal interval $I\subset S$ (we say in this case that $S$ is a \emph{suspension} of $T$). Then, given a sequence of rectangle presentations $(\mathcal{R}_n)_n$ for $S$ with increasing heights, 
one can produce a sequence of towers representations for $T$, which asymptotically have the same \emph{geometry} as the rectangle presentations (meaning the same ratios of width, heights, and the same gluing patterns between rectangles and towers), see  Figure~\ref{fig:rectangles_vs_towers}.

\begin{figure}
	\includegraphics[scale=1.8]{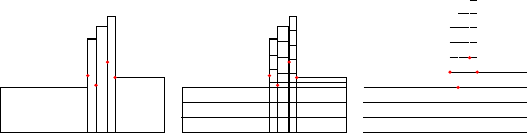}
	\caption{{\small Building Rohlin towers using  rectangle presentations.  The section which winds up around the rectangle presentation $\mathcal{R}(S)$ of the surface $S$,  makes it possible to  
		deduce properties of the Rokhlin towers by seeing them inside  
		rectangle presentations. Recall that the discontinuities of IET correspond to the first intersections of incoming separatrices, see \S~\ref{sec:vflow}.\label{fig:rectangles_vs_towers}}}
	\end{figure}
	
	To see this, notice that the interval $I\subset S$  can be visualized inside each $\mathcal{R}_n$ as a union of horizontal segments contained in the rectangles of $\mathcal{R}_n$ (see Figure~\ref{fig:rectangles_vs_towers});
 the horizontal segments are identified with each other according to the rectangle identifications of ${\mathcal{R}_n}$  (since they are part of the connected horizontal interval  $I\subset S$.
 From this correspondence, one can see that the dynamics of $T$ (which is the first return map of the vertical flow on $S $ to $I$ is exactly that one of a Rokhlin towers skyscraper (each tower being the union of floors contained in the same rectangle). We will use this correspondence to build Rokhlin towers for $T$ in Section~\ref{sec:Eproof} (see in particular \S~\ref{sec:Edegenerations}).}

\section{Birkhoff sums, matchings and exponential tails.}\label{sec:BS}
{In this section we introduce the tools that we need to verify the assumptions 
of the disjointness criterion for special flow (given by 
Proposition~\ref{prof:disjointnesssf}). We need to study the distribution of 
Birkhoff sums $S_{q} f$ of functions with  logarithmic singularities and to 
show that they have exponential tails along BT rigidity times.
In particular, we need estimates on the  derivatives 
 $S_q f'$ and $S_q f''$ (which are rather classical), as well as the notion (which is new to  this paper), of \emph{matchings of Birkhoff sums}.  We motivate and define the notion of \emph{matchings} (as well as \emph{balanced} matchings, see Definition~\ref{def:balancedm}) in \S~\ref{sec:matchingssec}. 
  In \S~\ref{sec:BSestimates} we recall some results providing estimates on Birkhoff sums of the roof and its derivatives; 
   Then, in \S~\ref{sec:exptails_via_match}, we exploit the existence of matchings with good properties and trimmed Birkhoff sums  linear bounds,  to prove the exponential tails estimates.}

\subsection{Matching orbits and Birkhoff sums}\label{sec:matchingssec}
In this section, we introduce 
 \emph{matchings}. 
We first heuristically explain  in \S~\ref{sec:matchingmotivation} the motivation  behind this  notion.  
 We then formally define the  notion of \emph{matchings of orbits} and then \emph{matchings of Birkhoff sums} in  \S~\ref{sec:matchingorbits} and \S~\ref{sec:matchingBS} respectively. 
 We then define \emph{balanced} matchings (see \S~\ref{sec:balancedmatchings}), which are matching sets with good geometry control of the distance from singularities.
   
\subsubsection{Motivation for matchings.}\label{sec:matchingmotivation}
Let $x,y $ be two points in $I$, fix $r\in \mathbb{N}$ and consider the  Birkhoff sums $S_r(x) $ and  $S_r(y) $.
If we want to estimate the difference $|S_rf(x) - S_r f(y) |$ (or the analogous differences between Birkhoff sums of the derivatives $f'$ and $f''$ of the roof function) using the mean-value theorem, it is necessary that $x,y$ both belong to the same continuity interval $(a,b)$ for $S_r f$, or, equivalently, that all the intervals $(a,b), T(a,b), \dots , T^{r-1}(a,b)$ do not contain singularities of $T$. 
This is the case, for example, if $(a,b)$ is the base of a Rohlin tower for $T$ with height $q$.
 More generally,  we can estimate $|S_rf(x) - S_r f(y) |$ by decomposing each Birkhoff sum $S_rf(x)$ and $S_r f(y)$ and, correspondingly (using the notation \eqref{eq:orbitsegment} introduced in \S~\ref{sec:iets} to denote orbit segments), decomposing 
 $\mathcal{O}_T(x,r)$ and 
$\mathcal{O}_T(y,r)$ into finitely many orbit segments, so that these orbit segments can be \emph{matched}, i.e. for each orbit segment of the form $\mathcal{O}_T(x_i,r_i)\subset \mathcal{O}_T(x,r)$ there is an orbit segment of the same length $\mathcal{O}_T(y_i,r_i)\subset \mathcal{O}_T(y,r)$ such that $| S_{r_i} f (x_i) -  S_{r_i} f (y_i)|$ can be estimated, as hinted above, using the mean-value theorem. This is the case when the interval $J:=I(x_i,y_i)$ is contained in an interval $(a_i,b_i)$ such that its iterates $T^j(a_i,b_i)$ for $0\leq m<r_i$ form a Rohlin tower. 


\subsubsection{Matching of orbits}\label{sec:matchingorbits}
Given $x_0,y_0\in I_0$ and $r\in \mathbb{N}$ consider the two orbit segments $\mathcal{O}_T(x_0,r)$ and $\mathcal{O}_T(y_0,r)$ (see \S~\ref{sec:iets}).  A \emph{matching} of these two orbit segments is formally defined as follows.
{
\begin{defn}\label{def:orbitmathcing}
Given a positive integer $k \in \mathbb{N}$, we  say that  two orbit segments  $\mathcal{O}_T(x,r)$ and $\mathcal{O}_T(y,r)$ \emph{admit a $k$-matching} if there exists 
 $k$-partition of $r$, 
  namely positive integers $r_1,\dots, r_k$ with $ r_1+r_2+\dots r_k=r$, 
such that we can decompose each orbit segments in $k$ disjoint orbit segments, namely 
\begin{align*}
\mathcal{O}_T(x,r)& = \mathcal{O}_T(x_1,r_1)\sqcup  \mathcal{O}_T(x_2,r_2)\sqcup \cdots \sqcup \mathcal{O}_T(x_k,r_k), 
\\ \mathcal{O}_T(y,r)& = \mathcal{O}_T(y_1,r_1) \sqcup  \mathcal{O}_T (y_2,r_2) \sqcup \cdots \sqcup  \mathcal{O}_T(y_k,r_k),
\end{align*}
so that,  for every $1\leq i\leq k$ there exists an interval $(a_i, b_i)$ which contains both points $x_i, y_{i}$ and 
is the base of  a Rohlin tower by intervals of height $r_i$.   
We then say that $\mathcal{O}_T(x_i,r_i)$ and $\mathcal{O}_T(y_i,r_i)$ admit a matching or simply can be matched if they admit a $k$-matching for some positive integer $k$.
\end{defn}}
\noindent Notice that the \emph{order} of the $\kappa$ orbit segments inside the bigger orbit segment of length $r$ may not be the natural one. For example, we can by convention set 
$$
x_1:=x,\quad  x_2:= T_{r_1}(x_1),  \quad \dots \quad x_k:= T^{r_{k}}(x_{\kappa-1})= T^{r_1+\dots + r_{\kappa}}(x_0)
$$
but then the corresponding points $y_i$ do not have the same order inside  $\mathcal{O}_T(y,r)$. We can describe their order by a permutation  $\nu:\{1,\dots, \kappa\}\to \{1,\dots, \kappa\}$, chosen  so that 
$$y_{\nu(1)}:= y,\quad  y_{\nu(2)}=T^{r_{\nu(1)}}y_{\nu(1)}, \quad \dots \quad , y_{\nu(\kappa)}:= T^{r_{\nu(\kappa-1)}}(y_{\nu(\kappa-1})).$$ 

\subsubsection{Matching of Birkhoff sums}\label{sec:matchingBS}
The following definition of matching of Birkhoff sums is based on matching of the underlying orbit segments (given in  Definition~\ref{def:matchingorbits} above).  
\begin{defn}\label{def:matchingorbits}
Given $x,y\in I_0$ and $r\in \mathbb{N}$, we say that the Birkhoff sums $S_r f(x) $ and  $S_r f(y) $ can be \emph{matched} if the orbits  $\mathcal{O}_T(x,r)$ and $\mathcal{O}_T(y,r)$ can be matched.
\end{defn}
\begin{remark}\label{rk:BSmatchings}
If  $S_r f(x) $ and  $S_r f(y) $ can be matched, we can decompose  each as
$$
S_r f(x) = \sum_{i=1}^k S_{r_i} f(x_i), \qquad S_r f(y) = \sum_{i=1}^k S_{r_i} f(y_i),$$
so that, for every $1\leq i\leq k$, $S_{r_i}f(x) $, as well as its first and second derivatives $S_{r_i}f' (x)$ and $S_{r_i}f''(x)$, are continuous for $x$ in an interval $(a_i,b_i) $ containing both $x_i$ and $y_i$, so that  the difference $| S_{r_i}f (x_i)- S_{r_i}f (y_i)|$ can be studied using mean value on $(a_i,b_i)$.
\end{remark}

{
\subsubsection{Matching sets}\label{sec:matchingset}
If we  now fix a reference Birkhoff sum $S_q f (y)$ and consider the set of points $x\in I_0$ such that $S_r f(x) $ and  $S_r f(y) $ can be matched, this gives a set which we call $(q,y)$-matching set:
\begin{defn}[matching sets]\label{def:matchingset}
Given $y \in I_0$ and $q\in \mathbb{N}$, we  say that a subset  $M\subset I_0$ is a \emph{$(y,q)$-matching set} if for every $x\in M$ the Birhoff sum $S_q f (x)$ 
 can be matched with the Birkhoff sum $S_q f (y)$. 
 {We say furthermore that $M$ is a  \emph{$(y,q)$-$\kappa$ matching set} if, 
 for every $x\in M$, $S_q f (x)$ admits a $k$-matching with  
 $S_q f (y)$ for a positive integer $k\leq \kappa$, i.e.~it can be matched with $S_q f (y)$ by decomposing it in at most $\kappa$ orbit segments.}%
\end{defn}
{
Let us give two examples of matching sets. Assume first that $T$ has a Rohlin towers presentation with $d=2$ towers, which correspond (in the sense of \S~\ref{sec:rectangles_vs_towers}) to the BT rigid presentation of $d=2$ rectangles described in the Example in \S~\ref{sec:BTrectangles}. Then, for any $y$ in the base of the towers  and $q=q_1+q_2$ equal to the sum of the two tower heights, the whole interval $I$ is a $(y,q)$-matching set, 
 since every orbit $\mathcal{O}(x,q)$ of length $q$ can be matched with $\mathcal{O}(y,q)$, as shown in Figure~\ref{fig: genus1matching}. 
 \begin{figure}
 	\includegraphics[width=.45\textwidth]{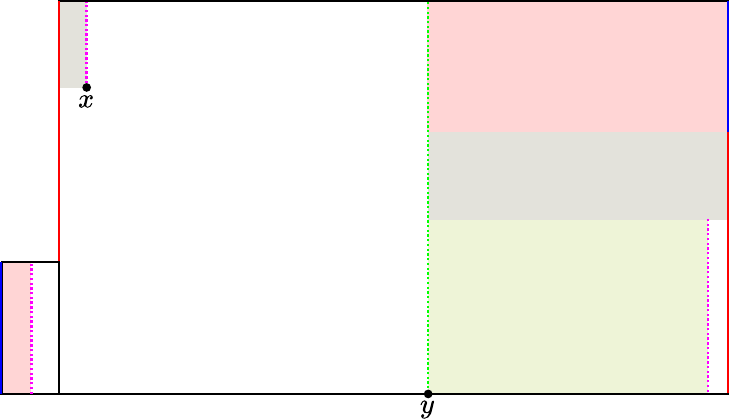}
 	\caption{An example of matching on genus 1 surface. The orbit of $x$ of length $q$, where $q$ is the height of the bigger tower, is matched to the orbit of $y$ of the same length. The colors indicate different ways with which the orbits are being matched. Notice that the connecting segments are not of the same length. }\label{fig: genus1matching}
 \end{figure}

On the other hand, if we consider Rohlin towers which correspond to a BT rigid presentation as in Definition~\ref{def:BTtowers} in \S~\ref{sec:BTtowers},  there is a \emph{bad zone} for rigidity (near the \emph{cramping} of discontinuities) which should be removed to obtain the matching set $M$ (see the comments in  \S~\ref{sec:BTrectangles} after Definition~\ref{def:BTrectangles}). In both these examples, $\kappa=3$, since all matchings can be achieved decomposing each orbit in $3$ segments. These type of matchings will be heavily exploited later, to bound Birkhoff sums using the mean value theorem (see \S~\ref{sec:matchingmotivation} for an introduction to the role of matchings in the proof).

Finally, to prove disjointness of an $L$-rescalings $(T_{Lt}^f)_{t\in\R}$ with 
the $K\leq L$ rescaling $(T_{Kt}^f)_{t\in\R}$ of the special flow 
$(T_t^f)_{t\in\R}$ for positive integers $K<L$, we will need to match not only 
orbits of length $q_n$ with a reference orbit $\mathcal{O}_T(y_n,q_n)$, but 
also orbits of length $L q_n$ with $L$ (and hence also $K$) copies of the 
orbit  $\mathcal{O}_T(y_n,q_n)$ (splitting the orbit of length $Lq_n$ into $L$ 
(resp.~$K$) segments of length $q_n$ each of which can be matched with 
$\mathcal{O}_T(y_n,q_n)$). For this, we define what we call \emph{$L$-fold 
matchings}:

\begin{defn}[$L$-fold matching sets]\label{def:Lmatchingset}
Given a positive integer $L\in \N$, $y \in I_0$ and $q\in \mathbb{N}$, we  say that a subset  $M\subset I_0$ is an \emph{$L$-fold} \emph{$(y,q)$-matching set} if for every $x\in M$ the $L$ Birhoff sum 
 $$
 S_q f (x), \ S_q f (T^q x), \  S_q f (T^{2q} x), \cdots, S_q f (T^{(L-1)q} x)
 $$
 can each be  matched with the reference Birkhoff sum $S_q f (y)$. 
 We say furthermore that $M$ is a $L$-fold \emph{$(y,q)$-$\kappa$ matching set} if each $ S_q f (T^{i q} x)$ for $0\leq i<L$  admits a $k$-matching with  
 $S_q f (y)$ for some $1\leq k\leq \kappa$. 
\end{defn}

\subsubsection{Balanced matchings}\label{sec:balancedmatchings} 
To estimate the values of the Birkhoff sum $S_q f$ and its derivatives at  points in a $(y,q)$-matching set $M$, we need additional assumptions on the matching set. In particular, since (as we recall in the next \S~\ref{sec:BSestimates}) controlling closests visits is essential to control Birkhoff sums of derivatives (see e.g.~Proposition~\ref{prop:accelerationset} in \S~\ref{sec:derivativesviaRV}), we need    
to guarantee that the orbit $\mathcal{O}_T(z,q)$ of $z\in M$ does not get too close to singularities. We will also need that the intervals used for matchings are not too small. 

 These additional requirements are axiomatised in the following Definition~\ref{def:balancedm} of what we call \emph{balanced} matching sets. 
We recall that $m(z,\beta,q)$ is the distance of $\mathcal{O}_T(z,r)$ from the set of discontinuities $Disc(T)$ (see \S~\ref{sec:iets}). 
\begin{defn}[Balanced matchings]\label{def:balancedm}
For $0<c<1$, we say that a  $(y,q)$-matching set $M$ is  $c$-\emph{balanced} if:
\begin{itemize}
\item[(B1)] the orbits $\mathcal{O}_T(y,q)$ and $\mathcal{O}_T(T^{-q}y,q)$  $c/q$-close, nor $1/cq$-far from any discontinuity in $Disc(T)$, namely 
$$\frac{c}{q}\leq  m(y,\beta,q)\leq \frac{1}{cq} ,\quad  
\frac{c}{q}\leq  m(T^{-q}y,\beta,q)\leq \frac{1}{cq} \qquad \text{ for\ any}\ \beta\in Disc(T).$$
\item[(B2)] for any $\beta \in Disc(T)$ and any $x\in M$, there exists at most one point   $z\in \mathcal{O}_T(x,q)$ such that  $d(z, \beta)\leq \frac{1}{cq}$;
\item[(B3)]  for all $x\in M$ the $k $ intervals $(a_i,b_i)$ given by Definition~\ref{def:orbitmathcing}  are such that  $b_i-a_i \leq \frac{1}{cq}$ for any $i=1,\dots, k$; 
\item[(B4)] for every $x\in M$ we have $\min_{j,j'\in\{0,\ldots,q-1\};\,j\neq j'}|T^jx-T^{j'}x|\ge\frac{c}{q}$. 

\end{itemize}
\end{defn}
\begin{remark}\label{rem:someoutside}
	Note that since $T$ acts as a translation on the matching intervals, for any $i\in\{1,\ldots,k\}$ any point $z\in(a_i,b_i)$ satisfies (B2) and also (B4) with $j,j'\in\{0,\ldots,r_i-1\}$. Hence some of the later estimates are possible also for points that not necessarily belong to $M$.
\end{remark}
\noindent Definition \ref{def:balancedm} can be easily extended to consider orbits of length multiplicity of $q$. 
\begin{defn}
    For $0<c<1$, we say that an $L$-fold $(y,q)$-matching set $M$ is 
    $c$-balanced, if {the conditions (B1) to (B4) in Definition 
    \ref{def:balancedm} are satisfied not only by each $x\in M$, but also by 
    all points of the form $T^{iq}(x)$, for  $x\in M$ and $0\le i<L$.}
\end{defn}

{
\subsection{Birkhoff sums estimates}\label{sec:BSestimates}
In this section we collect several estimates on Birkhoff sums of the roof and its derivatives that will be used in the next section in combination with matchings.

\subsubsection{Closest visits}\label{sec:closests}%
The contributions of the \emph{closest visits} of an orbit to the singularities play a key role in the study of Birkhoff sums of the derivatives of functions with logarithmic singularities  
(see e.g.~Proposition~\ref{prop:accelerationset} for the first derivative and Lemma~\ref{lem:secder} for the second). 
We therefore introduce  notation for closest visits to each singularity, from each side (i.e.~from the right and from the left).



\smallskip
Given an IET $T:I\to I$, a point $x\in I$ and a natural number $r$,   let us consider the orbit segment   of length $r$,
and for any  discontinuity $\beta\in Disc(T)$, let us denote by 
$x_\beta^{+}=x_\beta^+ (x,r) = T^{\ell^+_{\beta}}(x)$, where $0\le \ell^+_{\beta}< r$ (resp.~ $x_\beta^{-}=x_\beta^- (x,r)= T^{\ell^-_{\beta}}(x) $, where $0\le \ell^-_{\beta}< r$) 
 the closest visits of the orbit segment  $\mathcal{O}_T(x,r) $ to $\beta$  \emph{from  the right} (resp.~from the left).  Correspondingly, we denote by 
$m_\beta^{+}=m_\beta^+ (x,r)$ (respectively $m_\beta^{-}=m_\beta^- (x,r)$) 
 the minimum distance of the orbit points to  $\beta$  \emph{from  the right} (resp.~\emph{from the left}), namely, denoting by  $(x)^{pos}$ the 'positive' part
of $x$ (defined by $(x)^{pos}:= x$ if $x>0$, $(x)^{pos}:= 0$ otherwise): 
\begin{align*}
m_\beta^{+}& = m_\beta^{+}(x,r) : =   \min \{ (T^j x - {\beta} )^{pos} , \ 0\leq j < r \} = (T^{l_{\beta}^+} x - {\beta} )^{pos}=  x^+_\beta - {\beta} ,\\
m_\beta^{-} &= m_\beta^{-}(x,r): = \min \{  ({\beta}  - T^j x )^{pos}, \  0\leq j < r \} = ({\beta}-T^{l_\beta^-} x )^{pos}= \beta - x^-_\beta,
\end{align*}
 We also  set $m_i(x,r):= \min \{m_i^{+}(x,r), m_i^{-}(x,r) \} $ and 
give a name to the iterates where this closest distances (from either side) is achieved:
\begin{defn}\label{eq: defoftrimming}
For any $x\in I$ and $r\ge 0$, $\beta \in Disc(T)$, 
let  $0\le \ell_{\beta}=\ell_{\beta}(x,r)< r$ 
 be defined as the iterate satisfying
\[
m_\beta(x, r)= |T^{\ell_{\beta}}(x)-\beta|=\min_{0\le j<r}|T^{j}(x)-\beta|. \] 
\end{defn}
\noindent Notice then the distance $m(x,r)$  of $\mathcal{O}_T(x,r)$ from $Disc(T)$ (defined in \eqref{eq:mindistdef} in \S~\ref{sec:iets}) is then simply
$$
m(x,r):= \min  \{ m_\beta (x, r), \ \beta\in Disc(T) \} , \quad \text{where}\ m_\beta(x, r)=  \min \{ m^+_\beta(x, r), m_\beta^-(x, r)\}. 
$$

\subsubsection{Trimmed Birkhoff sums}\label{sec:trimmedBS_def}
The derivative $f'$ of $f$ with logarithmic singuarities 
 is not in $L^1$ 
(since singularities of type  $1/x$ are \emph{not} integrable). To bound the 
derivative Birkhoff sums $S_n f'$ it is  necessary to consider separately the 
contributions of the largest terms, which are controlled by the \emph{closest 
visits} of an orbit to the singularities.  We therefore introduce here the 
following notion of  \emph{trimmed} Birkhoff sum.\footnote{{We remark that 
other possible definitions of trimming appear in the literature. In particular 
in \cite{FKU} the notion of \emph{trimming} and the notation $\tilde 
S_{r}(f)(x)$ refers to a sum to which \emph{both} closest visits from the 
\emph{right} and from the \emph{left} are removed. We only remove the closest 
of the two since we will use trimming for orbits which only get close to the 
singularity $\beta$ from the right, \emph{or} from the left.}} 

Let $f$ be a function  continuous {and differentiable} on each continuity 
intervals of $T$, possibly singular at $Disc(T)$. For any orbit segment 
$\mathcal O_T(x,r)$, we recall that $\ell_\beta$ and $m_\beta (x,r)$ denote 
respectivecly the index and distance of the closest visit to singularities (see 
Definition \ref{eq: defoftrimming} 
 in the above subsection). 

\begin{defn}[Trimmed Birkhoff sums]
For any $x\in I$ and $r\in \mathbb{N}$ the \emph{trimmed $r^{th}$ Birkhoff sum} of $f$ is
\begin{equation}\displaystyle
	\tilde S_{r}(f)(x)=\sum_{\substack{0\leq j < r-1 \vspace{.3mm} \\ j\notin \{\ell_{\beta}\mid \, \beta\in Disc(T)\}}} f(T^j(x)) = \sum_{j=0}^{r-1}  f(T^j(x)) - \sum_{\beta\in Disc(T)} f(T^{\ell_\beta} (x))   .
\end{equation}
\end{defn}
\noindent Thus, the trimmed\footnote{The use of the terminology \emph{trimmed} Birkhoff sums is used in the literature to refer to the operation of  
 (\emph{trimming} a Birkhoff sums 
by  removing the largest elements (one, finitely many, or an increasing number, according to context) from the sum, in order to obtain for example limit theorems, see e.g.~\cite{Au-Sch}).} Birkhoff sum  for us is  defined removing from  Birkhoff sums all contributions given by all the closest visits to singularities. 

\subsubsection{Linear bounds on trimmed derivatives}\label{sec:sc via Bs}
The estimates that we need to prove tightness with exponential tails are what we call \emph{linear} control of the trimmed Birkhoff sums which correspond to subsegments of an orbit segment, as given by the following definition.
\begin{defn}[Orbits with linear bounds on  trimmed derivatives] \label{def:bounded}
We say that an orbit segment $\mathcal{O}_T(x,q)$, where $x\in I$ and $q\in \mathbb{N}$, has $\Cb$-\emph{linear bounds on trimmed derivatives} of $f$ for some $\Cb >0$ (or simply \emph{has linear bounds}) if
\[
\left|\widetilde{S}_r f'(x)\right| \leq \Cb\,\! q
 \ \  \text{for \ any}   \ 0\leq r \leq q.
 \]
\end{defn}
\noindent 
A result which yields linear bounds on trimmed derivatives was proved by the second author in  \cite{Ul:abs}  to prove absence of mixing  (there, linear bounds were shown for orbit segments $\mathcal{O}_T(x,q)$ with $x$  in the base and $q$ equal to  the height of special Rohlin towers built using Rauzy-Veech induction, see Proposition~\ref{prop:accelerationset} for the precise result\footnote{The definition of trimming in \cite{Ul:abs} is different, in that \emph{both} the closest visit from the right \emph{and} from the left are removed, but in the cases which we consider, removing the closest visit only leads to a sum of comparable order. We refer the reader to Section~\ref{sec:trimmedBS_def} for details.}). 
}
}

{
\subsubsection{Relative trimming and  its comparability in balanced matchings}\label{sec:reltrimming}
Given an orbit $\mathcal{O}_T(y,r) $ \emph{matched} to a given orbit segment $\mathcal{O}_T(x,r)$,  the closest visits to a given singularity may not appear at matched times.  In some proofs, it will be useful for us to trim $\mathcal{O}_T(y,r) $ not removing the contribution of closest visits, but 
removing the contributions of points which are \emph{matched} with the closest visits of  $\mathcal{O}_T(x,r)$.\footnote{For example, if  the reference orbit $\mathcal{O}_T(x,q)$ of  a $(x,q)$-matching
 matching set $M$, can be chosen to stay  \emph{far} from $Disc(T)$ (as in Definition~\ref{def:balancedm} of a balanced matching), it is not necessary (and it complicates the estimates) to  remove all the closest visits, since they may be actually very far from the discontinuities.}  We call this \emph{relative trimming}, formally defined as follows.
 


\begin{defn}\label{def:rel_trimming}
Assume that $\mathcal{O}_T(y,q) $ is matched to $\mathcal{O}_T(x,q)$ and consider 
 any orbit subsegment $\mathcal O_T(z,r)\subseteq\mathcal O_T(y,q)$, where the inclusion of orbits segments means that $z=T^j y$ for some $0\leq j<q$ and $0\leq r\leq  q-j$.  
We define the \emph{relatively trimmed Birkhoff sum} $ S_{r}^x(f)(z)$ trimmed relatively to the matched orbit of $x$ as follows:
\begin{align*}
\tilde S_{r}^x(f)(z)& :=\sum_{ j\in \{0,\ldots, r-1\}\setminus \mathcal{M}^x(z,r)} f(T^j(z)), \quad   \textrm{where} \\  \mathcal{M}^x(z,r) & := \{ 0\leq j< r | \ T^j(z)\ \text{ is matched to } T^{\ell_{\beta}}(x)\text{ for some }\beta\in Disc(T)\}.
\end{align*}
\end{defn}
\noindent  Notice that the contributions removed, of the form $T^j (z)$ for $j\in \mathcal{M}^x(z,r)$ are, by definition of $\mathcal{M}^x(z,r)$, those of the iterates which are matched to the closest visits in the orbit segment $\mathcal{O}(x,q)$ of $x$. 

Within a balanced matching set (see Definition~\ref{def:balancedm}), relative trimming and trimming of Birkhoff sums are comparable, as the following lemma shows. 
\begin{lemma}[Comparability of trimmings in a balanced matching set]\label{eq: trimmedythesame}
Given an IET $T$ and a function $f\in SymLog^2(T)$ and let $M$ be 
a balanced $\mathcal{O}(y,q)$-matching set. 
Then there exists $E\in\R$ 
 such that,   
 for every $x\in M$, 
\begin{equation}
	|\tilde S_{q}^x(f)(y)-\tilde S_q(f)(y)|\le  E,
\end{equation}
\end{lemma}
\begin{proof}
{First note that for any $g$ restricted to any $I_i$, with 
$i=1,\ldots,d$, is 
extendable to a $C^2$ function on $\overline{I_i}$, then for any $x\in M$ we 
have
\[
	|\tilde S_{q}^x(g)(y)-\tilde S_q(g)(y)|\le  \#Disc(T)\cdot\|g\|_{\infty}.
\]
Thus, we can assume that $f\in \mathcal{P}SymLog(T)$.}

For every $\beta \in Disc(T)$, let  $\ell_\beta^x:=\ell_\beta^x(y,q)$ be such that $T^{\ell^x_{\beta}}(y)$ is matched to $T^{\ell_{\beta}}(x)$.  
Then the set of indexes $\mathcal{M}^x(y,q)$ (see Definition~\ref{def:rel_trimming} of relative trimming) is exactly 
$
\mathcal{M}^x(y,q)=\{ \ell_\beta^x=\ell_\beta^x(y,q),$ $\beta \in Disc(T)\}
$. 
Thus, recalling the Definition~\ref{def:rel_trimming} of relative trimming and 
the form of the function $f$, we have 
\[
\tilde S_{q}^x(f)(y)=S_{q}^x(f)(y)-\sum_{\beta\in Disc(T)}-C_{\beta}\log(|\beta-T^{\ell_{\beta}^x(y,q)}(y)|)
\]
 where the constants $C_\beta$ for $\beta\in Disc(T)$ are defined to be 
\begin{equation}\label{def:Cbeta}
C_\beta:= \begin{cases}C_i^+ & \text{if}\ \beta=\beta_i\ \text{ and\ the\ closest\ visit\ occurs\ from\ the\ \emph{right},\  or}\\  C_{i}^- &   \text{if}\ \beta=\beta_i\ \text{ and\ the\ closest\ visit\ occurrs\ from\ the\ \emph{left}. }
\end{cases}
\end{equation}
Hence 
\[
\tilde S_{q}^x(f)(y)-\tilde S_q(f)(y)=\sum_{\beta\in Disc(T)}-C_{\beta}\log(|\beta-T^{\ell_{\beta}^x}(y)|)
-\sum_{\beta\in Disc(T)}-C_{\beta}\log(|\beta-T^{\ell_{\beta}}(y)|).
\]
By (B3), and then by (B2), we have that 
\[
|T^{\ell_{\beta}^x}(y)-T^{\ell_{\beta}}(x)|<\frac{1}{cq}\quad \Rightarrow \quad |T^{\ell_{\beta}^x}(y)-\beta|<\frac{c+c^{-1}}{q}.
\]
From this and from (B1) we get that 
\[
C_{\beta}\log(|\beta-T^{\ell_{\beta}^x}(y)|)
-C_{\beta}\log(|\beta-T^{\ell_{\beta}}(y)|)<C_{\beta}(\log(c+c^{-1})-\log(c)),
\]
which finishes the proof.
\end{proof}

\subsubsection{Second derivative estimates} 
To control trimmed Birkhoff sums of the second derivative of $f$ in a balanced matching, we can use  a simple classical estimate, which goes back to the work of Kocergin  (see e.g.~\cite{Ko:abs}) and is based on the sum of reciprocals of squares. This estimate holds whenever there is sufficient \emph{spacing} between points in the orbit (which given  by assumption (B4) in the  Definition~\ref{def:balancedm} of a balanced  matching). 

\begin{lemma}\label{lem:secder}
Given any {$f\in { {\mathcal{P}} SymLog \left(T\right)}$} with symmetric 
logarithmic singularities, let $M$ be a $c$-balanced $(y,q)$-matchings and let 
$x\in M$. 
	Then there exists $\hat C>0$, which depends only on $c$ {and $f$ (but not 
	on the matching  set as long as it is $c$-balanced)}, such that for every 
	$n\in\N$, $0\le r<q$ {and any $z\in M, r\in \N$ such that $\mathcal 
	O_T(z,r)\subset\mathcal O_T(x,q)$ we have the following estimates on the 
	trimmed Birkhoff sums of the second derivative $f''$:}
	\[
	0\leq \tilde S_{r}(f'')(z)\le  \hat Cq^2.
	\]
Moreover, if $z$ satisfies additionally that $m(z,r)\ge \frac{c}{q}$, then we have also a similar estimate for the  second derivative \emph{relative} trimmed Birkhoff sums, namely:
	\[
0\leq \tilde S_{r}^x(f'')(z)\le  \hat Cq^2.
\]
\end{lemma}
\begin{proof}  {Notice first that since $f\in {\mathcal{P}} SymLog 
\left(T\right)$, $f''\geq 0$ and hence both $\tilde S_{r}(f'')$ and $\tilde 
S_{r}^x(f'')$ are non-negative.} 
Since $x\in M$, which is by assumption $c$-balanced, and  $\mathcal O_T(z,r)\subset\mathcal O_T(x,q)$,  
	by assumptions $(B2)$ and $(B4)$ of a $c$-balanced matching (see Definition~\ref{def:balancedm}), 
	we have
	\[
	0\leq \tilde S_{r}(f'')(z)=\sum_{i=0}^{r-1} f''(T^i z)\le C_f\sum_{i=1}^{r}\left(\frac{ic}{q} \right)^{-2}
	=q^2\cdot c^{-2}C_f\sum_{i=1}^{r}\frac{1}{i^2}\le \frac{c^{-2}C_f\pi^2}{6}\cdot q^2.
	\]
	Furthermore, if $z$ is as in assumptions of the lemma, then we can also estimate the relative trimmed Birkhoff sums $\tilde S_{r}^x(f_p'')(z)$ by:
	\[
	|\tilde S_{r}(f'')(z)-\tilde S_{r}^x(f'')(z)|\le 2c^{-2}d\cdot q^2
	\]	Hence, the desired estimate holds  for $\hat C:=\frac{c^{-2}C_f\pi^2}{6}+2c^{-2}d$. 	
\end{proof}
\begin{remark}\label{rem:secder}
Notice that in the proof we have not used all the properties given by the 
assumption that $M$ is a $c$-balanced matching, but only $(B2)$ and $(B4)$. 
Thus, the estimate also holds for any suborbit $\mathcal{O}_T(x',r)$ of an 
orbit $\mathcal{O}_T(x,q)$ that satisfies $(B2)$ and $(B4)$ in the 
Definition~\ref{def:balancedm} of a $c$-balanced $(y,q)$-matching.
\end{remark}
}

{
\subsection{Exponential tails via matchings}\label{sec:exptails_via_match}
In this section we show that combining matchings (from \S~\ref{sec:matchingssec}) with linear bounds on trimmed derivatives (see Def.~\ref{def:bounded} in \S~\ref{sec:sc via Bs}) one can prove 
 \emph{exponential upper bounds} for the tail of Birkhoff sums of the roof function.
 We first introduce in \S~\ref{sec:goodmatchings} the notion of \emph{good matchings}  (adding further requirements to the matching orbit and the matching set of a balanced sequence of matchings  which will be  used to control Birkhoff sums,  see Definition~\ref{def:goodMn}), and  state a result on their existence (see Proposition~\ref{prop:Egoodmatchings}). In \S~\ref{sec:CorinnaConstant} we show that in good matchings one has good control on trimmed Birkhoff sums of matched orbits. Finally, in \S~\ref{sec:exptails_via_match}, we state and prove a criterion which gives exponential bounds from the existence of a sequence of good matchings.  

\subsubsection{Sequences of good matchings}\label{sec:goodmatchings}
A sequence of good matchings is a sequence of \emph{balanced} matching sets (see Definitions~\ref{def:matchingset} and~\ref{def:balancedm} respectively) with additional requirements on the \emph{matching orbit} $\mathcal{O}(y,q)$:  we want the matching orbit to satisfy linear bounds on trimmed derivatives (in the sense of Definition~\ref{def:bounded}). 
 These requirements are formalized in  the following Definition~\ref{def:goodMn}
 of \emph{good matchings}.  

\smallskip

\begin{defn}[good matchings for exponential tightness]\label{def:goodMn}
Given $T$ and $f\in SymLog^2(T)$,  $0<\gamma<1$ 
we say that a  sequence of ($L$-fold) matching sets $(M_n)_{n\in \N}$ is 
{$\gamma$-\emph{good for exponential tightness}} (or for short,   
\emph{$\gamma$-good}, or simply \emph{good}) for $T$ and $f$, 
 if there exists $0<c<1$,  $\kappa \in \mathbb{N}^+$ and {and $B_0>0$,}   such 
 that:
\begin{itemize}
\item[(G0)] $\limsup_{n\to\infty}Leb([0,1]\backslash M_n)\le \gamma$ as $n\to \infty$. 
\item[(G1)]  for any $n\in \N$, $M_n$ is a $c$-balanced ($L$-fold)  $(y_n,q_n)$-$\kappa$-matching set, with $q_n\to\infty$ as $n$ grows;
\item[(G2)] for each $n\in \N$, the matching orbit $\mathcal{O}_T(y_n,q_n)$ has $B_0$-\emph{linear bounds on trimmed derivatives}.
\end{itemize}
In addition, we say that we say that a  sequence of ($L$-fold) matching sets 
$(M_n)_{n\in \N}$ is {$(\gamma,B)$-\emph{good for exponential tightness}} (or 
for short,   \emph{$(\gamma,B)$-good}) for some $B>0$ and $0<\gamma<1$ if 
$(M_n)_{n\in \N}$ is {$\gamma$-\emph{good for exponential tightness}} and in 
addition:
\begin{itemize}
\item[(G3)] all orbits $\mathcal{O}_T(x,q_n)$ with $x\in M_n$ have $B$-\emph{linear bounds on trimmed derivatives, i.e.}
\begin{equation}\label{eq:mainCorinnaconstantq}
	|\tilde S_r(f')(x)|\le \Cb \, {q} \quad \text{ for\ any\ }n\in \N, \ \text{any}\ x\in M_n \text{ and }0\le r\le q_n.
\end{equation}
\end{itemize}
\end{defn}
\noindent Here assumption $(G2)$ concerns only the reference orbits 
$\mathcal{O}(y_n,q_n)$ to which all other orbits in each $M_n$ are matched, 
while the last assumption $(G3)$, requires that all points in $M_n$ satisfy 
linear bounds on trimmed derivatives (in the sense of 
Definition~\ref{def:bounded}) with a \emph{uniform} constant $B$ independent of 
$n$. { We will show that $(G3)$ can be derived from $(G2)$ under suitable 
assumptions (see Lemma~\ref{lemma:linearbounds} in the next 
\S~\ref{sec:CorinnaConstant}).}}


\smallskip

{We will show the following  result, giving the \emph{existence} of good 
matchings for exponential tails for typical IETs. We will build matching sets using BT rigid towers presentations.

{
\begin{defn}[Matching set given by a BT rigidity time]\label{def:givenby}
 We say that
a ($L$-fold)  $(y,q)$-matching $M$ is \emph{given by} a $(C,\epsilon,\rho)$-BT rigidity presentation with $d$ rectangles if 
$q$ is the common height of the towers $\mathcal{T}_2$ and $\mathcal{T}_d$ of a  $(C,\epsilon,\rho)$-BT rigid tower presentation. 



\end{defn}
 
{
For any $f\in SymLog (T)$, let us recall that $C_f$  is the constant  $C_f:= \sum_{i=0}^{d-1} C_i^+$ (see Definition~\ref{def:Cf}) when $f$ is written in the form given by Definition~\ref{def:SymLog}.

\begin{prop}[Existence of good matchings]\label{prop:Egoodmatchings}
For  almost every $\lambda \in \Delta_d$,  if we consider the IET $T=(\pi,\lambda)$ with 
a standard permutation $\pi$, 
 for every $L\in \N$, for  any $0< \rho<1$, and for any $f\in SymLog^2 (T)$, 
there exists $B>0$   such that,   
 for any $0<\epsilon<1$ 
 and any non empty interval $C=[c_0,c_1]\subset {[0,1/10]}$, there  exists  a 
 sequence of $L$-fold matchings $(M_n)_n$ which  {are $(\gamma,B)$-good for 
 exponential tails for   
$\gamma=\sup_{n\in\N}Leb(I\setminus M_n)\le (L+1) \rho c_1+\epsilon$ and 
  given by  $(C,\epsilon,\rho)$-BT rigid 
   presentation. } 
\end{prop}}
\noindent 
{The order of quantifiers in the statement is delicate and important in the 
later proofs. In particular we remark that the rigidity range $C=[c_0,c_1]$ and 
the $\epsilon$ parameters that enter the definition of $(C,\epsilon)$-BT 
rigidity  can be chosen \emph{independently} on $\rho$ and in turn $B$ (which 
depends only on $\rho$ and on $f$). More precisely, one can show that the 
dependence of $B$ on $f$, if we write $f=f^p+g$ (by 
Definition~\ref{def:SymLog}) is through the constant $C_f$ and $\Vert 
g'\Vert_{\infty}$ and is linear in both.} 
The proof of this Proposition will be given in Section~\ref{sec:Eproof}.

\subsubsection{Trimmed BS bounds in good matchings}\label{sec:CorinnaConstant}
{We conclude this subsection proving that $\gamma$-good matchings are actually 
automatically $(\gamma,B)$-good matchings (see Definition~\ref{def:goodMn}).}


\begin{lemma}[Trimmed derivatives bounds]\label{lemma:linearbounds}
 For any $0<\gamma,c<1$,
given  
 any sequence 
 $(M_n)_{n\in \N}$ of $\gamma$-good matchings for  $T$ and $f\in SymLog^2 (T)$, 
 so that for any $n\in \N$, $M_n$ is a $(y_n,q_n)$-matching given by 
 $c$-balanced  presentation, 
 there exists $B>0$, independent of $\gamma$, such that $(M_n)_n$ is a sequence 
 of  $(\gamma,B)$-good matchings.
\end{lemma}
\noindent {The constant $B>0$ (which, by Definition~\ref{def:goodMn} of 
$(\gamma,B)$-good matchings, given the trimmed derivative bounds 
\eqref{eq:mainCorinnaconstantq} for all $x\in M_n$, 
 depends only on the constant $B_0>0$ in the bound for the trimmed derivatives of the matching orbit $\mathcal{O}(y_n,q_n)$ (which exists by $(G2)$ in the Definition~\ref{def:goodMn} of $\gamma$-good matching sets)  and the balance constant $c$ (from  Definition~\ref{def:balancedm}), but not on other parameters of the BT-rigid rectangle presentation inside which $M_n$ is built (in particular not on the parameters $\epsilon>0$ and $C=[c_0,c_1]$). This will be important  in the disjointness proof.}
\begin{proof}
 {In view of the Definition~\ref{def:SymLog} of $f\in SymLog^2(T)$, we can 
 write $f=f_p+g$ where $f_p\in {\mathcal{P} SymLog \left(T\right)}$. Let us 
 first discuss $g$.  
 Notice that since $g$ in $\mathcal{C}^1$ on each $I_i$, $1\leq i\leq d$, we trivially have that $\Vert \tilde S_r(g')(x)\Vert_\infty \le \Vert g'\Vert_\infty  \cdot {q}$ for any $0\le r \le q$. Hence we can assume WLOG that $f=f_p$ i.e.~that  $f\in \mathcal{P} SymLog(T)$.
}
 
Let us first show 
 that \eqref{eq:mainCorinnaconstantq} holds if $\mathcal{O}(x,r)\subset \mathcal{O}(y_n,q_n)$ is an orbit subsegment of the matching orbit $\mathcal{O}(y_n,q_n)$ for $M_n$, or, explicitely, if we assume that $x=T^j y_n$ for some $0\leq j<q_n$ and $0\leq r<q_n-j$. In this case,  we can write
\begin{align*}
|\tilde S_{r}f'(x)| = |\tilde S_{r}f'(T^j y_n)| & =| \tilde S_{r+j}f'(y_n) -\tilde S_{j}f'( y_n)|\leq
\\ & | \tilde S_{r+j}f'(y_n) | + |\tilde S_{j}f'( y_n)|
 \leq  \Cb_0 (r+j)+ \Cb_0 j \leq 2 \Cb_0 q_n,
\end{align*}
where the last estimate follows from the linear bound assumption on $\mathcal{O}_{T}(y_n,q_n)$ (see Definition~\ref{def:bounded}) and the assumption that $ j<q_n$ and $r+j<q_n$ (in view of the orbit segments inclusions). 

\smallskip
Consider now a generic $x\in M_n$.	
	 Since $M_n$ is a $(y_n,q_n)$ $\kappa$-matching set, $\mathcal{O}_T(x, q_n)$ is matched with $\mathcal{O}_T(y_n, q_n)$ (in the sense of Definition \ref{def:orbitmathcing}). Thus, 
since $0\leq r<q_n$ and hence $\mathcal{O}_T(x, r)\subset \mathcal{O}_T(x,q_n) $, we can decompose also $\mathcal O_T(x,r)$  into $\ell\leq \kappa$  disjoint orbit subsegments
$$
		\mathcal{O}_T(x,r) = \mathcal{O}_T(x_1,r_1)\sqcup  \mathcal{O}_T(x_2,r_2)\sqcup \cdots \sqcup \mathcal{O}_T(x_k,r_\ell), \quad \text{with} \ r_1+\ldots+r_\ell=r,$$ 
 and \emph{match} them with $\ell$ disjoint orbit subsegments of  $\mathcal{O}_T(y_n,q_n)$, which we will denote by 
 $ \mathcal{O}_T(y_{n,1},r_1)$,  $\mathcal{O}_T (y_{n,2},r_2)$, $\cdots ,$  $\mathcal{O}_T(y_{n,\ell},r_\ell)$. 
Then, by definition of matchings,  we use mean value to estimate
$$
\left|\tilde S_{r_i} f'(x_i) - \tilde S_{r_i}^{x_i} f'(y_{n,i})\right| \leq |\tilde S_{r_i}^{x_i} f'' (\xi_i)||b_i-a_i|,\qquad 1\leq i\leq k, 
$$
for some $\xi_i\in (a_i,b_i) $, where $(a_i,b_i)$ is the interval given by the definition of matching between $\mathcal{O}_T(x_{i},r_i)$ and $\mathcal{O}_T(y_{n,i},r_i)$. 
By the Condition (B1) applied to $y_{n,i}$ and by the condition (B4) applied to $x_{i}$, we have either
\[
\tilde S_{r_i}^{x_i} f'' (\xi_i)=\tilde S_{r_i} f'' (\xi_i)\qquad\text{or}\qquad m(\xi_i,r_i)\ge \frac{c}{q}.
\]

Now, using the first part  of the proof to estimate  the Birkhoff sum   $\tilde S_{r_i} f'(y_{n,i})$ along the orbit subsegment $\mathcal{O}_T(y_{n,i},r_i)\subset \mathcal{O}_T(y_{n},q_n)$, and Lemma~\ref{lem:secder}  (which can be applied since $f\in \mathcal{P}Sym Log(T)$) to estimate $|\tilde S_{r_i}^{x_i} f'' (\xi_i)|$, we get that, for each $1\leq i\leq k$,  since $b_i-a_i\leq c/q_n $ by properties $(B1)$ and $(B3)$ of balanced matchings (see Definition~\ref{def:balancedm}), we get
$$
\left|\tilde S_{r_i} f'(x_i)\right| \leq  \left| \tilde S_{r_i} f'(y_{n,i})\right|+\frac{2\kappa}{c} + |\tilde S_{r_i} f'' (\xi_i)||b_i-a_i|\leq 2\Cb_0 q_n +   \frac{{\hat{\Cb} q_n^2}}{c q_n} \leq \left(2\Cb_0+\frac{\hat{\Cb}}{c}\right)\, q_n,
$$
(where $\hat{\Cb}$ is the constant given by Lemma~\ref{lem:secder}, {which 
depends only on $c>0$}). Summing up over $1\leq i\leq \ell$, where $\ell\leq 
\kappa$, this concludes the proof of the desired linear bounds on trimmed 
derivatives.
\end{proof}

}

\subsubsection{The exponential tails via matchings criterion}\label{sec:exp_tail_criterion}
We can now state and  prove the criterion to prove exponential tails from the  existence of good matchings (in the sense of Definition~\ref{def:goodMn}).

\begin{prop}[Exponential tightness via matchings] \label{prop:exptails_via_match}
Assume $f\in { SymLog^2 \left(T\right)}$ has symmetric logarithmic singularities. Let $(M_n)_{n\in\N}$ be a sequence of $\gamma$-good $c$-balanced 
$(y_n,q_n)$-matchings.
Then there exists $b_0>0$  and $K>0$ such that for every $b\geq b_0$ we have
$$
\limsup_{n\to \infty } Leb(\{x\in M_n :|S_{q_n}(f)(x)-S_{q_n}(f)(y_n)|\geq b \})\le  K e^{-b/C_f}.
$$
\end{prop}
{\begin{proof}
Since by assumption the $(y_n,q_n)$ matchings $M_n$ are part of a sequence 
$(M_n)_n$ of good matchings,
{by Lemma~\ref{lemma:linearbounds}, there exists 	 $\Cb>0$  such that 
$(M_n)_n$ are $(\gamma,B)$-good matchings,} so that for every $n\in\N$, 
	\begin{equation}\label{eq: Corinnaconstant}
	|\tilde S_{r}(f')(x)|<\Cb q_n\quad\text{for } 0\le r\le q_n, \ {\text{for\ any}\ x\in M_n}.
	\end{equation}
	 Fix $n\in\N$ and $x\in M_n$. 
{Write again  $f=f_p+g$ where $f_p\in {\mathcal{P} SymLog \left(T\right)}$ and 
$g$ is in particular bounded (in view of Definition~\ref{def:SymLog}). 	}
Let us split the Birkhoff sums keeping aside the contributions of the trimmed visits given by the pure logarithmic singularities of $f_p$ as:
	\begin{equation}\label{eq: decompintotrimmedandclosest}
		\begin{split}
	|S_{q_n}&(f_p)(x)-S_{q_n}(f_p)(y_n)|\\
	&\le \left|\sum_{\beta\in Disc(T)}{C_\beta}\log(m(x,\beta,q_n))-{C_\beta}\log{(m(y_n,\beta,q_n))}\right| 
	+|\tilde S_{q_n}(f_p)(x)-\tilde S_{q_n}(f_p)(y_n)|.
	\end{split}
	\end{equation}
Since all good matchings 	$M_n$ are $\kappa$-matching sets with the same $\kappa$ (see Definition~\ref{def:matchingset}), there exists a decomposition of the orbit $\mathcal O_T(x,q_n)$ and $\mathcal O_T(y,q_n)$ into $k\le \kappa$ pieces
	\begin{align*}
		\mathcal{O}_T(x,q_n)& = \mathcal{O}_T(x_1,r_1)\sqcup  \mathcal{O}_T(x_2,r_2)\sqcup \cdots \sqcup \mathcal{O}_T(x_k,r_k), 
		\\ \mathcal{O}_T(y_n,q_n)& = \mathcal{O}_T(y_{n,1},r_1) \sqcup  \mathcal{O}_T (y_{n,2},r_2) \sqcup \cdots \sqcup  \mathcal{O}_T(y_{n,k},r_k),
	\end{align*}
with $r_1+\ldots+r_k=q_n$, as in Definition \ref{def:orbitmathcing}. 
{In view of Lemma \ref{eq: trimmedythesame},} we can estimate the Birkhoff sums 
in the following way
\begin{align}\label{eq:matchingsplit}
|\tilde S_{q_n}(f_p)(x)-\tilde S_{q_n}(f_p)(y_n)|& \le \left|\tilde S_{q_n}(f_p)(x)-\tilde S_{q_n}^x(f_p)(y_n)\right|+E\\ & \nonumber =\left|\sum_{i=1}^{\kappa} \left(\tilde S_{r_i}(f_p)(x_i)-\tilde S_{r_i}^x(f_p)(y_{n,i})\right) \right|+E,
\end{align}
{where  $E$ is the constant given by Lemma \ref{eq: trimmedythesame}. 

We claim that we get that for every $i=1,\ldots,k$ and every $z\in[x_i,y_{n,i}]$ we have the following estimate
\[
\left|\tilde S_{r_i}^x({f_p}')(z)\right|\le \Cb'\cdot q_n \ \ \text{ for every }z\in[x_i,y_{n,i}],\qquad \text{where}\ \Cb':=\Cb+\frac{\hat \Cb}{c}.
\]
{
(Notice that, although this is a linear bound on trimmed Birkhoff sums of the derivative, it does not follow from $(G2)$ in the Definition~\ref{def:goodMn} of good matching since  interval $[x_i,y_{n,i}]$ is not necessarily included in $M_n$, so we have to prove it.) 
 To prove the claim,}} in view of Remark \ref{rem:someoutside}, by Lemma \ref{lem:secder} (and Remark \ref{rem:secder}), we get
\begin{equation}\label{eq: secderwithnovis}
|\tilde S_{r_i}^x({f_p}'')(z)|\le \hat \Cb\cdot q_n^2, \text{ for every }z\in[x_i,y_{n,i}]
\end{equation}
where $\hat \Cb$ does not depend on the point or $n$. 
By \eqref{eq: Corinnaconstant} and by (B3) in the Definition \ref{def:balancedm}, we then get the claimed bound. 
 Now, by Mean Value Theorem we get from \eqref{eq:matchingsplit} that 
\begin{equation}\label{eq:relBS}
|\tilde S_{q_n}(f_p)(x)-\tilde S_{q_n}(f_p)(y_n)|\le \Cb'', \quad where \, \Cb'':=\frac{\Cb'\kappa}{c}+E.
\end{equation}
{To estimate the analogous quantity for $g$, one can again use the 
decomposition into $k$ matched pairs of suborbits and mean value theorem, with 
the trivial estimates $\Vert S_{r_i}g'\Vert_\infty\leq r_i \Vert 
g'\Vert_{\infty}$ (and recalling that $r_1+\cdots +r_\kappa=q_n$ and 
$q_n|I_n|\leq 1$) to get  that
$$
| S_{q_n}(g)(x)- S_{q_n}(g)(y_n)|\le \left|\sum_{i=1}^{\kappa} \left(\tilde S_{r_i}(g)(x_i)-\tilde S_{r_i}(g)(y_{n,i})\right) \right|\leq \sum_{i=1}^\kappa \Vert g'\Vert_\infty r_i |I_n| \leq  
\frac{\Vert g'\Vert_\infty}{c} .
$$
Combining these estimates with 
\eqref{eq: decompintotrimmedandclosest}  and \eqref{eq:relBS}, 
we get
\begin{equation}\label{eq:forgottentrimmed}
	|S_{q_n}(f)(x)-S_{q_n}(f)(y_n)|\le \left|\sum_{\beta\in Disc(T)}{C_\beta}\log(m(x,\beta,q_n))-{C_\beta}\log{(m(y_n,\beta,q_n))}\right|+\Cb''  
+\frac{\Vert g'\Vert_\infty}{c}	
	\end{equation}
 Set $\Cb''':= \Cb''+\Vert g'\Vert_\infty/c$.  } Let $b_0:= 10(\Cb'''/C_f-\log 
 c)$, where $c$ is given by (B1) and let $b\ge b_0$. We claim that if 
\[
\min_{\beta\in Disc(T)} m(x,\beta,q_n)\ge \frac{e^{-b/C_f+(\Cb'''/C_f-\log c)}}{q_n},
\]
then $|S_{q_n}(f)(x)-S_{q_n}(f)(y_n)|<b$. Indeed, by \eqref{eq:forgottentrimmed} and  properties (B1) and (B2)  in the Definition~\ref{def:balancedm} of balanced matching, we have 

\[
|S_{q_n}(f)(x)-S_{q_n}(f)(y_n)|\le \sum_{\beta\in Disc(T)}C_{\beta}(-\log e^{-b/C_f+(\Cb'''/C_f-\log c)}+\log c)+\Cb'''\le b.
\]
Thus 
\begin{align*}
\{x\in M_n & :|S_{q_n}(f)(x)-S_{q_n}(f)(y_n)|\geq b \} \subset \\ & \subset  \bigcup_{\beta\in Disc(T)}
\bigcup_{i=0}^{q_n}T^{-i}\left[\beta-\frac{e^{-b/C_f+(\Cb'''/C_f-\log c)}}{q_n},\beta+\frac{e^{-b/C_f+(\Cb'''/C_f-\log c)}}{q_n}\right].
\end{align*}
Hence we get that
\[
Leb\left( \{x\in M_n :|S_{q_n}(f)(x)-S_{q_n}(f)(y_n)|\geq b \}  \right)\le K\cdot e^{-b/C_f},
\]
where $K:=2d \cdot e^{\Cb'''/C_f-\log c}$, which finishes the proof.

\end{proof}
}

\section{Construction of BT rigidity degenerations and good matchings.}\label{sec:Eproof}

The main goal of this section is to  prove Proposition~\ref{prop:Egoodmatchings}, i.e.~
to show that for almost every (Keane) IET we can obtain a sequence of good matchings.  
To produce good matchings, we first prove Proposition~\ref{prop:Edegenerations}, by showing the existence of rectangle presentations (see \S~\ref{sec:rectangles}) with BT rigid presentations (in the sense of Definition~\ref{def:BTrectangles}) for almost every translation surface in $\mathcal{H}(1,1)$. These, in turn, are constructed using (natural extensions of) the Rauzy-Veech induction algorithm, by now a classical and powerful tool to study IETs and translation surfaces. 

 In \S~\ref{sec:RV} we first recall some basic definitions and properties concerning Rauzy-Veech induction 
needed in the rest of this paper. The existence of BT-degenerations (Prop.~\ref{prop:Edegenerations}) is proved in \S~\ref{sec:degenerationRV} cutting and pasting good presentations of zippered rectangles. Then, in \S~\ref{sec:derivativesviaRV} we recall a result proved by the second author  \cite{Ul:abs} (see also \cite{FKU} for the formulation given here) that allows to produce times with linear bounds on trimmed derivatives using (the natural extension of) Rauzy-Veech induction. 
Finally, we combine these two results in  \S~\ref{sec:Edegenerations}.
to produce the desired good matchings from BT-degeneration times $q_n$ such that there is an orbit $\mathcal{O}_T(y_n,q_n)$ with linear bounds on trimmed Birkhoff sums.

\subsection{Background on Rauzy-Veech induction}\label{sec:RV}
 We recall here some results on Rauzy-Veech (RV for short) induction. For more details on the definition of the algorithm 
 we refer e.g.\ to  lecture notes by Yoccoz \cite{Yo} or Viana \cite{Vi}.

\subsubsection{The Rauzy-Veech algorithm for  IETs.} \label{sec:basic}
The Rauzy-Veech induction algorithm associates to
a.e.~IET, 
 a sequence of IETs  which are induced maps of $T$ onto a sequence of nested subintervals $I^{(n)}$ contained in $[0,1]$. These are chosen so that the induced maps are again  IETs of the same number $d$ of exchanged intervals.  The interval  $I':= I^{(1)}\subset I^{(0)}:=[0,1]$ associated to one step of the algorithm  is defined as follows: let $I^{t}_d:=[\beta_{d-1},1]$ (where $t$ stays for \emph{top}) be the last (i.e.~right-most) exchanged subinterval and let $I^b_d$
 be the  last after the exchange, i.e.~$I^b_d:=T(I_{j_d})$ where $j_d$ is such that  
  $T(I_{j_d})$ has $1$ as right endpoint. 
If $|I^t_d|>|I^b_{d}|$ (called case \emph{top} or \emph{type} $0$), $[0,1]\backslash I^t_d =[0,\beta_{d-1})$, while  if $|I^b_{d}|>|I^t_d|$ (called case \emph{bottom}, or \emph{type} $1$), $I':=[0,1]\backslash I^b_d$.  
We denote by $T'$ is the induced IET obtained as first return map on $I'$. 
The map  \emph{Rauzy-Veech map} $\mathcal{V}$ then associates to a $T=(\pi,\lambda)$  such that $|I^t_d|\neq |I^b_d|$
  the IET $\V (T)=\V (\pi,\lambda)$  obtained by renormalizing $T'$ by $|I'|$ so that the renormalized IET is again defined on an unit interval. If $|I^t_d|=|I^b_d|$, the induction is not defined. 
 
\smallskip

 The {Rauzy-Veech map}  $\V $ is then defined for all $n\in \N$ on a 
 full Lebesgue measure subset of the space
\begin{equation}\label{eq:defX}
X:= X(\mathcal R) =\mathcal R\times  \Delta_{d}  = \{ (\pi, \lambda)\ | \ \pi \in \mathcal R,\ \lambda \in \Delta_d \},
\end{equation}
where $\mathcal{R}$ is the \emph{Rauzy-class}  of $\pi$ (namely the smallest invariant set of permutations closed under iterations of $\V$) and $\Delta_d$ the lengths simplex, defined in \S~\ref{sec:iets} (precisely on the set 
 of IETs which satisfy the Keane condition
\smallskip
Veech proved in \cite{Ve:gau} that $\V$ is conservative and admits an invariant measure $\mu_{\V}$ on $X(\mathcal{R})$ which 
 is equivalent to the Lebesgue measure, but infinite. Zorich showed in \cite{Zo:fin} that one can accelerate\footnote{The acceleration of a map is obtained a.e.-defining  an integer valued function $z(T)$ which gives the return time to an appropriate section. The accelerated map is then given by $\mathcal{Z} (T) := \V^{z(T)} (T )$.} the map  $\V $ in order to obtain a map $\mathcal{Z}$, which we call \emph{Zorich map}, that admits a  \emph{finite} invariant measure $\mu_{\mathcal{Z}}$. Let us also recall that both $\V$ and its acceleration $\Z$ are \emph{ergodic} 
 with respect to $\mu_{\V}$ and $\mu_{\Z}$ respectively \cite{Ve:gau, Zo:fin}.
 
\subsubsection{Rauzy-Veech matrices and  lengths cylinders.}\label{sec:RVmatrices}
Given  $T=(\pi,\lambda)$ in the domain of $\V$, write $\mathcal{V}(\pi,\lambda)=(\pi',\lambda')$. Then, denoting by $|\cdot |$  the vector norm $|{\lambda}|=\sum_{i=1}^d \lambda_i$,  one can write
$$\lambda'=\frac{(A)^{-1}\lambda}{|(A)^{-1}\lambda|}, \quad \text{where}\  A=A(T) \in SL(d,\mathbb Z)$$
is 
 a non-negative matrix with entries $0$ or $1$ and with $1$ on the diagonal and only one $1$ off the diagonal. 
For any $n\in \mathbb{N}$,  set $A_n:= A(\V^n T )$. We will say that $(A_n)_n$ is the sequence of Rauzy-Veech matrices associated to $T$. Then
\begin{equation}\label{eq:Bgamma}
\mathcal V^n(\pi_0,\lambda)= \left(\pi^{(n)}, \lambda^{(n)}\right), \quad \text{where}\ 
\lambda^{(n)}=\frac{A(0,n)^{-1}\lambda}{\left| A(0,n)^{-1}\lambda\right|}, 
\ \text{for} \ \ A(0,n):=A_1 A_2 \cdots   A_{n} .
\end{equation}

\smallskip

\subsubsection{Invertible Rauzy-Veech induction and zippered rectangles}\label{sec:invertibleRV}
The \emph{natural extension} $\widehat{\V}$ (resp.~$\widehat{\Z}$) of the Rauzy-Veech map $\V$ (resp.~of the  Zorich map $\Z$), also known as \emph{invertible Rauzy-Veech} (resp.~\emph{invertible Zorich} induction) is an invertible  map $\widehat{\V}$ (resp.~$\widehat{\Z}$) defined on a domain $\widehat{X}$ 
which is an \emph{extension} of $\V$ (resp.~$\Z$), i.e.~such that there exists a projection $p \colon \widehat{X} \rightarrow X$ for which $p \widehat{\V} = \V p$ (resp.~$p \widehat{\Z} = \Z p$).

To define the domain of the \emph{natural extensions} of $\widehat{\V}$ and $\widehat{\Z}$,  for each $\pi \in \mathcal R$ let $\Theta_{\pi}\subset \mathbb{R}^{d}$ be the polyhedral cones  given by the inequalities 
\bes
\Theta_{\pi} := \left\{ {\tau}=(\tau_1,\dots ,\tau_d)  \in   \mathbb{R}^{d}\ \Big| \quad  \sum_{i=1}^{k} \tau_{\pi^{-1}i} <0, \, k=1,\dots, d-1 \right\},
\ees
and let $\widehat{X}_d$ be the following space of triples  (that admits a geometric interpretation in terms of the space of zippered rectangles,  see for example  \cite{Yo, Vi} for details):
\bes
\widehat{X}_d := \widehat{X}_d (\mathcal R)=\bigcup_{\pi\in \mathcal{R} }\{ \pi\} \times \Delta_d\times \Theta_\pi  = \{ \, (\pi, {\lambda}, {\tau} ) \ | \quad \pi \in \mathcal{R}, \, {\lambda}\in \Delta_d, \, {\tau} \in \Theta_{\pi}   \}.
\ees
\noindent The action of the invertible extension $\widehat\V$  on triples $(\pi,\lambda, \tau) \in \widehat{X}_d$ is obtained by extending the projective action of $\V$, i.e.~defined by
$$
\widehat{\mathcal{V}}(\pi,\lambda, \tau)=(\pi',\lambda', \tau'), \qquad \text{where}\ (\pi',\lambda')=\V (\pi,\lambda), \quad \tau'=\frac{A^{-1}\tau}{|A^{-1}\lambda|}, 
$$
where $A=A(T)\in SL(d,\mathbb Z) $ is the Rauzy-Veech matrix defined in 
\S~\ref{sec:RVmatrices}. {To each triple $(\pi,\lambda, \tau) \in 
\widehat{X}_d$ with a (bi-)infinite orbit\footnote{This is the case for example 
if both $\lambda$ and $\tau$ have coordinates that are irrationally 
independent. More precisely, $\widehat{\mathcal{V}}^i(\pi,\lambda, \tau)$ is 
defined for all $i\in \mathbb{Z}$ if and only if $(\pi,\lambda)$ satisfies the 
Keane condition (that guarantees that $\widehat{\mathcal{V}}^i(\pi,\lambda, 
\tau)$ is defined for all $i\geq 0$, and the translation  surface associated to 
$(\pi,\lambda,\tau)$ has no horizontal saddle connections, see \cite{MUY} and \cite{Be:back}.} under 
$\widehat{\mathcal{V}}$ we can associate a \emph{bi-infinite} sequence 
$(A_i)_{i\in \mathbb{Z}}$ of Rauzy-Veech matrices, by setting 
\begin{equation}\label{eq:seqA}
A_i:=A (\pi^{(i)}, \lambda^{(i)} ), \quad \text{where}\ (\pi^{(i)}, \lambda^{(i)} , \tau^{(i)}):= \mathcal{V}^i (\pi, \lambda, \tau), \ \  \text{for \ all}\ i\in \mathbb{Z}.
\end{equation}}

It is well known that $\widehat{\mathcal{V}}$ preserves a natural (infinite)\footnote{To have a \emph{finite} invariant measure, one can consider again the (invertible) Zorich acceleration $\widehat{\mathcal{Z}}$ of $\widehat{\mathcal{V}}$, restricted to the hypersurface $\hat X^1_d$ of triples   $(\pi,{\lambda},  {\tau}) $  with $Area (\pi,{\lambda},  {\tau}) =1$.} invariant measure $\widehat{\mu}_{\mathcal{V}}$ on $\widehat{X}_d$, absolutely continuous with respect to the Lebesgue measure.  
Moreover, $\widehat{\V}$ is \emph{conservative} and ergodic with respect to $\widehat{\mu}_\V$; in particular, Poincar{\'e} recurrence holds (see e.g.~\cite{Vi} for details).

\subsubsection{Cylinders and Markovian structure}\label{sec:loss} {
Given any triple $(\pi, \lambda, \tau)$, let  $A_i:=A(\mathcal{V}^i(\pi, 
\lambda, \tau) )$ for $i\in \mathbb{Z}$ be the associated (bi-sided) sequence 
of RV matrices, see \eqref{eq:seqA}. Fix two integers $n, m>0 $ and set 
\begin{equation}\label{eq:Apm}
A^-:=A_{-m}A_{-m+1}\cdots A_{-1}, \quad \text{and}\quad A^+:=A_{0}A_1\cdots A_{n-1},
\end{equation} and  
let us define 
\begin{equation}
\Delta_{A^+}:= A^+\, \Delta_d= \left\{ \frac{A^+ \lambda}{|A^+ \lambda|}, \quad \lambda \in \Delta_d\right\}\subset \Delta_d
\end{equation}
the simplex image of $\Delta_d$ by the projective action of $A^+$, and denoting by $\pi^{(-m)}$, as in \eqref{eq:seqA}, the permutation of $ \mathcal{V}^{-m}(\pi, \lambda, \tau)$,  
\begin{equation}
\Theta_{A^-}:= A^-\, \Theta_{\pi^{(-m)}}= \left\{ {A^- \tau}, \quad \tau \in \Theta_{\pi^{(-m)}} \right \}\subset \Theta_\pi
\end{equation}
the simplex image of $\Theta_{\pi^{(-m)}}$ by the action of $A^-$. Notice that if $A^\pm=B^\pm C^\pm$ then
\begin{equation}\label{eq:inclusion}
\Delta_{B^+C^+}\subset \Delta_{B^+}, \quad \text{while}\quad \Theta_{B^-C^-}\subset \Theta_{C^-}. 
\end{equation}
The product  set $\{\pi\}\times \Delta_{A^+} \times \Theta_{A^-} $ 
plays the role of (bisided) \emph{cylinder} for the symbolic coding of RV induction\footnote{Indeed, given any triple  
$(\pi, \lambda',\tau')$ with $\lambda' \in \Delta_{A^+}$ and $\tau' \in \Theta_{A^-}$, is such that the types and the Rauzy-matrices of the iterates  $\mathcal{V}^i(S)$ for $i=-{m},\dots, n$ have the same permutations and the same types than the  corresponding iterates    $\mathcal{V}^i(T)$ for $i=-{m},\dots, n-1$ of $T$; thus in particular the sequence $(A'_i)_{i\in \mathbb{Z}}$ or Rauzy-Veech matrices for $(\pi, \lambda',\tau')$  is such that $A_i'=A_i$ for $i=-{m},\dots, n-1$.}. In particular, we have the following relation  (see for example \cite{FKU}, \S~8.2), which shows that $\mathcal{V}^n$ acts as (the $n$-power of) the \emph{shift} map on the symbolic coding given by RV matrices:
\begin{equation}\label{bisided_Markov}
 \widehat{\mathcal{V}}^n \left( \{\pi\}\times \Delta_{A^+} \times \Theta_{A^-}\right) =
 \{ \pi_{n} \}\times \Delta_d \times \Theta_{A^-A^+} .
\end{equation} 
We recall that if $A^+>0$ is a positive matrix (i.e.~all entries are strictly positive), then $\Delta_{A^+}\subset \Delta_d$ is compact. We say (following \cite{AGY}) that $A^-$  is \emph{strongly positive} if  $\Theta_{A^-} \subset \Theta_\pi$ is also compact. For almost every $(\pi,\lambda,\tau)$ there exists $n,m$ such that $A^+$ and $A^-$ defined as in \eqref{eq:Apm} are both strongly positive, so that the product $\{\pi\}\times \Delta_{A^+} \times \Theta_{A^-}$  is a compact set in $\hat X_d$. We say that $K\subset \hat X_d$ is a \emph{neat} compact set if it  has the form $K=\{\pi\}\times \Delta_{A^+} \times \Theta_{A^-}$ where $A^+, A^-$ are both strongly positive matrices, that correspond to a \emph{neat} path\footnote{We recall that the types  of a (finite) sequence of moves of Rauzy-Veech induction can be described in terms of (finite) paths on the Rauzy-class of $\pi$. In \cite{AGY}, a path $\gamma$ is called \emph{neat} if $\gamma=\gamma_i\gamma_e$ implies that either $\gamma=\gamma_i$ or $\gamma=\gamma_e$. } in the sense of \cite{AGY}.}}

\subsection{Constructing bounded type rigidity degeneration via Rauzy-Veech}\label{sec:degenerationRV}
In this section we prove Proposition~\ref{prop:Edegenerations}, by showing the existence of 
BT rigid  rectangle presentations (as defined in \S~\ref{sec:BTrectangles})  
for a.e.~$S\in \mathcal{H}(1^{2g-2})$. To build these rectangle presentations,  
we will use the invertible Rauzy-Veech induction $\hat{\mathcal{V}}$ and 
zippered rectangles. 

\subsubsection{The BT rigidity  open sets.}\label{sec:openLemma}
The following Lemma provides some 
open sets  in the parameter space 
$\hat{X}_{4g-3}$
 giving $(C,\epsilon,\rho)$-BT rigidity:
{
\begin{lemma}[Rauzy-Veech parameters for BT rigid rectangle presentations] \label{lemma_openset} For every 
$C\subset \R_+$, where $C=[c_0,c_1]$ with $  0<c_0<c_1<1$
and any  $0<\epsilon, \rho<1$,  there exists an open sets $\mathcal{U}= \mathcal{U}(C, \epsilon, \rho)$ in the space 
  of zippered rectangles $\hat{X}_d$ with $d=4g-3$  such that if the  
  translation 
  surface $S$ is given by a triple $(\pi,\lambda,\tau)$
in  $\mathcal{U}$, then $S$ admits a $(C,\epsilon,\rho)$-BT rigid 
rectangle presentation. %
\end{lemma}}
\noindent 
We anticipate that the BT rigid rectangle presentation that we will build are not directly given by the zippered rectangles corresponding to the triple, 
 but are obtained \emph{from} them 
 by suitably cutting and pasting them (see Figure~\ref{fig:BTrigid_cutpaste} and the proof of Lemma~\ref{lemma_openset}).

\begin{proof}[Proof of Lemma~\ref{lemma_openset}]
{
Let
us first define 
auxiliary open set $\mathcal{U}':= \mathcal{U}'(C', \epsilon',\rho')$, 
for parameters $\epsilon', \rho'$ in $(0,1)$ and  $C'=[c_0',c_1']$ with $0< c_0'< c_1'<1$. 
 We will show that, given $\epsilon, \rho$ and $C$ as in the statement, if $\epsilon',\rho'$ are chosen sufficiently small and $c_0',c_1'$ suitably, then we can take  $\mathcal{U}'(C', \epsilon',\rho')$ to be the desired open set. 
 The set  $\mathcal{U}' = \mathcal{U}'(C', \epsilon',\rho') $  
 is  the subset of 
triples  $(\pi, \lambda, \tau)$ in $\hat{X}_d$ given by
$$
\mathcal{U}':=
 \{ (\pi, \lambda, \tau)\in  
\{\pi_0\}\times \Lambda_{c_0',c_1',\epsilon'} \times \Omega_{\rho'} 
\}, $$
where $\pi_0$ is a standard permutation and $ \Lambda_{c_0',c_1',\epsilon'} 
\subset 
\Delta_d$ and $\Omega_{\rho'}\subset \Theta_{\pi_0}$ are given by:
\begin{align}\label{eq:defLambda'}
\Lambda_{c_0',c_1',\epsilon'} &:=  \{\lambda \in \Delta_d \mid \lambda_{2}, 
\ldots, \lambda_{d-1}<\epsilon' \  \text{ and } c_0' 
<\lambda_{1}-\lambda_{d}<c_1'   \}  \subset \Delta_d,\\ 
  \Omega_{\rho'} & :=\{\tau\in \Theta_{\pi_0}\mid 0<\tau_{2}, \ldots, 
  \tau_{d-1}<\rho'  \}\subset \Theta_{\pi_0} .\label{eq:defOmega'}
\end{align}
Since $\mathcal{U}'$ 
is defined by open conditions it is clearly open.}  An example of a polygonal presentation as well as the corresponding zippered rectangles with data  $(\lambda, \tau, \pi)$ 
 in $\mathcal{U}'$ 
 are shown in Figure~\ref{fig:BTrigid} (left).  Let us denote $R_1, \dots , 
 R_d$ the rectangles of the corresponding zippered rectangles ($R_i$ being the 
 rectangle over the interval $I_i$, $1\leq i\leq d$). Observe that   that, on 
 one hand (because of the inequalities on $\lambda$) the  'middle rectangles'  
 $R_2,\ldots, R_{d-1}$ have $\epsilon'$-small width,  while on the other hand,  
 there is a  substantial shift dynamics 
on the remaining two towers (that could be seen as a bounded type rotation-like 
dynamics), {given by $\lambda_1-\lambda_d$, that by construction satisfies  
$c_0'\leq\lambda_1-\lambda_d \leq c_1'$.}  
Moreover, (in view of the inequalities on $\tau$), the singularities of the 
translation surface (shown as red dots in Figure~\ref{fig:BTrigid}) are all 
\emph{cramped}, i.e.~also have vertical displacement at most $d\rho'$. 
Furthermore, the 'middle' heights $q_2,\ldots, q_{d-1}$ are all comparable (they 
differ by at most $d\rho'$). 

\begin{figure}
	\includegraphics[scale=2]{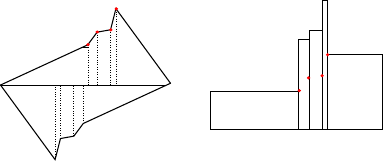}
	\caption{The polygonal representation (left)  and the corresponding zippered rectangles presentation (right) given by a triple $(\lambda,\pi,\tau)$ in $\mathcal{U}'$.
	 In this picture the width of the rectangles $R_2,R_2,R_4$ over $I_2, I_3, I_4$ is less then  $\epsilon'$ and  as the difference between the vertical coordinates of the singularities (red dots in the picture) is controlled by $\rho'$. \label{fig:BTrigid}}  
	\end{figure}
	
	\smallskip
	To produce a BT rigid rectangle presentation from these zippered rectangles,  
we consider the shorter base interval $\widetilde{I}\subset I:=[0,1]$ given by 
 $$\tilde I:=\left[\tfrac{1}{8}, 
 \tfrac{9}{8}-\la_{d}\right)=\left[\tfrac{1}{8}, (1-\la_{d})+ 
 \tfrac{1}{8}\right) \subset [0,1].$$ 
A rectangle presentation over $\tilde I$ (with $d+1$ rectangles) is obtained by 
cutting and stacking the rectangles over $I$ (see 
Figure~\ref{fig:BTrigid_cutpaste}) to get rectangles $\tilde R_1, \cdots, 
\tilde R_{d+1}$ over the intervals $\tilde I_1, \cdots \tilde I_{d+1}$ of the 
IET $\tilde T$ obtained inducing $T$ to $\tilde I$.  The  glueings between the 
top sides  of the rectangle presentation are given by the IET $\tilde T$. One 
can check (see Figure~\ref{fig:BTrigid_cutpaste}) that  the permutation 
$\tilde{\pi}$ of $d+1$ elements of $\tilde I$  and the length vector 
$\tilde{\lambda}=(\tilde{\lambda}_1, \ldots, \tilde{\lambda}_{d+1})$ are given 
by:
\begin{equation}\label{eq:lengths}
\tilde{\pi} = (2\ \star\ \star\ \star\ (d+1)\ 1), \qquad
\tilde\lambda
=\left(\lambda_{1}-\lambda_{d} ,
\lambda_{d}- \tfrac{1}{8},
\lambda_{2},
\ldots,
\lambda_{d-1}, \tfrac{1}{8}\right).
\end{equation}
Thus, all  intervals of continuity of $\tilde T$ are $\epsilon'$-small, but 
 the first, the second and the last one. The second and the last one,  $\tilde 
 I_2$ and  $\tilde I_{d+1}$, undergo the same displacement (i.e.~their distance 
 is the same before and after the exchange). The 
 first interval $\tilde I_1$, by definition of $\Lambda_{c_0',c_1', \epsilon'}$, has length of $c_0'<|\tilde I_1|<c_1'$  and  is responsible 
for the value of the displacement  of  $\tilde I_2$ and  $\tilde I_{d+1}$ under 
$\tilde T$.

 Notice also that the rectangle over $\tilde I_1$ has height $q_1$, and the 
 rectangles over $\tilde I_2$ and  $\tilde I_{d+1}$ are of equal height (see 
 Figure~\ref{fig:BTrigid_cutpaste}), given by 
 $q:=q_1+q_d$ (since they are both obtained cutting and stacking parts of the zippered rectangles over $I_1$ and $I_d$). 
 
Let us estimate the relative displacement $\Delta:= q |\tilde I_1|$. By construction, since 
$|\tilde I_1|= \lambda_1-\lambda_d $, the inequalities on $\lambda$ given by 
being in 
$\Lambda_{c_0',c_1',\epsilon'}$, 
given that  $c_0'<|\tilde I_1| <c_1' $. 
	To guarantee that $\Delta\in C=[c_0,c_1]$, we need to show that we can pick $c_0'$ and $c_1'$ so that $c_0\le qc_0'<qc_1'\le c_1$. For this purpose, let us evaluate $q$. On one hand we have
	\[
	1-\operatorname{Area}(\tilde R_1)\ge \operatorname{Area}(\tilde 
	R_2\cup\tilde R_{d+1})=1-\operatorname{Area}(\tilde R_1\cup \tilde R_3\cup 
	\ldots\cup \tilde R_d)
	\]
	and since $\operatorname{Area}(\tilde R_2\cup \tilde R_{d+1})=q(\tilde 
	\lambda_2+\tilde \lambda_{d+1})=q\lambda_d$, we get
	\[
	\lambda_d^{-1}(1-\operatorname{Area}(\tilde R_1))\ge q \ge 
	\lambda_d^{-1}(1-\operatorname{Area}(\tilde R_1\cup \tilde R_3\cup 
	\ldots\cup \tilde R_d)).
	\]
	From the conditions on $\lambda_d$, we have
	\[
	2\lambda_d-c_1'\le \lambda_1+\lambda_d\le 1=\lambda_1+\ldots+\lambda_d\le 
	2\lambda_d+\epsilon'-c_0'
	\]
	and thus ${(1-\epsilon'+c_0')}/{2}\le \lambda_d\le 
	{(1+c_1')/}{2}$. 
	Moreover, since $q_1\le q$, then $\operatorname{Area}(\tilde R_1)\le qc_1'$.
Hence
	\[
	q\le \lambda_d^{-1}(1-\operatorname{Area}(\tilde R_1))\le \frac{2}{1-\epsilon'+c_0'}(1-qc_1') \quad \Rightarrow \quad
	q\le \frac{2}{(1-\epsilon'+c_0')(1+c_1')}.
	\]
	Moreover, since the heights of zippered rectangles are given by $[\tilde 
	q_j]_{j=1}^d=\Omega_{\tilde \pi}^T[\tau_j]_{j=1}^d$, we have that the 
	heights $q_3, \ldots, q_d$ of the rectangles $\tilde R_3, \ldots, \tilde 
	R_d$ satisfy 
	\begin{equation}			
	-d\rho'\le-d\max_{j=2,\ldots,d-1} \tau_j\le \min_{j=3,\ldots,d} (q_j-q)\le 
	\max_{j=3,\ldots,d} (q_j-q) \leq d\max_{j=2,\ldots,d-1} \tau_j\le d\rho'
	\end{equation}
	and thus
	\[
		q\ge  \lambda_d^{-1}(1-\operatorname{Area}(\tilde R_1\cup \tilde 
		R_3\cup \ldots\cup \tilde R_d)) 
		\ge \frac{2}{1+c_1'}(1-qc_1'-q\epsilon'- d\rho'\epsilon').
	\]
	Thus we obtain
	\[
	q\ge \frac{2(1-d\rho'\epsilon')}{(1+c_1')(1+c_1'+\epsilon')}\ge 
	\frac{2(1-3\epsilon')}{(1+c_1')(1+c_1'+\epsilon')}.
	\]
	
	\smallskip
To conclude, for a given $c_0,c_1$, we need to pick $c_0'$ and $c_1'$ so that 
\begin{equation}\label{eq: precise c's}
c_0<\frac{2c_0'(1-3\epsilon')}{(1+c_1')(1+c_1'+\epsilon')} \quad \text{ and }\quad c_1> \frac{2c_1'}{(1-\epsilon'+c_0')(1+c_1')}.
\end{equation}
Since the expression on the left-hand side is strictly smaller than the expression on the right-hand side for every choice of $c_0'$ and $c_1'$, by picking $\epsilon'$ small enough, depending on $c_0$ and $c_1$, we can guarantee that $\Delta\in [c_0,c_1]$.

\begin{figure}
	\includegraphics[scale=2]{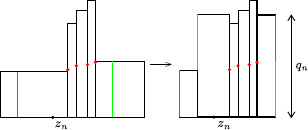}
	\caption{By choosing properly a shorter  subinterval as a base of the rectangle presentation, we obtain 
	a BT rigid rectangle presentation. The green and purple sides are identified. 
	\label{fig:BTrigid_cutpaste}}
	\end{figure}

\smallskip
The glueings of vertical sides of the rectangles shown in (see 
Figure~\ref{fig:BTrigid_cutpaste}) are such that the left side of the 
rectangle over $\tilde I_1$ is identified (by a translation) to the top right 
side of the last rectangle, over $\tilde I_{d+1}$. Moreover, by construction (by the inequalities on $\tau$ given by being in 
$\Omega_{\rho'}$), the points of discontinuity 
of the vertical glueings between the rectangles $\tilde R_2, \dots,\tilde 
R_{d+1}$ have vertical 
distance proportional  $q\rho'$ (with a proportionality constant which depends on the ratios of heights, which are by construction all  comparable. This shows that, choosing 
$\epsilon'<\epsilon/4$ sufficiently small so that \eqref{eq: precise c's} is satisfied and  $\rho'=\rho/4$, 
 one can guarantee that 
for any triples $(\lambda, \pi, \tau)$ that visit $\mathcal{U}'(C, 
 \epsilon', \rho' )$,
 the \emph{areas} (which are proportional to the widths) of $\tilde 
R _3, \cdots, \tilde R_d$ are less then $\epsilon$ and the vertical distance 
between discontinuities is less then $\rho q$.   
This concludes the proof that $X$ admits a  $(C,\epsilon,\rho)$-BT rigid rectangle representation in the sense of 
 Definition~\ref{def:BTrectangles} and thus that we can take $\mathcal{U}(C,\epsilon,\rho)$ to be equal to this $\mathcal{U}'(C',  \epsilon', \rho' )$. 
\end{proof}

%

The proof of Proposition~\ref{prop:Edegenerations} (namely the existence of BT rigid rectangle presentations)  follows immediately by this Lemma and ergodicity of Rauzy-Veech induction. 
\begin{proof}[Proof of Proposition~\ref{prop:Edegenerations}] 
{Consider the set $\mathcal{U}:=\mathcal{U}_{C,\epsilon,\rho}$ given 
by Lemma~\ref{lemma_openset}.  
In view of the ergodicity of $\hat{\mathcal{V}}$, for almost every 
$(\pi,\lambda,\tau)\in \hat{X}_d$,} there exists an increasing sequence 
$\{n_k\}_{k=0}^{\infty}$ s.t.
\begin{equation}\label{eq: returnset}
		\mathcal R^{n_k} (\pi,\lambda,\tau)\in \mathcal{U} \qquad 
		\forall k\in \N.
\end{equation}
Let us show that this implies that the surface $X$ given by 
$(\pi,\lambda,\tau)$ has an increasing sequence $(q_n)_n$ of $(C,\epsilon,\rho)$-BT rigid times. 
 To see this, for every $k\in\N$ consider the $n_k^{th}$ inducing subinterval $I^{(n_k)}$ given by Rauzy-Veech induction starting from $T=(\lambda, \pi)$ and let
$\lambda^{(n_k)}$ be the length vector of $\mathcal{V}^{n_k}(T)$. Notice that the lengths of the induced IET on  $I^{(n_k)}$ are
$|I^{(n_k)}| =\lambda^{(n_k)}_i$, for $1\leq i\leq d$.  
  Let us denote by $\tilde T_k$ the induced IET on  the interval  
 $$\tilde I^k:=\left[\tfrac{1}{8}|I^{(n_k)}|, \tfrac{9}{8}|I^{(n_k)}|-\la^{(n_k)}_{5}\right)\subset I^{(n_k)}.$$ 
We denote the interval exchanged by $\tilde T_k$  by $\tilde I^k_i$, for $1\leq 
i\leq {d+1}$. 
   Then, as in the proof of Lemma~\ref{lemma_openset} (see  
   \eqref{eq:lengths}), $\tilde T_k$ is an IET with a standard permutation 
   $\tilde{\pi}$ and length vector  
$$
\tilde{\lambda}^k:= |I^{(n_k)}|
\left [   \lambda^{(n_k)}_{1}-\lambda^{(n_k)}_{d} ,\
  \lambda^{(n_k)}_{d}-\tfrac{1}{8},\
  \lambda^{(n_k)}_{2},\
 \ldots,\  \lambda^{(n_k)}_{d-1},\
\tfrac{1}{8}\right).
$$
From Lemma~\ref{lemma_openset}, since $(\lambda^{(n_k)}, \pi^{(n_k)}, 
\tau^{(n_k)})\in \mathcal{U}$, the rectangles $\tilde R^k_i$, 
$1\leq i\leq d+1$ over the intervals $\tilde I_i^k$ give a $(C,\epsilon,\rho)$-BT 
rigid presentation. {Finally, since by construction $|\tilde I^k|\to 0$ as $k\to \infty$, the heights of the rectangles of $\mathcal{R}_k$ , which are all comparable (see the proof of Lemma \ref{lemma_openset}) grow to infinity as $k$ grows. }
\end{proof}

\subsubsection{Typical existence of BT-rigid times.}\label{sec:Edegenerations_proof}
{Let us now deduce Corollary~\ref{cor:E_BTrigidity} from  Proposition~\ref{prop:Edegenerations},  by showing that BT rigid rectangle presentations provide BT-rigidity times.}
\begin{proof}[Proof of Corollary~\ref{cor:E_BTrigidity}] 
Given $C\subset \mathbb{R}_+$ and $\epsilon>0$, choose $\rho:=\epsilon/3d$. 
By Proposition~\ref{prop:Edegenerations}, almost every $S$ in  $\mathcal{H}(1^{2g-2})$ 
admits a sequence of 
$(C,\epsilon/3,\rho)$-BT rigid rectangle presentations $\{ 
\tilde{R}^k_1,\cdots, \tilde{R}^k_{d+1}\}$, $k\in \mathbb{N}$. 
 Set 
\begin{equation}\label{eq: defrho}
q_k:=q^{(n_k)}_{1}+q^{(n_k)}_{d}, \qquad 	\Delta_k:=q_k|\tilde I_1^k|,
\end{equation} 
so that $q_k$ is the common height of $\tilde R^k_2$ and $\tilde R^k_{d+1}$ 
(since $\tilde R^k_2$ and $\tilde R^k_{d+1}$ come from zippered rectangles over 
$I^k_1$ and $I^k_d$ respectively)  and 
$\Delta_k$  the relative displacement of the continuity intervals  $\tilde 
I^k_2$ and $\tilde I^k_{d+1}$ of the IET $\tilde T_k$.

 {We will show that for each $k\in \mathbb{N}, $ $q_k$ is a $(C,\epsilon)$-rigidity time} (in the sense of Definition~\ref{def:BTrigid}).   

Consider a set of initial points $x$ which belong to the subset $S^k:= (\tilde 
R^k_1\cup 
\tilde R^k_2\cup \tilde R^k_{d+1})\setminus(\varphi_{-q}(\tilde 
R^k_3\cup\ldots\cup\tilde 
R^k_{d}))$. Notice that $S^k$, after glueing the 
boundaries of the rectangles according to the identifications, is (essentially) 
topologically 
a torus with a slit. 
In order to have rigidity, we need also to remove a \emph{bad zone}, that we call $B_k$ and is defined as follows. 
For every $k\in\N$, let $b_k := \tilde\lambda^k_1 + \cdots +  
\tilde\lambda^k_d  $ be the 
right endpoint of $\tilde I^k_d$ and let 
$\underline{\xi}_k$ and $ \overline{\xi}_k$ be respectively the smallest and 
largest heights of the discontinuities on the boundaries of $ \tilde R_2, 
\ldots , \tilde R_{d+1}$, namely
$$
\underline{\xi}_k:= \min \{  \tau_1^{(n_k)} , \tau_1^{(n_k)}+ \tau_2^{(n_k)}, \cdots,  
\sum_{i=1}^{d-1} \tau_i^{(n_k)} \}, \qquad   \overline{\xi}_k:= \max \{  
\tau_1^{(n_k)} , \tau_1^{(n_k)}+ \tau_2^{(n_k)}, \cdots,  
\sum_{i=1}^{d-1} \tau_i^{(n_k)} \}. 
$$
Define  $X_k:= S^k \backslash B_k $ (or equivalently, recalling the definition 
of $S^k$, $X_k= (\tilde R^k_1\cup \tilde R^k_2\cup \tilde R^k_{d+1}) \backslash 
B_k $), where  $B_k$ is defined to be the rectangular region (shaded in 
Figure~\ref{fig:BTrigidrectangles}) of width $\tilde \lambda^k_1$  next to the 
cramping of the singularities, given by 
\begin{equation}\label{eq:defBk}
B_k= \left[ b_k,  b_k + \tilde \lambda^k_1  \right] \times \left[ \underline{\xi}_k,   \overline{\xi}_k \right] . 
\end{equation}
Since  $|\overline{\xi}_k-\underline{\xi}_k|$ { is at most $d\rho q_k$, by 
choice of $\rho=\epsilon/d$, 
$ Area(B_k)= |\overline{\xi}_k-\underline{\xi}_k|\, \lambda^k_1\leq d\rho 
=\epsilon/3$. 
This, together with the fact that by property $(R1)$ of a 
$(C,\epsilon/3,\rho)$-rigid rectangle 
presentation (see Definition~\ref{def:BTrectangles}) $Area(R^k_3 \cup\ldots\cup 
\tilde R^k_d )<\epsilon/3$, 
we get that $Area(X_k)>1-\epsilon$. }

\smallskip
We now claim that $X_k$ is a $C$-rigidity set for the vertical flow $\varphi_\R=(\varphi_t)_{t\in\mathbb{R}}$ 
on  $S$. 
For all 
points $x\in X_k$ which have distance at least $\lambda_1^k$ from the right 
boundary of $\tilde R^k_d$, since $\Delta_k =q_k|\tilde I_1^k|\in C$, one has that
\begin{equation}\label{eq:BTshift}
{
\frac{c_0}{q_k}\leq  d(\varphi_{q_k}(x), x) = \big|\tilde {I^k_1}\big| \leq \frac{c_1}{q_k}.}
\end{equation}
using the BT dynamics of the vertical flow on $X_k$ and the vertical identifications, as in the Example of BT rigidity with $d=2$ rectangles in \S~\ref{sec:BTrectangles}.  For all points $x\in X_k$ close  $\lambda_1^k$ -close the right boundary of $\tilde R^k_d$, but \emph{below} the bad region $B_k$, \eqref{eq:BTshift} also holds trivially, since all adjacent rectangles are glued to each other below height $\underline{\xi}_k$. 
To locate $\varphi_{q_k}(x)$ for the initial points $x$ \emph{above} the region $B_k$,  
we note that due to the way $\tilde R^k$ was constructed, one has to exploit the vertical glueings between the top parts of the rectangles  $\tilde R^k_3,\ldots,\tilde R^k_d$. One can then see that $\varphi_{q_k}(x)$ has  an \emph{horizontal} displacement  from $x$ equal to $ \lambda^k_1\geq 1/8q_k$ (as in the other cases)  and no vertical displacement, 
  so that \eqref{eq:BTshift} still holds.
Thus, \eqref{eq:BTshift}  holds for every $x\in X_k$ and since we have already 
shown above that $Leb(S\setminus X_k)<\epsilon$, 
this  concludes the  proof that $(q_k)_k$ are all $(C,\epsilon)$-BT rigidity times for the 
vertical flow $\varphi^\R$ on $X$ in the sense of Definition~\ref{def:BTrigid}.
		\end{proof}

\subsection{Trimmed derivative linear bounds  using Rauzy-Veech induction}\label{sec:derivativesviaRV} 
To build good matching sets, we need (in addition to the BT rigid rectangle presentations produced in the previous \S~\ref{sec:Edegenerations}) tools to guarantee that the matching orbit has linear bounds on trimmed derivatives (see condition $(G2)$ of Definition~\ref{def:goodMn}).  
In \cite{Ul:abs} (see also \cite{FKU})  the last author, exploiting  Rauzy-Veech induction, builds many orbit segments for which linear bounds on trimmed derivative bounds hold (see Definition~\ref{def:bounded}). These good orbit segments correspond to 
 Rohlin towers given by (special times of) Rauzy-Veech induction, in the following sense.

\subsubsection{Linear bounds on Rohlin towers given by RV induction.}
Let us say that an orbit segment $\mathcal{O}_T(x,q)$ \emph{goes once up} a Rohlin tower $\mathcal{T}$ for $T$ if $x\in J$ where $J$ is the base of $\mathcal{T}$ and $q$ is the height of $\mathcal{T}$. 
\begin{defn}\label{def: derivativebound}
Given $f\in Sym Log^2 (T)$, we say that a Rohlin tower $\mathcal{T}$ with base $J$ and height $q$ \emph{gives} $B$-\emph{linear bounds on trimmed derivatives} (for $f$) for some $B>0$ if every orbit segment which goes once up $\mathcal{T}$ satisfies $B$-linear bounds on trimmed derivatives, i.e.
\[
\left|\widetilde{S}_r f'(x)\right| \leq B r 
 \ \  \text{for \ any}   \ 0\leq r \leq q\ \text{and \ any}\ \ x\in J .
 \]
\end{defn}

 \smallskip
 Given a Keane IET, if $\{ I^{(n)}, \ {n\in\N}\}$ 
 is the sequence of inducing intervals given by Rauzy-Veech induction, we denote by $( \mathcal{T}^{(n)}_j)_{n\in\N}$, $1\leq j\leq d$ the Rohlin towers  over the continuity intervals $I^{(n)}_j$ given  by
$$
\mathcal{T}^{(n)}_j:= \left\{  I^{(n)}_j,T(I^{(n)}_j), \dots, T^{h^{(n)}_j-1}I^{(n)}_j\right\}, \qquad 1\leq j\leq d,
$$
where $h^{(n)}_j$ is the first return time of  $I^{(n)}_j$ to $I^{(n)}$. We will refer to these Rohlin towers are \emph{Rohlin towers given by Rauzy-Veech induction} at time $n$.

\smallskip
The following Proposition is a technical
version of a result  proved by the last author in  \cite{Ul:abs} (refined\footnote{While the content of Proposition is essentially the central result in the work \cite{Ul:abs} of the second author,  the precise form of the accelerating set, and in particular its product structure (which is important here, as well as in \cite{FKU}, to combine these estimates with rigidity) was not stated in the proof in \cite{Ul:abs}. The adaptations to the proof in \cite{Ul:abs} that allow to show that the accelerating set has the desired product form appear in the Appendix $C$ in \cite{FKU}. }
 Rohlin towers given by Rauzy-Veech induction at times which corresponds to visits to the set $E$  give linear bounds on trimmed derivatives for any function $f$ in $SymLog^2 (T)$.  

\begin{prop}[Acceleration for bounded  trimmed derivative sums, see Proposition 8.1 in \cite{FKU}]\label{prop:accelerationset}
For any $d\geq 2$, any irreducible permutation $\pi$  and any neat compact set  $K=\{\pi\}\times \Delta_{A}\times \Theta_{A} $ (see \S~\ref{sec:loss}), there exists  $Y 
\subset \Theta_{A}$  such that the  product set $E$ given by 
$$
E
:= \{ \pi\}\times \Delta_{A} \times Y = \{ (\pi,\lambda, \tau)\in \hat{X}\, | \,\, \lambda\in \Delta_{A}, \, \tau \in Y \}
$$
has  $\hat{\mu}_{\hat{\mathcal{}}}(E)>0$, and 
 {given any $f$ in $\mathcal{P} SymLog (T)$
there exists   $B_0>0$}  
so that, if
 $(\pi,\lambda,\tau)$ 
 is recurrent to $E$
  under $\widehat{\Z}$ and $(n_\ell)_{\ell\in \N}$ is the sequence of  visits, i.e.~$$\widehat{\mathcal{Z}}^{n_\ell}(\pi,\lambda,\tau)\in E \quad \textrm{for\ every}\ \ell\in \N,$$ then for any  
  all the Rohlin towers  $( \mathcal{T}^{(n_\ell)}_j)_{\ell\in\N}$, $1\leq j\leq d$ over the inducing intervals $I^{(n_\ell)}_j$, $1\leq j\leq d$ give $B_0$-linear bounds on trimmed derivatives.
 \end{prop}

 \subsubsection{Persistence of linear bounds on large towers.}
We will need also the following Lemma, that shows that \emph{large} Rohlin towers obtained by cutting and stacking a bounded number of Rohlin towers which give linear bounds on trimmed derivatives, also give linear bounds on trimmed derivatives (up to changing the constant in the linear bound).

\begin{lemma}[Persistence of linear bounds, Lemma 8.4 in \cite{FKU}.]\label{lemma:finitecombination}
Given $T$ Keane and $f\in SymLog^2(T)$, assume that the sequence $(n_k)_k$ is  such that the Rohlin towers
 $\mathcal{T}^{(n_k)}_j$,  $1\leq j\leq d$,  given by Rauzy-Veech induction at times  in the  subsequence  $(n_k)_k$  give $\Cb_0$-linear bounds on trimmed  derivatives for $f$. 
 
 Assume furthermore that $(m_k)_k$ is obtained by \emph{shifting} $(n_k)_k$, i.e.~that there exists an integer $N\in \N$ such that 
 $$m_k := n_k+N, \qquad \text{ for\ every}\  k\in \N.$$ 
\noindent Then if  $(\mathcal{T}^k)_{k\in\N}$ is a sequence of towers of the form $\mathcal{T}^k:= \mathcal{T}^{(m_k)}_{j_k}$, where  $j_k\in\{1,\ldots,d\}$  are such that there exists $\nu>0$  such that the tower $\mathcal{T}^{(m_k)}_{j_k}$ has measure
$$\left|\mathcal{T}^{(m_k)}_{j_k}\right|= \lambda^{(m_k)}_{j_k} h^{(m_k)}_{j_k}\geq \nu  >0 \quad\text{for all}\quad k\in\N,$$
then
{$f$ has $B_1$-linear bounds on trimmed derivatives along $(\mathcal{T}^k)_{k\in\N}$ for some constant  $B_1>0$ which depends only on $B_0, N$ and $\nu$.} 

 \end{lemma}
\noindent The proof of this Lemma (which is a simple exercise of decomposition of Birkhoff sums) can be found\footnote{The statement of the Lemma is a specical case of Lemma 8.4, since, as remarked in \cite{FKU} just after the statement,  the first assumption (A1) of Lemma 8.4 holds automatically  if
 the product of the cocycle matrices
  between time $n_k$ and $m_k$ has uniformly bounded norm independently on $k$, as assumed in Lemma~\ref{lemma:finitecombination} here.}  in \S~8.2.1 in \cite{FKU}.

\subsection{Proof of existence of good matchings}\label{sec:Egoodmatchingsproof}
{In this subsection we prove  Proposition~\ref{prop:Egoodmatchings}, finding times where BT rigid Rohlin towers (coming from a BT rigid rectangle presentations) give linear bounds on trimmed derivatives.  The idea is again to produce this times using Rauzy-Veech induction, by considering RV times  $(n_\ell)_\ell$ which correspond to visits  to the product set $E$ given by Proposition~\ref{prop:accelerationset}, followed
by a visit at time $m_\ell = n_\ell+N $  (after a finite number $N$ of iterates)   to the open set  (given by Lemma~\ref{lemma_openset}) which gives the desired $(C,\epsilon,\rho)$-BT rigid rectangle presentations. The persistence of the trimmed derivative bounds (given by Lemma~\ref{lemma:finitecombination}) is used to show that linear bounds  on trimmed derivatives (which hold by construction at each time $n_\ell$) still hold at the shifted times  
 $(m_\ell)_\ell$ for the matching orbit.

\begin{proof}[Proof of Proposition~\ref{prop:Egoodmatchings}] 
Notice first  that 
 it suffices to prove that for any given $L\in \N$ there is a full measure set of IETs with 
a standard permutation $\pi$, such that the conclusion is true,  since 
 then we can intersect over $L \in \N$ and still get a full measure set. {Also,  choosing sequence  $\rho_n\to 0$, it suffices to prove that there is a full measure set of IETs for which the conclusion holds for a given $\rho\in (\rho_n)_n$ (since a $(C,\epsilon,\rho_n)$-rigid presentation is also a $(C,\epsilon,\rho)$ rigid presentation for every $\rho>\rho_n)$. 
  Let us hence fix $L\in \mathbb{N}$ and $0<\rho<1$.  }

\smallskip
 We now  split the proof in steps for clarity. 
 
 {
 \smallskip
\noindent{\it Step 1 (Definition of  the return set)}:   
Let us define a suitable neat compact set $K=\{\pi\} \times \Delta_{A} \times \Theta_{A}$  (see \S~\ref{sec:loss} for definitions and notation), to which we will apply Proposition~\ref{prop:accelerationset} to get an acceleration set for $B_0$-linear bounds on trimmed derivatives. Let  $ \Omega_{\rho'}\subset \Theta_\pi$ be the set given by \eqref{eq:defOmega'} where $\rho'=\frac{\rho}{4}$.
Let $A^-:=A_{-m}\cdots A_{-1}$ where  $(A_i)_{i\in \mathbb{Z}}$ is the sequence associated as in  \eqref{eq:seqA} to a Birkhoff generic $(\pi, \lambda,\tau)$ with   $\tau\in \Omega_{\rho'}$, and 
$m$ is chosen sufficiently large so that $
 \Theta_{A^-}\subset \Omega_{\rho'}\subset \Theta_\pi
 $
(this is possible since $\Omega_{\rho'}$ is open and for a Birkhoff generic  $\tau\in  \Omega_{\rho'}$,  the cones $\Theta_{A_{-m}\cdots A_{-1}}$ have diameter going to zero as $m$ grows, see \cite{MUY} or \cite{Be:back}). Then set $A=A_{-n}\cdots A_{-1}$ where $n$ is chosen so that $n>m$, so that $A>0$ can be written as $A=A' A^-$ (taking $A:=A_{-n}\cdots A_{-m-1}$), it is neat\footnote{This can be achieved for example by adding a block of length more than $n/2$ of top (or bottom) Rauzy-Veech moves, see \cite{AGY}.} and $\pi^{(-n)}=\pi$.
Consider now the (neat) compact set $\{\pi\}\times \Delta_{A}\times \Theta_{A}$ and let $Y\subset \Theta_{A}$ and $B_0>0$ be the set and the constant given by  Proposition~\ref{prop:accelerationset}. 

For any $\epsilon>0$ and $C=[c_0,c_1]\subset \mathbb{R}_+$, consider now the open set $\mathcal{U}(C,\epsilon, \rho)$ given by (the proof of)  Lemma~\ref{lemma_openset}, which has the product form $\mathcal{U}=\Omega_{\rho'}\times \Lambda'$, where $\Lambda':= \Lambda_{c_0',c_1',\epsilon'}$ is given by  \eqref{eq:defLambda'} and $c_0', c_1', \epsilon' $ are chosen as in the proof of Lemma~\ref{lemma_openset}. We now claim that the set $F=F(\rho,C,\epsilon)\subset \hat{X}_d$ given by 
\begin{equation}\label{def:Ek}
F=F(\rho,C,\epsilon):=E \cap \mathcal{Z}^{-n} (\mathcal{U}(C,\epsilon,\rho)), \qquad \text{where}\ E:={ \{ \pi \} \times \Delta_{A}\times Y },
\end{equation}
has $\mu_{\widehat{Z}}(F)>0$. The proof of this will take the rest of this Step. Let us first consider 
the set $
{( \{ \pi \} \times \Delta_{A}\times \Theta_A )} \cap \mathcal{Z}^{-n} (\mathcal{U}(C,\epsilon,\rho))$, which contains $F$ and is by construction open, and show that it is  non-empty.
To see this, consider any $(\pi,\lambda, \tau)$, where  $\tau\in \Theta_\pi$ and $\lambda$ is  such there $\lambda := A \lambda'/ \Vert A\lambda'\Vert$ for some $\lambda'\in  \Lambda'$.  Then, by construction  
$(\pi,\lambda, \tau)\in E=\{\pi\}\times \Delta_{A}\times Y$ (since $A\lambda'/ \Vert A\lambda'\Vert \in \Delta_{A}$ by definition \eqref{eq:defLambda'} of $\Delta_A$). Moreover, since $A$ was defined so that $A=A' A^-$, by \eqref{bisided_Markov} we have that
$$
\mathcal{V}^n(\pi, \lambda,\tau)\in \mathcal{V}^n( \{\pi\}\times \Delta_{A' A^-}\times \Theta_{\pi})= \{\pi\}\times \Delta_{d}\times \Theta_{A^-A'}\subset \{\pi\}\times \Delta_{d}\times \Omega',
$$
where the last inequality follows since
$\Theta_{A'A^-}\subset \Theta_{A^-}$ by \eqref{eq:inclusion} and 
 $A^-$ was defined so that $ \Theta_{A^-}\subset  \Omega'$.  
Since by definition the $\lambda$-coordinate of $\mathcal{V}^n(\pi, \lambda,\tau)$ belongs to $  \Lambda'$,  we have that $(\pi, \lambda,\tau)\in E$ also belongs to $\mathcal{V}^{-n}(\Lambda'\times \Omega')=\mathcal{V}^{-n}(\mathcal{U}(C,\epsilon,\rho))$, and hence that $(\pi, \lambda,\tau)$ belongs to $F$.  
Remark that $\mu_{\widehat{\mathcal{Z}}}$ gives positive measure to open non-empty sets. Since $F$ is contained in this open, non-empty set, and the Lebesgue measure of $Y$ is positive, it follows that $\mu_{\widehat{\mathcal{Z}}}(F)>0$.

\smallskip
\noindent {\it Step 2 (Definition of typical IETs and of  return times):}
Let us now define the set of full measure IETs which will satisfy the conclusion. By ergodicity 
  of $\widehat{\mathcal{V}}$, there exists a a set  $\widehat{X}^\ast$ such that $\widehat{\mu}_{\mathcal{Z}} (\hat{X}\backslash \widehat{X}^\ast) =0$,
such that any $(\pi,\lambda,\tau)\in \widehat{X}^\ast$ is Birkhoff generic, and thus, in particular 
its orbit under $\widehat{\mathcal{V}}$ visit any open, non empty set in $\widehat{X}^\ast$. Since the invariant measure $\mu_\mathcal{V}$ on the IET space $X$ is given by the pushforward  $\mu_\mathcal{Z} = p_* \widehat{\mu }_\mathcal{Z}$ and is equivalent to the Lebesgue measure on $X$, it follows that there exists a full measure set $X^\ast$ of IET $T=(\pi, \lambda)$ that have a \emph{lift}, i.e.~a triple $(\pi, \lambda, \tau)$ in the fiber $p^{-1}(\pi,\lambda)$ of the projection map $p:\widehat{X}\to X$, which belongs to  $\widehat{X}^\ast$.

For any non-empty $C=[c_0,c_1]\subset \mathbb{R}_+$ and $0<\epsilon<1$, since the  set $F=F(C,\epsilon,\rho)$  defined as in \eqref{def:Ek}   has $\widehat{\mu}_{\V}(F)>0$ by Step 1, 
for any  $(\pi,\lambda)\in X^\ast$, there exists a lift  $(\pi,\lambda,\tau)\in \widehat{X}^\ast$} and
 an increasing sequence $(n_k)_{k\in\N}$ such that
\begin{equation}\label{eq:returns}
\widehat{\mathcal{V}}^{n_k} (\pi,\lambda,\tau) \in F(C,\epsilon,\rho) := E\cap \hat{\mathcal{V}}^{-N}(\mathcal{U}(C,\epsilon,\rho)), \qquad \forall \ k \in \N.
\end{equation}
Moreover, setting $(m_k)_{k}$ to be the \emph{shifted sequence} given by $m_k:=n_k+n$ for each $k\in \mathbb{N}$, by definition \eqref{def:Ek} of $F$ as preimage, we also have 
\begin{equation}\label{eq:returns2}
\widehat{\mathcal{V}}^{m_k} (\pi,\lambda,\tau) = \hat{\mathcal{V}}^{n}(\widehat{\mathcal{V}}^{m_k} (\pi,\lambda,\tau))\in \mathcal{U}(C,\epsilon,\rho), \qquad \forall \ k \in \N.
\end{equation}

\smallskip
\noindent {\it Step 3 (Shifted return times give BT rigid tower presentations)}: 
We will now use the  sequence  $(m_k)_{k}$ of induction times defined in Step 2 
  to build the desired sequence of good matchings $(M_n)_n$.
it gives a sequence of $(C, \epsilon, \rho)$-BT rigid (Rohlin) towers presentations for the IET $T=(\pi,\lambda)$ in the sense of Definition~\ref{def:BTtowers}.
For any $k\in \mathbb{N}$, since by \eqref{eq:returns2} $m_k$ is a visit to the set {$\mathcal{U}(C,\epsilon,\rho)$}, by Lemma~\ref{lemma_openset}, the translation surface $S$ associated to the lift $(\pi,\lambda, \tau)$ of $T$ chosen in Step 2 
 has a $(C,\epsilon, \rho)$-BT rigid rectangle presentation 
 $ \mathcal{R}_k:=\mathcal{R}_k(S):= \{ \tilde R^k_1,\, \cdots,  \tilde R^k_{d+1} \,; \, \sim_{\mathcal{R}_k} \}$, $ k\in \mathbb{N}$ 
  (see also the proof of Proposition~\ref{prop:Edegenerations}). 

This sequence $( \mathcal{R}_k )_k$ of  BT rigid \emph{rectangles} presentations for $S$ produce a sequence of BT rigid \emph{towers} presentations 
 $ \mathcal{T}^k(T):= \{ \mathcal{T}^k_1, \cdots,  \mathcal{T}^k_{d+1} \}$, $k\in \mathbb{N}$, for $T$,  
through the relation between rectangles and towers explained in \S~\ref{sec:rectangles_vs_towers}.  
For $1\leq i\leq d+1$, the base of the Rohlin tower $ \mathcal{T}^k_1$ is the base $\tilde I^k_i$ of the rectangle $\tilde R^k_i$. We will denote by $\tilde q^k_i$ the (integer) height of the tower $ \mathcal{T}^k_i$ and set
$q_k:=\tilde q^k_2=\tilde q^k_{d+1}$ to be the common height of $\mathcal{T}^k_2$ and $\mathcal{T}^k_{d+1}$.   
{Let $\Delta_k:=q_k|\tilde I_1^k|$. By Lemma~\ref{lemma_openset}, we  have that $c_0<\Delta_k<c_1$, so this gives that $q_k$ is a $C$-BT rigidity time with $C=[c_0,c_1]$.  }

The \emph{cramping} of the singularities of the rectangle presention $\mathcal{R}_k(S)$ (condition $(iv)$ of Definition~\ref{def:BTrectangles}) induce \emph{cramping} of 
 discontinuities of $T$ (in the sense of $(v)$ of Definition~\ref{def:BTtowers}); more precisely, {if $0\le \underline{\ell}_k<\overline{\ell}_k'<\tilde q_k$ are such that $\underline{\ell}_k$ is the first iteration of $\tilde I^k$ via $T$ that sees a discontinuity (or in other words, this $\underline{\ell}_k$ is the height of the lowest discontiinuity) 
 and $\overline{\ell}_k$ is the last one before first return (or in other words, this $\underline{\ell}_k$ is the height of the lowest discontinuity), then 
 \begin{equation}\label{eq: tausareclose}
 \overline{\ell}_k-\underline{\ell}_k\le \rho \cdot  q_k.
 \end{equation} 

{Finally, we claim that for any point $x\in I\setminus X_n$, 
and every $\beta\in Disc(T)$ if $x_{\beta}^+$(respectively $x_{\beta}^-$) is the closest visit to $\beta$ from the right (resp.~left) in time $q$, and $(a^+,b^+)$ (respectively $(a^-,b^-)$) is the level of $\mathcal T_2$ (resp. $\mathcal T_{d+1}$) which contains $x_{\beta}^+(x_{\beta}^-)$, then
\begin{equation}\label{eq: distfromdisc}
\max\{\big||m^+_{\beta}(x,q_k)-\beta|-|m^+_{\beta}(x,q_k)-a^{+}|\big|, \big||m^-_{\beta}(x,q_k)-\beta|-|m^-_{\beta}(x,q_k)-b^{-}|\big| \}\le |\tilde I^k_3|+\ldots+|\tilde I^k_d|.
\end{equation}
%
This follows from 
by the definition of $\Omega_{\rho_0}$ given in the proof of Lemma~\ref{lemma_openset}, since by considering the vertical orbits starting at the base of the zippered rectangles $R_1$, $R_d$ as the return sets for the horizontal translation flow, we get that each horizontal orbit of such points crosses the zippered rectangles $\tilde R_3,\ldots,\tilde R_d$ at most once (for a detailed proof of this claim, we refer the reader to Proposition 4.1 in \cite{BFT}).  In view of the bound on $ |\tilde I^k_i|$ for $i=3,\ldots,d$ and \eqref{eq: distfromdisc}, the final condition follows}.

Thus, this concludes the proof that the \emph{towers} presentations 
 $ \mathcal{T}^k(T):= \{ \mathcal{T}^k_1, \cdots,  \mathcal{T}^k_{d+1} \}$ are $(C,\epsilon, \rho)$-BT rigid presentations (according to Definition~\ref{def:BTtowers}) for any $k\in \mathbb{N}$. 
}

\smallskip

\noindent {\it Step 4. (Choice of reference orbit)}: 
 {We now choose for every $k\in\N$ 
a point $z_k\in I$ which will serve as a reference point when constructing the matching orbit segments. We take for example\footnote{{It is worth mentioning, that the point $z_k$ can be chosen among many other points and, in particular, does not have to be in the base of a Rohlin tower. In the end, we are proving that the matching set is \emph{large}, and the matching property is \emph{transitive}. Hence every point from the final matching set would work as well.}} $z_k$ to be the mid-point of the interval $\tilde I_2^k$, that is the base of the tower $\mathcal T_2$.} 
 Let  $q_k$,  $k\in \N$, be the common height of the 'large' BT rigid towers $\mathcal{T}^k_2$ and $\mathcal{T}^k_{d+1}$ constructed in Step 2.
The orbit $\mathcal{O}_T(z_k,q_k)$ of lenght $q_k$ of the reference point $z_k$ will be our reference orbit for matchings.

\smallskip

\noindent {\it Step 5. (Large towers at shifted return times give linear trimmed derivative bounds)}: Consider any $f\in\pSymLog^2 (T)$. 
Since by its definition  {the recurrence set  $F\subset E$ (see \eqref{def:BKL} in Step 1)}
 and by Step 2 $(n_{k})_{k\in\N}$ is a subsequence of the returns for $(\pi,\lambda,\tau)$ to $E$, by Proposition \ref{prop:accelerationset}, there exists $B_0$ such that all the Rohlin towers $\mathcal{T}^{(n_k)}_{j}$ with $1\leq j\leq d+1$, $k\in \N$ give  
 $B_0$-linear bounds on trimmed derivatives for $f$.  

Consider now the shifted return sequence $(m_k)_k$ and let $n$ be the integer such that $m_k=n_k+n$ for each $k\in \N$. Consider the 'large'  Rohlin towers  $\mathcal{T}^k_2$ and $\mathcal{T}^k_{d+1}$ 
in the Rohlin tower rigid representations $\mathcal{T}^k$ build in Step 3 
and remark that $Area(\mathcal{T}^k_2)\geq 1/8$ and $Area(\mathcal{T}^k_{d+1})\geq 
1/8$ for every $k $ sufficiently large. Thus, by the persistence of linear 
bounds  Lemma~\ref{lemma:finitecombination}, there exists a constant $B_1>0$, 
depending only on $B_0$, $n$ and $\nu=1/8$, such that the Rohlin towers  
$(\mathcal{T}^k_2)_k$ and $(\mathcal{T}^k_{d+1})_k$ all give $B_1$-linear bounds on 
trimmed derivatives for $f$. In particular, since any subsegment 
$\mathcal{O}(z_k,r)$ of the orbit  $\mathcal{O}(z_k,q_k) $ crosses at most once 
the Rohlin towers over $\mathcal{T}^k_2$ and  $\mathcal{T}^k_{d+1}$ ({since $z_k$ was chosen in Step 4 to be the midpoint of $\tilde I_2^k$}), and  $q_k=\tilde q_k^2=\tilde q_k^{d+1}$,  
we have that, for any $r\leq q_k$, 
$$
|S_r f' (z_k)|\leq |S_r f' (z_k)| \leq B_1 \tilde q_k  \text{ if}\ 0\leq r< q_k,
$$
so that $\mathcal{O}(z_k,q_k) $ also satisfies $B_1$-linear bounds on trimmed Birkhoff sums for every $k\in \N$.

\smallskip
\noindent {\it Step 6. (Definining matching sets and matchings)}: 
 We will now define, for each $k\in \N$, a subset $M_k:=M_k^L\subset \mathcal{T}^k_2\cup \mathcal{T}^k_{d+1}$ and show that it is a $L$-fold $(z_k,q_k)$-matching set.  We set  $M_k:= [0,1]\backslash X_k$ where $X_k=X_k^L$ is a region 
which we throw away since we cannot define matchings either already for the 
orbits of length $q_k$, or in order to have an $L$-fold matching (see 
Definition~\ref{def:Lmatchingset}):
\begin{equation}\label{eq:defXk}
X_k=X_k^L:=\left(\bigcup_{i\in\{3,\ldots,d\}}\bigcup_{j=-Lq_k}^{q^k_i-1}\tilde I^k_{i}\right)\ \cup\  \bigcup_{i=0}^{L-1} \left(\bigcup_{j=\underline{\ell}_k-iq_k}^{\overline{\ell}_k-iq_k} T^j[\partial\tilde I^k_{d+1}, \partial\tilde I^k_{d+1}+|I^k_1|]\right).
\end{equation}
Notice that the first part of the union in \eqref{eq:defXk} contains the 'small' towers $\tilde {\tau}_i^k$ over $\tilde I^k_i$ for $i=3,4,5$, as well as the points that enter them in time at most $ L q_k $; the second part of the union in  \eqref{eq:defXk} contains the intervals which belong to the \emph{bad zone} $B_k$ (which have width $|I_1^k|$ and endpoints at height between $\underline{\ell}_k$ and $\overline{\ell}_k$, which are the smallest and largest of the discontinuities heights) where one does not control BT rigidity (see the definition \eqref{eq:defBk} of $B_k$  in the proof of Proposition~\ref{prop:Edegenerations} in \S~\ref{sec:Egoodmatchingsproof}), as well points whose orbit enter the bad zone after a multiple $l q_k$ with $1\leq l< L$ (since throwing these points is needed to control $L$ fold matchings). {In particular, if $X_k^1$ denotes the set \eqref{eq:defXk}, $X_k^1$ is union of the small towers  $\tilde {\mathcal{T}}_3^k\cup \ldots\cup \tilde {\mathcal{T}}_d^k$ and the bad zone $B_k$, we have that
\begin{equation}\label{eq:Ltimesnotbad}
\mathcal{O}(T^{j q_k}x, q_k)\cap X_k^1=\emptyset, \quad \text{for}\ j=0,\dots, L-1.
\end{equation}}
We now illustrate that for every $k\in\N$, 
the set $M_k:=I\setminus X_k$ is a {$L$-fold } $(z_k,q_k)$-matching set.  We note that due to the way $\tilde I^k$ was defined, the left-endpoints of the floors of the towers $\mathcal{T}^k_1$ and $\mathcal{T}^k_2$  over $\tilde I^k_1$ and $\tilde I^k_2$ do not contain any 
discontinuity of $T$. Moreover, since $\mathcal{T}^k_1$, $\mathcal{T}^k_2$ and $\mathcal{T}^k_{d+1}$ 
were cut from the Rohlin towers over the 'large' subintervals of  $I^{(n_k)}$ (see Figure~\ref{fig:BTrigid_cutpaste}), we have the following identifications: 
 \begin{align}\label{eq: gluing1}
		T^{i}(\tilde I^k_1) & \text{ is glued to the right-hand side of 
		}T^{q_k-q_k'+i}(\tilde I^k_{d+1})\text{ for }i=0,\ldots q_k'-1,
\\ \label{eq: gluing2}
	T^{q_k'+i}(\tilde I^k_2) & \text{ is glued to the right-hand side of 
	}T^{i}(\tilde I^k_{d+1})\text{ for }i=0,\ldots q_k-q_k'-1,
\end{align}
where $q_k'$ is the height of the tower $\mathcal T_1^k$.
A crucial consequence of these identifications is that one can verify that 
\begin{equation}\label{eq: partialrigidity}
	\max_{x\in M_k}|T^{q_k}(x)-x|<\frac{1}{q_k},
	\end{equation}	
which shows that	$\{q_k\}_{k\in\N}$ are BT rigidity times for points in $M_k$.

{From these observations, we claim that it follows that given any  $x\in I\setminus X_k$,  $\mathcal{O}_T(x,q_k)$ is matched to the reference orbit segment $\mathcal{O}_T(z_k,q_k)$, so that $M_k$ is a $\mathcal{O}_T(z_k,q_k)$-matching set.} The choice of the matching intervals depends on the initial 
position of $x$ in $M_k$ and exploits BT rigidity times and vertical identifications. We 
illustrate all possible matchings in Figure \ref{fig: matchings}.   There are six possible types of matchings, corresponding to six possible regions of $M_k$ in which $x$ can be located (these correspond to the six subpictures in  Figure \ref{fig: matchings}; in each the initial location of $x$ is in the shaded region in that pictures). In color in the figures of  Figure \ref{fig: matchings}, we illustrate the matching intervals which match $\mathcal{O}_T(x,q_k)$ to  $\mathcal{O}_T(z_k,q_k)$; 
 different colors in the figure correspond to  matching intervals of different lengths (so that in a $k$-matchings there are $k$ colors).   One can check that all $M_k$ are $\kappa$-$(z_k,q_k)$ matchings with $\kappa=3$ (since in all subfigures of Figure \ref{fig: matchings} at most $3$-colors, i.e.~three different lengths of matching intervals, are used). 

\begin{figure}
	\includegraphics[scale=0.7]{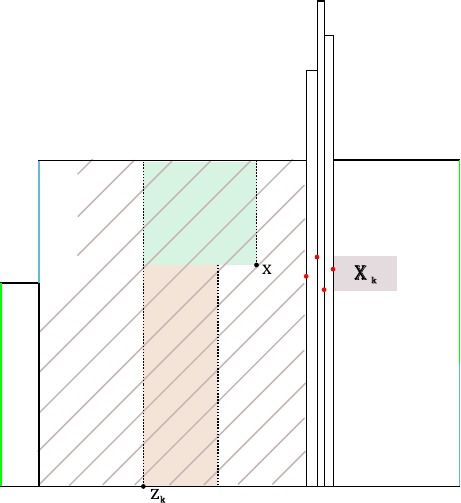}
	\hspace{13mm}
	\includegraphics[scale=0.7]{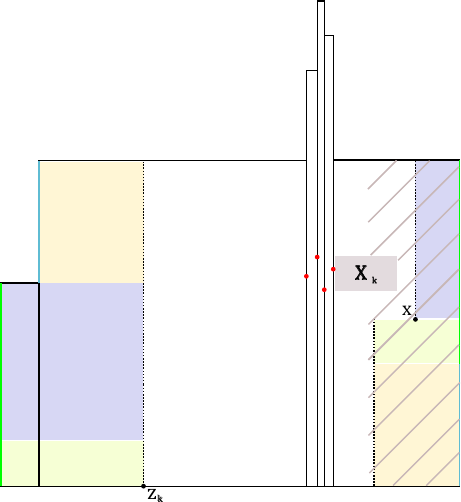}
	\vspace{3mm}
	
	\includegraphics[scale=0.7]{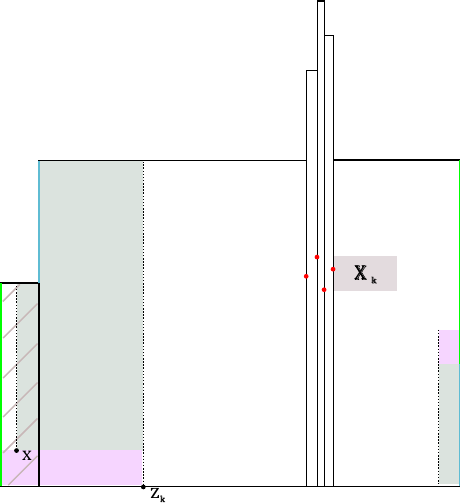}
	\hspace{13mm}
	\includegraphics[scale=0.7]{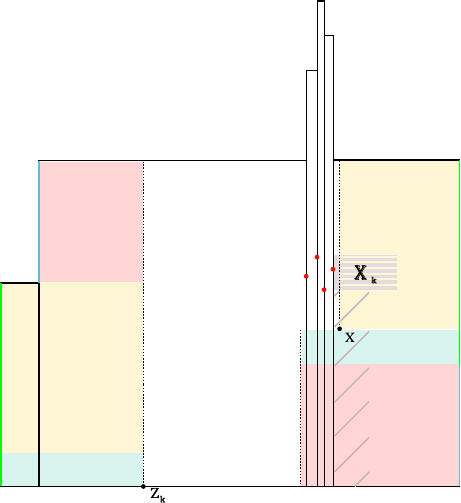}
	\vspace{3mm}
	
	\includegraphics[scale=0.7]{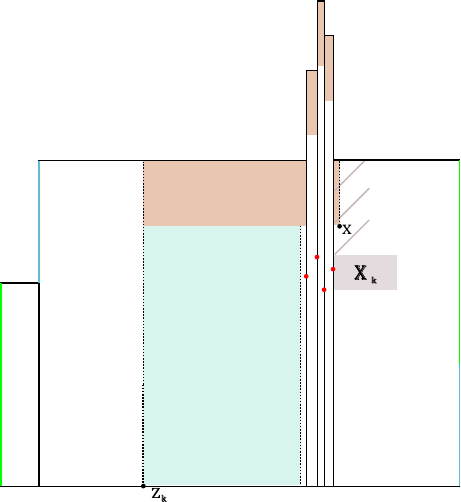}
	\hspace{13mm}
	\includegraphics[scale=0.7]{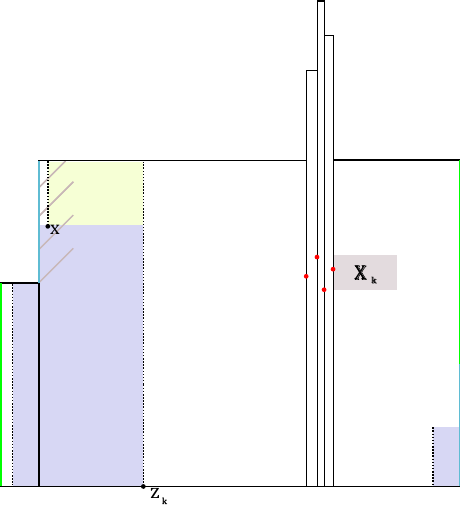}
	\caption{\label{fig: matchings} All six possible types of matchings
	 between the orbit $\mathcal{O}_T(x,q_k)$ of the point $x\in M_k$ and the reference orbit $\mathcal{O}_T(z_k,q_k)$. 
	The red points denote the discontinuities of $T$, while 
	the blue and green segments on the vertical sides are identified due to the sides rectangles (and hence Rohlin towers) being identified. In each subfigure, the shaded region denotes the region of possible locations of $x$ in which the  illustrated type of matching applies. 
Colored areas denote the location of matching intervals. Each color denote a set of matching intervals of the same length, thus showing that each matching is a $3$-matching ($3$ colors only).}
\end{figure}

As an explict example, we illustrate the first subfigure that is the case when $x\in  T^j(\tilde I^k_2)$ for some 
$j\in\{0,\ldots,q_k\}$. Under this assumption, we can write
\begin{align*}
\mathcal{O}_T(x,q_k)& =
\mathcal O_T(x,q_k-j)\sqcup \mathcal 
O_T(T^{q_k-(j-p_k)}x,j) , \\
\mathcal{O}_T(z_k,q_k)& =\mathcal O_T(z_k,j)\sqcup  O_T(T^{j}z_k,q_k-(j))),
\end{align*}
and match  the orbit segment $\mathcal O_T(x,q_k-j)$ with $\mathcal O_T(T^{j}z_k,q_k-j))$ (via the intervals {in the green area in the top left subfigure of}  Figure~\ref{fig: matchings}), while $\mathcal 
O_T(T^{q_k-j}x,j)$ is matched with $\mathcal O_T(z_k,j)$ (via the intervals in the brown colored region of the same  top left subfigure of  Figure~\ref{fig: matchings}),.

{To conclude the proof that  $M_k$, for each $k\in \N$, is an $L$-fold matching set (see Definition~\ref{def:Lmatchingset}), we now need to match to the reference orbit $\mathcal{O}(z_k,q_k)$ 
also to all other orbits segments of length $q_k$ 
obtaining subdividing the orbit  $\mathcal{O}(x, L q_k)$  of $x\in I\setminus X_k^L$ of length $Lq_k$ into $L$ pieces, namely $\mathcal{O}(T^{j q_k } x, q_k)$ for $j=1, \cdots, L-1$.  The definition of $X_k=X_k^L$ guarantees that for every such $j$, the 
point $T^{jq_k}x$ belongs to {$I\backslash X_k^1$ (i.e.~their orbits of length $q_k$ avoid the small rectangles and the bad set $B_1$), see \eqref{eq:Ltimesnotbad})}, and hence allows to apply to each  $\mathcal{O}(T^{j q_k} x, q_k)$  exactly the same 
matching arguments.}

\smallskip
\noindent {\it Step 7.   Verifying the balanced matchings properties:}
Let us show that the matching sets $M_k =I\setminus X_k$ are $1/5$-balanced matchings for every $k\in\N$, (see Definition \ref{def:balancedm}). Thus, we have to prove that properties $(B1)$ to  $(B4)$ of Definition \ref{def:balancedm} hold for $c:= 1/5$. {We will then show at the end of this step that the same proof also gives that they are indeed $L$-fold $1/5$-balanced matching sets.}

To see that (B1) holds, it is more convenient to look at the dynamics or orbits inside the Rohlin towers over the intervals $I^{(n_k)}_j$, $1\leq j\leq d$ (before the cut and stack corresponding to inducing on the interval $\tilde I^k$). The orbit $\mathcal{O}_T(z_k,q_k)$ (since $q_k=q^{(n_k)}_1+ q^{(n_k)}_d$ and $z_k\in I^{(n_k)}_1$) goes once through the Rohlin tower over $I_{1}^{(n_k)}$, and once through the tower over the interval $I_{d}^{(n_k)}$. Moreover, since the relative displacement $\Delta_k<1/10$ (by assumptions on $C$), we have that the point $T^{q_{1}^{(n_k)}}(z_k)$ lies in the middle half of the interval $I_{d}^{(n_k)}$, while $T^{q_{1}^{(n_k)}+q_{5}^{(n_k)}}(z_k)$ lies in the middle half of the interval $I_{1}^{(n_k)}$. Hence every point in the orbit  $\mathcal{O}_T(z_k,q_k)$   has distance at least $\frac{1}{4}|I_{5}^{(n_k)}|$  from  the endpoints of levels of towers, which contain in particular the set $Disc(T)$. It remains to notice that for $k$ large enough we have by \eqref{eq: returnset} that $|I_{d}^{(n_k)}|>\frac{9}{10q_k}$. 

The proof of (B2) goes along the similar lines as the proof of (B1), by noticing that if a point does not belong to a level of tower which has a discontinuity as its endpoint then its distance from any element of $Disc(T)$ is at least the length of one level of one of the bigger towers. On the other hand, the property (B3) follows from the fact that no matching interval is longer than the length of the $\tilde I^k$, as seen in Figure \ref{fig: matchings}.

Let us now prove (B4).  For a given $k\in \N$, assume that $x\in M_k$. If 
$x\notin \bigcup_{j=-q_k}^{\tilde q_1^k-1}(\tilde I_1^k)$, then the orbit of 
$x$ of length $q_k$ is fully contained in the Rokhlin tower 
$\left(\bigcup_{j=0}^{q^k-1}T^j(\tilde I^k_{2}\cup \tilde I^k_{d+1})\right)$ and 
each point is on a different floor. Thus $(B4)$ follows from the fact that $\Delta_k\in C=[c_0,c_1]$.
If on the other hand $x\in \bigcup_{j=-q_k}^{\tilde 
q_1^k-1}(\tilde I_1^k)$, then $T^{\ell}x\in \tilde I_1^k$ for some 
$\ell\in\{0,\ldots,q_k\}$ if and only if $T^{\ell}x\notin \tilde I_2^k\cup 
\tilde I_{d+1}^k$. Thus, again by $\Delta\in C$, 
we get that (B4) 
is satisfied for every $x\in M_k$.

\smallskip
\noindent {\it Step 8.   Verifying the good matchings properties:}
We now show that $(M_k)_{k\in\N}$ is a   sequence of  $(\gamma, B)$-good matchings for exponential tails for $f$, in the sense of Definition $\ref{def:goodMn}$, for  $\gamma:= (L+1) \cdot {\rho}\cdot c_1 +\epsilon$ (as claimed in the statement of Proposition~\ref{prop:Egoodmatchings}) {and a suitable $B>0$.  We then have to verify properties (G0) to (G3) of Definition~\ref{def:goodMn}.}

To prove (G0), we have to estimate the measure of $[0,1]\backslash M_k =[0,1]\backslash ([0,1]\backslash X_k)= X_k$, see  \eqref{eq:defXk}. 
While the measure of the  first union in the definition \eqref{eq:defXk} of $X_k$ convergences to 0 as $k\to\infty$ (since $\epsilon_k=1/k\to 0$), the second part of the union in  \eqref{eq:defXk} has non-trivial, although small, measure. More precisely, in view of the control on $\tau$ entries in terms of $\rho$ given by
 \eqref{eq: tausareclose},   we have 
\begin{equation}\label{eq: trashissmall}
	\limsup_{k\to +\infty} Leb(X_k)\le (L+1) \cdot {\rho}\cdot c_1 +\epsilon. 
\end{equation}
This shows that (G0) holds with  $\gamma:= (L+1) \cdot {\rho}\cdot c_1 +\epsilon$ up to  replacing $(M_k)_k$ with a subsequence if necessary.  

The proof that (G1)  holds, namely that  $(M_k)_{k\in\N}$ are $c$-good for $c=1/5$, was given in Step 7. 
 In Step 5 we proved that the reference orbit 
 $\mathcal{O}_T(z_k,q_k)$ satisfies  $B_1$-linear bounds for every $k$ for a uniform constant $B_1>0$ by using that $(n_k)_k$ are visit times to the set $E$ of Proposition~\ref{prop:accelerationset} (and the persistence Lemma~\ref{lemma:finitecombination}); {thus (G2) holds taking as constant the constant $B_1$ (in the place of $B_0$ in (G2)).

Finally, by Lemma~\ref{lemma:linearbounds}, there exists $B>0$ (which depends only on $B_1$) such that for any $k\in \mathbb{N}$, any orbit $\mathcal{O}(x,q_k)$ of  $x\in M_k$ satisfy $B$-linear bounds on trimmed derivatives, namely (G3) of Definition~\ref{def:goodMn} holds. 
This concludes the proof that $(M_k)_k$ are 
$(\gamma,B)$-good matching sets for exponential tails.
  
The 
proof that  $(M_k)_k$ are also  
$(\gamma,B)$-good $L$-\emph{fold} matching sets is similar. Indeed, the arguments in Step 7 (giving $1/5$-balance) as well as the arguments in this Step, can be applied to all other orbit segments $\mathcal{O}(T^{jq_k}x, q_k)$ for $j=1,\ldots,L-1$ in which each orbit $\mathcal{O}(x, L q_k)$ with $x\in M_k$ is subdivided. As in Step 6, the definition of $X_k:=X^L_k$ guarantees that each of the initial points  $T^{jq_k}x$ stays in $X_k^1$, see \eqref{eq:Ltimesnotbad}, and allows to reason in exactly the same way to verify the required properties.}
\end{proof}

%
%
%
%
%
%
%

\section{Tail estimates and  disjointness} \label{spectral disjointness}
In this section we will prove Theorem~\ref{thm:main}, namely disjointness for 
rescalings.  To do so, we apply the disjointness criterion given by Proposition~\ref{prof:disjointnesssf}. 
The key technical result of this section is that some \emph{tail sets}, namely sets of points which lie in the \emph{tail} of the limit distribution of Birkhoff sums (at multiples of) rigid times with exponential tails, 
 can be \emph{distinguished}, namely have different measures (see \S~\ref{sec:tailoverview} for more details). 

We first describe, in \S~\ref{sec:tailoverview}, the strategy of proof and the idea in a special, simplified setting. 
As we explain in \S~\ref{sec:final} (see the beginning of the proof of Theorem~\ref{thm:main}),  it suffices to prove that the flows  $\varphi^{L}_\R$ and $\varphi^{K}_\R$ are disjoint  for all pairs of positive integers  $K, L \in \N$. 
To keep the notation simple and clarify the main ideas of the proof, we will first prove the main estimate (Proposition~\ref{prop:distinguishedtailsL2}) with full details in the special case in which we consider  $\varphi^K_\R$ and $\varphi^L_\R$ with $K=1$ and $L=2$ (see \S~\ref{sec:distinguishedtrailssL2}).  
We will then indicate the modifications one needs to do for the general case in \S~\ref{sec:distinguishedtailsgeneral}.
Finally, the final arguments to get the proof of Theorem~\ref{thm:main} are presented in \S~\ref{sec:final}.

\smallskip
\noindent {\it Assumption}: To simplify the exposition of the proof, we assume throughout this section and the next that, WLOG, $f$ is normalized so that the constant $C_f$ given by Definition~\ref{def:SymLog} is  $C_f=1$, so that
\be \label{eq:C1} \sum_{i=0}^{d-1}C_i^- = 1= \sum_{i=1}^{d}C_i^+.
\ee
 Notice also that this implies that all the constants $C_i^\pm$ are in $[0,1]$.

{
\subsection{Strategy of the proof.}\label{sec:tailoverview} 
Let us start by giving in this section an overview of the strategy of the proof and the key idea behind the key tail estimate  
in  a simplified setting.  

 \subsubsection{Tail sets. } 
In view of the Criterion \ref{krytspek}, if $(q_k)_k$ is a sequence of rigidity times with exponential tails,   we need to 
find a set for which the limit distributions of the (centralized) Birkhoff sums 
of $f$ of length $q_n$, namely $S_{q_k}f - a_k$ for a suitable $a_k$, and  $S_{2q_k}f - 2a_k$  are going to differ, i.e.~can be \emph{distinguished}. More precisely, 
we will find a sequence $(t_k)_k$ of real numbers, converging to some $t\in\R$, such that, if  we define the tail sets
$$
A_k:= \{ x\in [0,1) |\, \, S_{q_k}f -a_k> t_k\} , \quad\text{and} \quad  B_k:= \{  x\in [0,1) |\, \, S_{2q_k}f-2a_k > 2t_k\}, \quad k\in \mathbb{N} , 
$$
then $Leb (B_k) - Leb( A_k)>c>0$.  
Here, $q_n$ are 
will be the heights of Rauzy-Veech towers obtained by Rauzy-Veech 
algorithm. More precisely, we will use a sequence of BT rigidity times and the corresponding good sets for exponential tails produced in Section~\ref{sec:Egoodmatchingsproof} and $q_k$ will be the common height of the \emph{large} towers $\tilde{\mathcal{T}}^k_2$ and $\tilde{\mathcal{T}}^k_{d+1}$ of the BT-rigid presentation. The constant $c$ providing the lower bound will be proportional to the displacement given by BT-rigidity.   
We call the set $A_k$ the \emph{single tail} set (since it corresponds to the Birkhoff sums for a \emph{single} iteration of the rigidity time $q_k$ only); we call he set $B_k$ the \emph{double tail} set (since it corresponds to the Birkhoff sums for a \emph{double} iteration of the rigidity time $q_k$, namely for $S_{2q_k}f$). The control of the Birkhoff sums distributions  that we prove, will not be obtained on the whole interval, but rather considering only a sequence of good matching sets $(M_k)_{k}$. A  main difficulty to overcome, is the fact that the measure of the set \emph{reminder} set $X_k=[0,1)\setminus M_k$ is also proportional to the displacement. This makes the required estimates delicate; however we are able to show that the constant of proportionality can be made much smaller for the set $X_k$, thus making the assumptions of the disjointness criterion given by Propostiion~\ref{krytspek} hold.}

In the proof we will show that the double tail set $B_k$ contains a substantial part of $A_k$, but also an additional part, so that $B_k\backslash A_k$ has larger measure  than $A_k\backslash B_k$. Thus, 
$
Leb(B_k)-Leb(A_k)
= Leb(B_k\backslash A_k) -Leb(A_k\backslash B_k)>0 , 
$
and prove a lower bound for the measure of this difference. 
To estimate the measures of these tail sets, we first   show that the Birkhoff sums $S_{q_k}f$ and $S_{2q_k}f$ both  blow up \emph{monotonically} near the endpoints of the floors of the Rohlin towers (see Step 3 in the proof of Proposition~\ref{prop:distinguishedtailsL2} in \S~\ref{sec:distinguishedtrailssL2}), so that the tail sets $A_k$ and $B_k$ can both be described by a union of intervals which contain these floors \emph{endpoints} where the Birkhoff sums blow up.  We will call these intervals \emph{tail intervals}, or more precisely \emph{single tail intervals} when they belong to $A_k$, and \emph{double tail intervals} when they belong to $B_k$.  Single (resp.~double) tail intervals  occur at {near the \emph{endpoints} of the floors,  due to closest visits to  a singularity  from the \emph{left}  (in which case they happen near the right endpoint of an interval of the tower $\mathcal{T}^k_{2}$), or due to closest visits to a singularity from the \emph{right}  (in which case they happen near the left endpoint of an interval of the tower $\mathcal{T}^k_{d+1}$). We call these tail intervals respectively \emph{left} and \emph{right} tail intervals). }

A picture of the double tail and single tail sets is shown in Figure~\ref{fig:tailsets}, which shows that the sets both consists of intervals that are either to the right of a floor of the tower $\tilde{\mathcal{T}}^k_2$ or to the left of a floor of the tower $\tilde{\mathcal{T}}^k_{d+1}$ and hence travel close to the singularities (that all belong to the boundary of the 'small towers' $\tilde{\mathcal{T}}^k_3,\cdots \tilde{\mathcal{T}}^k_d $).  
 The \emph{tail intervals} belonging to   different floors of the same tower  (either $\tilde{\mathcal{T}}^k_2$ or $\tilde{\mathcal{T}}^k_{d+1}$) turn out to all have approximately the same measure (see Steps 3 to 8 of the proof of Proposition~\ref{prop:distinguishedtailsL2}), so let us focus in the next \S~\ref{sec:keyestimateidea} on one right interval and one left interval and describe the single tail and double tail restricted to these two intervals.

	

 \subsubsection{The idea of the key estimate in the simplified setting.}\label{sec:keyestimateidea}
Let $(a,b)$ be one of the  floors  of the \emph{large} tower $\tilde{\mathcal{T}}_2$. For simplicity let it be the floor such that $b$ is one of the singularities. 
Let us compare the single and double tail inside $(a,b)$, namely $A_k\cap (a,b)$ and  $B_k\cap (a,b)$. One can  first show   that if $x$ is sufficiently close to $b$, $S_{q_k}f (x)$ is monotonically increasing (see Step 1 of the proof of Proposition~\ref{prop:distinguishedtailsL2}).  Thus, both sets are intervals with $b$ as an endpoint. Furthermore, since for any $x\in (a,b)$,  $T^{q_k}x$ is to the left of $x$, by monotonicity $ S_{q_k}f(T^{q_k}x) <  S_{q_k}f(x) $, so that 
 $$
 S_{2q_k}f(x) = S_{q_k}f(x)+ S_{q_k}f(T^{q_k}x)< 2 S_{q_k}f(x). 
 $$
 From this, one can see that if $x\in B_k $, then  $x\in A_k$ (since  $S_{2q_k}f(x)>2a_k+2t_k$ implies that $S_{q_k}f(x)>a_k+t_k$), so that $B_k\cap(a,b)\subset A_k\cap (a,b)$ i.e.~in this floor (as in other floors of this tower $\tilde{\mathcal{T}}^k_2$), the double tail set is contained in the single tail set.

  We now want to estimate the measure of the difference $A_k\backslash B_k\cap (a,b)$. 
By monotonicity,   $B_k\cap (a,b)$ is an interval of the form $(v_k, b)$ where the point $v_k$ where 'the double tail starts' is defined by  $S_{2q_k}f(v_k)-2a_k=2t_k $. It will be chosen, so that $|b-v_k|\approx \tfrac{\Delta_k}{q_k}$, where $\Delta_k=q_k\tilde\lambda_1^{k}$ is the relative shift given by $T^{q_k}$. Moreover, $x \in (a,b)\backslash B_k$ is in 
$A_k\backslash B_k\cap (a,b)$ exactly when $2S_{q_k}f(x)\ge 2t_k+2a_k= S_{2q_k}f(v_k)$, namely when  
$2S_{q_k}f(x)- S_{2q_k}f(v_k)\ge0$.  
 One can show (exploiting the bounds on trimmed Birkhoff sums of the derivative) that the main order of $S_{q_k}f (x)$ for $x$ close to $b$ is given by the contribution of the \emph{closest visits} to singularities. Furthermore, since singularities in a BT rigid presentation are all \emph{cramped} close to each other (see Figure~\ref{fig:tailsets}, in which singularities are indicated by red dots) and by assuming that $C_f=\sum_{i=1}^d C_i^-=1$, the effect of the $d$ singularities is (up to a negligible error, computed in Step 4 of the proof of Proposition~\ref{prop:distinguishedtailsL2}) the same than the contribution of the closest taking the constant $1$, namely $-\log(b-x)$.  Similarly, $ S_{2q_k}f(v_k)$ can be approximated by the contribution the closest visit for  $S_{q_k}f(v_k)$ (which gives a contribution $-\log(b-v_k)$) and of the closest visit of $S_{q_k}f(T^{q_k} v_k)$ (which is of order $(\log(b-T^{q_k}v_k)$). 
Thus,  the difference $2S_{q_k}f(x)- S_{2q_k}f(v_k)$, when $x=v_k-\tfrac{1}{2}(b-v_k)$ (i.e. halfway between $v_k$ and $T^{q_k}(v_k)$), is approximately 
\[
 -2\log(b-x)+\left(\log(b-T^{q_k}v_k)+\log(b-v_k)\right)=-2\log\left(\frac{3}{2}\right)+\log(2)<0.
\]
By continuity, if we consider a point  $x$  in $(v_k, b)$
 sufficiently close to $v_k$, we continue to have  that $2S_{q_k}f(x) - S_{2q_k}f(v_k) <0 $, so that $x\notin A_k$.  
More precisely, $x$ can be chosen so that $|b-x|=\frac{\alpha}{q_k}e^{-D}$, with $\alpha\in(\sqrt{2},\frac{3}{2})$ (in the main proof that follows we take $\alpha=\frac{10}{7}$). This shows that strictly less than a third of  $ A_k\cap  (a, b)$ does \emph{not} belong to $B_k$.

Similarly, one can then treat also the left intervals of the tails, namely those that correspond to orbits come close to the cramped singularities from the left. 
 Since again one can show that they all have approximately the same size (see Step 7 and 8 of the proof of Prop.~\ref{prop:distinguishedtailsL2}),  let us here only  consider the floor $(\hat{a},\hat{b})$  which has a singularity $\hat{a}$ as (left) endpoint.  
Again, we assume that $x$ is sufficiently close to $\hat{a}$, so that $S_{q_k}f (x)$ is monotone and we can approximate its values with the contribution of the closest visits, which are all cramped and contribute as the closest. {We first show that both the single and the double right tail on this side too  are union of intervals. In particular, we show that 
$B_k\cap (\hat{a},\hat{b}) = (\hat{a}, \hat{v}_k)$ where $ \hat{v}_k$ is defined by $S_{2q_k}f(\hat{v}_k) -2a_k= 2t_k$.  To show that any $x\in  (c, \hat{v}_k)$ is also in $B_k$,  since  $T^{q_k}x$ is (as before) to the \emph{left} of $x$ (and hence now moves \emph{in direction} on the singularities rather than away from them) we have to discuss more cases, according to whether $T^{q_k}x$ is still to the right of $c$ (in which case one can use a simple monotonicity argument, analogous to the one on the left side), or if it is to the left of $c$ (and hence visits  the singularities also from the left). In this case one can still use matching and monotonicity to show that $(\hat{a}, \hat{v}_k)\subset B_k$ (see Step 8 of the proof of Prop.~\ref{prop:distinguishedtailsL2}  for details). 

A key difference on this side is that now that   the single tail interval $A_k\cap (\hat{a},\hat{b})$ in $(\hat{a},\hat{b})$  is \emph{contained} in the double tail interval  $(\hat{a},\hat{v}_k)=B_k\cap (\hat{a},\hat{b})$ (see Figure~\ref{fig:tailsets}). Indeed,  while $T^{q_k}x< x$ as 
as before , now
  $S_{q_k}f (x)$ is decreasing; therefore, if we assume for example that $T^{q_k}(x) $ is still in $(\hat{a},\hat{b})$, 
 $ S_{q_k}f(T^{q_k}x) >  S_{q_k}f(x) $ and
 $$
 S_{2q_k}f(x) = S_{q_k}f(x)+ S_{q_k}f(T^{q_k}x)> 2S_{q_k}f(x),
 $$  
 so if $ S_{q_k}f(T^{q_k}x) -a_k> t_k/2$ (i.e.~$x\in A_k$), we also have that  $S_{2q_k}f(x)-2a_k > t_k$ (i.e.~$x\in B_k$). 
 
  Exploiting  convexity as well as the symmetry of the function $f$ (which plays here an essential role, since it allows for a fine control of the constants involved in the estimates),  one can then  prove, similarly as before, that if {$\hat v_k$ 
  is the right endpoint of  the right double tail $(\hat{a},\hat{b})\cap B_k$, all the points $x$ such that { $\hat v_k-\frac{\Delta_k}{2q_k}\le x<\hat v_k$ } are contained  in $(\hat{a},\hat{b})\cap B_k\backslash A_k$. Notice that (essentially for a convexity effect)  this is an interval of  size roughly $\frac{1}{2(\alpha-1)}$ (i.e.~$\frac{7}{6}$ for $\alpha=\frac{10}{7}$)} of the size of the interval $(a,b)\cap A_k\backslash B_k$.  
Since { $7/6>1$} this gives an heuristic explanation of
 why we can expect  
$$ Leb \left((\hat{a},\hat{b})\cap B_k\backslash A_k \right)>  Leb \left((a,b) \cap A_k\backslash B_k \right)$$
and how to provide an estimate from below for its measure. Since all the other tail intervals whose orbits get close to \emph{all} cramped singularities, from the bottom one to the top one, in one \emph{run} (as defined in Step 2 of the proof of Prop.~\ref{prop:distinguishedtailsL2})  give  contributions to the single and double tails of comparable sizes to the tail intervals in $(a,b)$ and $(\hat{a},\hat{b})$ (see Steps 5 to 8 of the proof of Prop.~\ref{prop:distinguishedtailsL2}). We consider separately the tail intervals which are contained in floors close to the cramped singularities, whose orbits travels close to some but not all of them (we call these \emph{incomplete runs}). These provide a contribution of comparable, but  negligible order (in the sense that it is proportional to the measure of the union of all the other tail intervals, but with a small proportionality constant; see Step 9 of the proof of Prop.~\ref{prop:distinguishedtailsL2} for details).}  Adding up the contribution of all left and right tail intervals (corresponding to complete and incomplete runs) 
we then show that $ Leb ( B_k\backslash A_k )-  Leb ( A_k\backslash B_k)>0$, and furthermore give a lower bound for the measure of the difference.  {}


The proof of Proposition~\ref{prop:distinguishedtailsL2} that we present in the next section \S~\ref{sec:distinguishedtrailssL2} follows the lines of this simplified sketch,  taking into account the contribution of the error terms due to the simplifications and approximations made above.

\begin{figure}
	\includegraphics[scale=0.9]{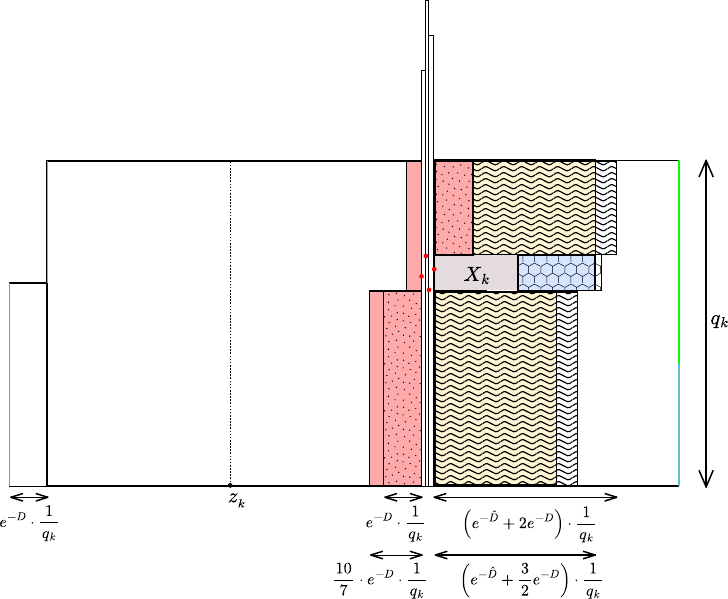}
	\caption{{An illustration of the tail sets  considered in 
	proof of Theorem \ref{thm:main}. \emph{Single} tail intervals are in color: 
\emph{right} single tail intervals in red, \emph{left}  ones  in yellow;  those coming from incomplete visits are in light blue. 
\label{fig:tailsets}	 \emph{Double} tail intervals are shaded with patterns:  \emph{left} double tail intervals are dotted, right ones are shaded with waves and the right double tail intervals which come from incomplete visits have an octagonal pattern. One sees that the single tail \emph{contains} the double tail intervals (dotted) on the right, while on the left the double tails intervals (with waves) contain the (yellow) single tails intervals.}}
	\end{figure}
	
\subsection{Distinguishing simple and double tails}\label{sec:distinguishedtrailssL2}
The following proposition shows that properly chosen tails can be distinguished when $K=1$ and $L=2$.

\begin{prop}[Distinguishing tails when $K=1,L=2$]\label{prop:distinguishedtailsL2}   {For a sufficiently small  $\rho>0$, for    any $B>0$,  
 there exists 
  $ \epsilon>0$ and  a non-empty $C=[c_0,c_1]\subset \mathbb{R}_+$, 
}
 such that 
 for any sequence $(M_k)_k$ of  $(z_k,q_k)-$matchings {for $T$ and $f\in SymLog^2(T)$} given by a $(C,\epsilon, \rho)$-BT rigid presentation  {which are {$(\gamma, B)$-good}} for 
 {$\gamma<3 \rho c_1+\epsilon$,} 
 there exists a \emph{bounded} sequence $(t_k)_k$
  such that if, for every $k\in \N$, we consider the tail sets:
\[
A_k:=\{x\in M_k \mid S_{q_k}(f)(x) -a_k \ge t_k \}, \quad  \text{and} \quad  
B_k:=\{x\in M_k\mid S_{2q_k}(f)(x)-a_k \ge 2t_k  \} , 
\]
 then we have that $Leb(B_k)-Leb(A_k)\geq \gamma$. 
	\end{prop}

	\begin{proof}[Proof of Proposition~\ref{prop:distinguishedtailsL2}]
The proof is divided into steps, which will hopefully help to follow the strategy.
 

\smallskip

\noindent {\it Step 0: Choice of parameters.} 
Let us fix any $0<\rho\leq  10^{-6}$. If {$C_i$, $i\in \{0, \cdots , d-1\}$ are the constants from Definition \ref{def:SymLog}, let 
\be \label{Cmin}
C_{min}:= \min\{C_i^\pm>0\mid p=0,\ldots,d-1\}.
\ee
For the parameter $B>0$ given by the Definition \ref{def: derivativebound},} consider a number $D\in\R$ such that:
\begin{equation}\label{eq: defD}
D\ge 10^6\cdot\min\{1,B\}\ \text{ and }\ Be^{-D}\le \min\{\tfrac{1}{20}\log(\tfrac{50}{49}),C_{min }\cdot e^{-2B}\}.
\end{equation}
(The parameter $D$ will be used to define the size of the tail, see 
Step 1 in the proof of Prop.~\ref{prop:distinguishedtailsL2}).  Set 
$$ 
 \epsilon:= 10^{-6}e^{-2D-B}. 
$$
We then define  the BT-rigidity interval $C=[c_0,c_1]$  in terms of $D$ and the
auxiliary constant $\delta>0$ to be  
$$c_0:= (1-\delta)e^{-D}, \quad c_1:=e^{-D} , \qquad  \text{where} \ \delta=e^{Be^{-D}}-1 .$$
\medskip

\noindent {\it Step 1: Definition of the tail sets.} 
{Let  $(M_k)_{k\in\N}$ be a sequence of $(\gamma,B)$-good} $(z_k, q_k)$-matching sets given by  $(C,\epsilon, \rho)$-BT rigid presentations. 
{Since by assumption $\gamma<3 \rho c_1+\epsilon$, we have that
\begin{equation}\label{eq: gammaestimate}
\gamma
\leq 3\cdot 10^{-6} e^{-D}+10^{-6}e^{-2D-B}< 4\cdot 10^{-6}e^{-D}.  
\end{equation}}
 We will denote by ${{\mathcal{T}}}^k_i$ for $1\leq i\leq d$  the towers of the $k^{th}$ BT rigid presentation (used to define  $M_k$). 
We define the centering sequence $(a_k)$   {using the Birkhoff sums of the  matching point $z_k$ for $M_k$ as a reference point, namely setting: 
\[
a_k:=\sum_{i=0}^{q_k-1}f(T^i(y_k)), \qquad k \in \mathbb{N}.
\] 
To choose the tail parameters $(t_k)_k$ (which will be used to define the tail sets), we consider for every $k\in \N$  points $v_k^\pm$ in the base whose orbits 
$\mathcal{O} (v_k^\pm, 2q_k)$ of lenght $2q_k$ have distance from singularities respectively from the right and the left of order $e^{-D}/q_k$ (where $D$ is the constant defined above). More precisely,  $v^-_k\in (a,b):=\tilde I_2^k$ belongs to the base of the (large) tower $\tilde I_{2}^k$ and  $v^+_k\in  \tilde I_{d+1}^k$  
 belongs to the base of the (large) tower $\tilde I_{d+1}^k$, that we will denote $(\hat{a},\hat{b}):= \tilde I_{d+1}^k$, and are implicitely defined by\footnote{{Notice that  $T^{q_k} v_k^{\pm} = v^{\pm}_k - \Delta_k/q_k < v^{\pm}_k$. Therefore, the closests visits from the right within $\mathcal{O} (v_k^-, 2q_k)$ is contained in the first half of the orbit, namely in  
 $\mathcal{O} (v_k^-, q_k)$, while the closest visits from the left in  $\mathcal{O} (v_k^+, 2q_k)$ is contained in the second half of the orbit, namely 
 $\mathcal{O} (T^{q_k} v_k^+, q_k)$. This is the reason why we consider the distance from $\hat{c}$ of the point  $T^{q_k} v_k^+$.}} 
\bes 
 b- v_k^-= \frac{e^{-D}}{q_k}  , \quad \text{and}\quad   T^{q_k} (v_k^+)-\hat{a}=\frac{e^{-D}}{q_k} .
 \ees
We then compare the Birkhoff sums $ S_{2q_k}(f)$ at these two points and let $v_k\in \{v_k^+, v_k^-\}$ be the point with the smallest one,  so that 
\be \label{eq:vkmin}
S_{2q_k}(f)(v_k)= \min \{ S_{2q_k}(f)(v_k^+), S_{2q_k}(f)(v_k^-)\}.
\ee
We can then define the tail parameters $(t_k)$ {using the Birkhoff sums of $v_k$, by setting}  
\be \label{def:tk} t_k:= \frac{S_{2q_k}(f)(v_k)}{2}-
a_k = \frac{1}{2}\left(S_{2q_k}(f)(v_k)-
2a_k \right), \qquad k \in \mathbb{N}.
\ee
Consider now, for every $k\in \mathbb{N}$, the \emph{single  tail set} }
\[
A_k:=\{x\in M_k \mid S_{q_k}(f)(x)-a_k\ge t_k \} 
\]
and  the \emph{double tail} set
\[
B_k:=\{x\in M_k \mid S_{2q_k}(f)(x)-2a_k\ge 2 t_k\}.
\]
Note that the tail sets are defined considering only points in the matching set 
$M_k=I\setminus X_k$.  
{We remark also, for later use, that, by the definitions of the tail sets and $v_k$, so that \eqref{eq:vkmin} holds,
\be \label{rk:vk}
\text{if} \ v_k=v_k^-, \ \text{then}\, v_k^+\in B_k, \quad \text{while}\quad \text{if} \ v_k=v_k^+, \ \text{then}\, v_k^-\in B_k.
\ee }
Let us now show that, as claimed, the sequence  $(t_k)_k$ giving the thresholds of the tail sets 
 $A_k$ and $B_k$ 
  is \emph{bounded} (as claimed as part of the statement we are proving). To show this,  let us bound  the difference of Birkhoff sums $S_{q_k}(f)(v_k)-S_{q_k}(f)(z_k) $ by estimating separately the difference of the trimmed parts $\tilde S_{q_k}(f)(v_k)- {\tilde S}_{q_k}(f)(z_k) $ (using  mean value theorem and  \eqref{eq: Corinnaconstant}) and the contribution of  the closest visits (recalling that their contributions is positive, that the closest distance is bounded below by {$e^{-D}$} by construction, and using that $\sum_{i=0}^{d-1}C_f^+=1$). We then get the estimate 
 \begin{equation}\label{eq: boundonBS}
 0\le	S_{q_k}(f)(v_k)-a_k = S_{q_k}(f)(v_k) -S_{q_k}(f)(z_k) \leq \tilde S_{q_k}(f)(v_k)-\tilde  S_{q_k}(f)(z_k) - \sum_{i=0}^{d-1}{C^+_i}\log {m(v_k,q_k)}   \le B +D.
 \end{equation}
One can then show that $0\leq t_k\leq S_{2q_k}(f)(v_k)-2a_k \leq 2(B+D)$ for every $k$, so that $(t_k)_k$  is bounded. 
 \medskip
 
\noindent {\it Step 2: Tail intervals location.} 
 Let us first identify the points $x\in M_n$ such that the orbit $\mathcal{O}(x,q_n)$ gets close to {\emph{each} of the cramped singularities in sequence (from the lowest to the heighest), at the same distance.} Let 
 $\underline{\ell}_k$ and $\overline{\ell}_k$  in $\{0,1,\cdots, q_k-1\} $ be the heights  of (the floor which contains) respectively the \emph{lowest} and the \emph{heighest} of the cramped singuarities (see also the definition in Step 3 of Proposition~\ref{prop:Egoodmatchings}).  Let $(a_{0},b_{0})= T^{\underline{\ell}_k} \tilde 
I^k_2= T^{\underline{\ell}_k} (a,b)$ be the floor of the tower $\tilde {\mathcal{T}}^k_2$ at the height of the lowest cramped singularity.  One can then consider the intervals\footnote{{One \label{ft:intervals} can show that $T^{-i}$ on $(a_0,b_0)$  is continuous  for $0\leq i\leq q_k-1-( \overline{\ell}_k-\underline{\ell}_k)$ and hence the preimages  $T^{-j}(a_0,b_0)$ considered in \eqref{eq:ajbj} are indeed intervals. One can indeed verify that, by the tower representation dynamics, none of the (cramped) singularities is hit by an iterate $T^i(a_0,b_0)$ until $i=q_k-( \overline{\ell}_k-\underline{\ell}_k)$. }} given by
\be \label{eq:ajbj} (a_j,b_j):= T^{-j}(a_0,b_{0}), \quad   \text{for}\ j\in\{0,\ldots,q_k-1 - ( \overline{\ell}_k-\underline{\ell}_k) \}.
\ee   
{These intervals contain  points whose orbit up to time $q_k$ gets close to the \emph{all} cramped singularities \emph{from the left}, in order from the lowest to the heights (we will refer to  this as a \emph{run} of closests visits from the left). We  will  show that these intervals contain the \emph{left} (single and double) tail intervals, which have as endpoints the $b_j$s. 
Similarly, to describe the part of the tail sets 
which  correspond to the closest visits \emph{from the right side}, 
{let   $(\hat{a}_{0},\hat{b}_{0})= T^{\underline{\ell}_k} \tilde 
I^k_{d+1}= T^{\underline{\ell}_k} (\hat{a},\hat{b})$  and consider the following preimages, that one can also show (as in footnote \ref{ft:intervals}) are all intervals:
\be \label{eq:ajbjhat}
(\hat a_j,\hat 
b_j):=T^{-j}(\hat a_{0},\hat b_{0}), \quad 
\text{for}\  
j\in\{0,\ldots,q_k-1-(\overline{\ell}_k-\underline{\ell}_k)\}. \ee
These intervals contain  points whose \emph{forward} orbit up to time $q_k$ gets close to the \emph{all} cramped singularities \emph{from the right} in a \emph{run} (from the lowest to the heighest). We will show that they contain the \emph{right} (single and double) tail intervals, which have the $\hat{a}_j$s as left endpoints. 

\smallskip
Remark that if $x\in T^i(\hat{a}, \hat{b}) \backslash X_k$ for $0\leq i\leq \overline{\ell}-\underline{\ell}$ (i.e.~belongs to the floors of $\mathcal{T}^k_{d+1}$ which \emph{contain} the bad zone $X_k$ but is not it),  then $\mathcal{O}(x,q_k)$ will also have (as any orbit) a closest visit to each of the singularities, but it will not contain a full \emph{run} of visits (from the lowest to the heighest) with the  same approximte distance (up to a small error estimated in Step 4). We then say that this orbit (or the corresponding Birkhoff sum $S_{q_k}f(x)$) has an \emph{incomplete run} of closest visits.} These points can still belong to the (single or double) tail sets (in Figure~\ref{fig:tailsets} these points are denoted in light blue or hexagonal pattern).
}

\medskip
\noindent {\it Step 3: Monotonicity near the endpoints.} 
We now claim that the function $S_{q_k}(f)$ is monotone, close to the endpoints  of the intervals defined in Step 1 (more precisely, increasing in $(a_j,b_j)$ sufficiently close to $b_j$ and decreasing in $(\hat{a}_j, \hat{b}_j)$ sufficiently close to $\hat{a}_j$. 
More precisely, since (from the form of $f$ in Definition~\ref{def:SymLog}) visits from the left of a discontinuity give positive contributions to $f'$ and visits to the right negative ones,  
since $B$ is chosen in accordance to Definition \ref{def: derivativebound}, for points in $y\in (a_j,b_j)$ which are sufficiently close to $b_j$ (which get close to singularities \emph{from the left} in time $q_k$), 
 we can estimate 
{}
$$
 S_{q_k}(f')(y) \geq \sum_{i=1}^d\frac{C_i^-}{m^-_i(y,q_k)} +\tilde{S}_{q_k}(f')(y)   \geq  \frac{1}{|b_j-y|} -  B q_k 
$$
so that, if $|b_j-y|<1/Bq_k$, the $ S_{q_k}f$,  and similarly $ S_{2 q_k}f$,  are increasing. Similary, 
for a point  $y\in (\hat{a}_j, \hat{b}_j) $ sufficiently close to $\hat{a}_j$,  we can estimate 
{}
$$
 S_{q_k}(f')(y) \leq  \sum_{i=0}^{d-1}\frac{C_i^+}{m^-_i(y,q_k)} +\tilde{S}_{q_k}(f')(y)   \leq  \frac{1}{|\hat{a}_j-y|} -  B q_k 
$$
so that, if $|\hat{a}_j-y|<1/B q_k$,  then $ S_{q_k}f$ is  decreasing.   Similarly one can  show that  also $ S_{2q_k}f$ is increasing in $(a_j,b_j)$ sufficiently close to  $b_j$ and descreasing in $(\hat{a}_j,\hat{b}_j)$ sufficiently close to $\hat{a}_j$ as long as the distance from the endpoints is not greater than $1/2B$.

The tail thresholds chosen in Step 1 automatically guarantee that all the tail sets that we will describe in the next Steps are contained in these monotonicity intervals near enpoints.
 (Indeed, we will always work with points that are at most ${e^{-D+2B}}/{q_k}$ far from either the  endpoint $b_j$ of $(a_j,b_j)$ or the $\hat{a}_j$ for right tail intervals), and by definition $D\geq 10^6\cdot {\min\{1,B\}}$, so that one has that $e^{-D+2B}\leq 1/2B$. This guarantees  that the    $A_{k}\cap(a_j,b_j)$ and $B_k\cap(a_j,b_j)$ are both \emph{intervals} with $b_j$ as endpoint and similarly that the   $A_{k}\cap(\hat{a}_j,\hat{b}_j)$ and $B_k\cap(\hat{a}_j,\hat{b}_j)$ are \emph{intervals} with $\hat{a}_j$ as endpoint.

\smallskip
\noindent {\it Step 4: Errors in closest visits approximation.}
For a given $j\in\{0,\ldots,q_k-1-(\overline\ell_k-\underline\ell_k)\}$, the distance from the actual discontinuities of the orbit of a points $y$ in $(a_j,b_j)$ or $(\hat a_j,\hat b_j)$
is not $|b_j-y|$ or $|y-\hat a_j|$,  since the intervals $\tilde I_3^k,\ldots, \tilde 
I_d^k$ have positive lengths. These distances may differ by up to $|\tilde I^k_3|+\ldots+|\tilde I^k_d|$, see \eqref{eq: distfromdisc}. However, by the definition of $BT$-rigid configuration, for  $k$ large enough we have
\[
|\tilde I_3^k|+\ldots+ |\tilde I_d^k|\le \frac{\epsilon}{q_k}.
\] 
Let $j=0,\ldots,q_k-1$ and $y\in(a_j,b_j)$ with $|b_j-y|\ge \frac{e^{-D-B}}{q_k}$. Remembering that $\sum_{i=0}^{d-1}C_i^+=1$, we have 
\begin{equation}\label{eq: residuum}
	\begin{split}
	-\log(b_j-y)& +\sum_{p=1,\ldots,d-1} C_p\log({m_{p}^-(y,q_k)})=
	\sum_{p=1,\ldots,d-1} C_p\left(\log\frac{m_{p}^-(y,q_k)}{b_j-y}\right)\\
	&\le \log\left(\frac{b_j-y+\frac{\epsilon}{q_k}}{b_j-y}\right)= 
	\log\left(1+\frac{\epsilon}{q_k(b_j-a)}\right)\le \log(1+\frac{\epsilon}{e^{-D-B}})\le 10^{-6}e^{-D},
	\end{split}
\end{equation}
where the last inequality follows from the choice of $\epsilon$.
We define
\[
\xi_k^j:=\sup_{y\in\left(a_j,b_j-\frac{e^{-D-B}}{q_k}\right)}-\log(b_j-y) +\sum_{p=1,\ldots,d-1} C_p\log({m_{p}^-(y,q_k)}).
\]
Analogously, for the right-hand side, we define 
\[
\hat\xi_k^j:=\sup_{y\in\left(\hat a_j+\frac{e^{-D-B}}{q_k},\hat b_j\right)}-\log(\hat a_j-y) +\sum_{p=1,\ldots,4} C_p\log({m_{p}^+(y,q_k)})
\]
and by proceeding analogously as in \eqref{eq: residuum} we get that for any $k$ large enough,  we have
\begin{equation}\label{eq: residuasmall}
	\xi_k:=\max_{j=0,\ldots,q_k-1}\{\xi_k^j,\hat\xi_k^j\}<10^{-6}e^{-D}.
	\end{equation}


\smallskip
\noindent {\it Step 5:  Estimate of the left double tail intervals.} 
\smallskip
{From now on we assume that $v_k=v_k^-$. If the point $v_k$ defined in Step 1 had been $v_k^+$ 
	we would have right tail intervals that are larger than the left tail intervals (i.e.~$b-v_k>\hat v_k-\hat a$) and  the monotonicity arguments and computations done in the following Steps would remain analogous, with the role of the left-hand side tails and right-hand side tails swapped.
}
Let us describe the set $B_k\cap \bigcup_{j=0}^{q_k-1}(a_j,b_j)$, where $(a_j,b_j):= 
T^{-j}(a_{0},b_{0})$ are the sets defined in Step 2 {(dotted in Figure~\ref{fig:tailsets})}. We will show that  $B_k\cap (a_j,b_j)$
are intervals  roughly the same size (which correspond to the points whose closest 
visit is on the left-hand side of the singularities) namely of the intervals 
of the form $(v^j_k,b_j)$.  
For each  $j\in\{0,\ldots,q_k-1\}$ consider the interval and let
$z^j_k\in$ be the starting point of the tail, namely the unique point (which exists by monotonicity) such that
\begin{equation}\label{eq: defzkj}
S_{2q_k}(f)(v^j_k)-2a_k=S_{2q_k}(f)(v_k)-2a_k:=2t_k,
\end{equation}
and let  us define
\begin{equation}\label{eq: defDjk}
D^j_k:=-\log(b_j-v^j_k)-\log(q_k) \quad \Leftrightarrow   \quad  b-v^j_k=\frac{e^{-D^j_k}}{q_k}  .
\end{equation}
By the choice of $v^j_k$ and monotonicity (Step 3), we then have the inclusions
	\begin{equation}\label{eq: intervals on the left}
		[v_k^j,b_j)\subset B_k 
		\text{ 		for every } 
		j\in\{0,\ldots,q_k-1-(\overline{\ell}_k-\underline{\ell}_k)\}.
	\end{equation}
We now want to compare $D^j_k$ and $D$ to show that all left double tail intervals $[v_k^j,b_j)$, which by definition have size ${e^{-D^j_k}}/{q_k}$,  have roughly the same size, comparable to ${e^{-D}}/{q_k}$ . We claim that we have a first have a comparison that comes from the choice of $B$ (see Definition \ref{def: derivativebound}) and \eqref{eq: residuasmall}, which is
\begin{equation}\label{eq: trivialDwithD}
	D-2B\le D^j_k\le D+2B.
\end{equation} 
{Indeed, otherwise, if for example $D^j_k<D-2B$, then by Mean Value Theorem we would have
\[
\begin{split}
S_{2q_k}&(f)(v_k^j)-S_{2q_k}(f)(v_k)\\
&<(D-2B)-D-\log(e^{-D+2B}+\Delta_k)+\log(e^{-D}+\Delta_k)+\tilde S_{2q_k}(f)(v_k^j)-\tilde S_{2q_k}(f)(v_k)+4\xi_k\\
&< -{2}B+4\xi_k<0,
\end{split}
\]
which contradicts the choice of $v_k^j$. One proves analogously that $D^j_k>D+2B$ cannot hold either.
Moreover, we have
\begin{equation}\label{eq: BS splitting}
\begin{split}
0=S_{2q_k}(f)(z^j_k)-S_{2q_k}(f)(v_k)
&=\tilde S_{q_k}(f)(z^j_k)+\tilde S_{q_k}(f)(T^{q_k}z^j_k)-\tilde 
S_{q_k}(f)(v_k)-\tilde 
S_{q_k}(f)(T^{q_k}v_k)\\
&-\log(e^{-D^j_k}+\Delta_k)+\log(e^{-D}+\Delta_k)+D^j_k-D+\xi_k,
\end{split}
\end{equation}
where $\xi_k$ is defined in \eqref{eq: residuasmall}. By the choice of $B$, the fact that $M_n$ is a good matching set and Mean Value Theorem we get
\[
\tilde S_{q_k}(f)(z^j_k)+\tilde S_{q_k}(f)(T^{q_k}z^j_k)-\tilde 
S_{q_k}(f)(v_k)-\tilde 
S_{q_k}(f)(T^{q_k}v_k)\le 4Be^{-D+2B}.
\]
Note that the expressions $-\log(e^{-D^j_k}+\Delta_k)+\log(e^{-D}+\Delta_k)$ 
and $D^j_k-D$ are of the same sign. Hence, by \eqref{eq: BS splitting} and  
\eqref{eq: residuasmall}, we get
\begin{equation}\label{eq: DcompareswithD}
	|D^j_k-D|\le 4Be^{-D+2B}+|\xi_k|\le 5B{e^{-D+2B}}.
\end{equation}
\noindent Equation \eqref{eq: DcompareswithD}, in view of the definition \eqref{eq: defDjk} of $D_k^j$, shows that the part of the double tail inside a left tail interval $(a_j,b_j)$, namely $(a_j,b_j)\cap B_k$, is an interval $(z_j^k,b_j)$ whose length $|b_j-z_j^k|$ are the same up to a multiplicative constants. 

\smallskip
\noindent {\it Step 6: Estimate of the  left single tail intervals.} We will now show that the single tail from the left (in red in Figure~\ref{fig:tailsets}) satisfies
\begin{equation}\label{eq:singleinclusion}
(a_j,b_j)\cap A_k\subset \left(b_j-\frac{10}{7}(b_j-v_k^j),b_j\right), \qquad \forall\  0\leq j<q_k.
\end{equation}
 This is a  crucial step to show that the tail $A_k$ is sufficiently smaller than $B_k$.
 Let $j\in\{0,\ldots,q_k-1\}$ and let $(a_j,b_j)$, $v_k^j$ and $D_k^j$ be as above. We will show that the point 
\begin{equation}\label{eq: defykj}
y_k^j:=b_j-\frac{10}{7}(b_j-v_k^j)
=b_j-\frac{10}{7}\cdot\frac{e^{-D_k^j}}{q_k}
\end{equation}
does not belong to $A_k$. By monotonicity of the tail (Step 1) this suffices to get the desired inclusion \eqref{eq:singleinclusion}. Note that by \eqref{eq: DcompareswithD} we have
 \begin{equation}\label{eq: MVTinterval}
 	|y_k^j-v_k^j|,\ |T^{q_k}y_k^j-T^{q_k}v_k^j|<\frac{e^{-D}}{q_k}.
 \end{equation}
 We want
 \begin{equation}
	S_{q_k}(f)(y_k^j)-a_k\le \frac{1}{2}(S_{2q_k}(f)(v_k^j)-2a_k)=\frac{1}{2}(S_{2q_k}(f)(v_k)-2a_k),
\end{equation}
or, in other words we need to show that
\begin{equation}\label{eq: ynotintail}
2S_{q_k}(f)(y_k^j)\le S_{2q_k}(f)(v_k^j).
\end{equation}
By \eqref{eq: MVTinterval} and Mean Value Theorem, we have 
\[
|\tilde S_{q_k}(f)(y_k^j)-\tilde S_{q_k}(f)(v_k^j)|<Be^{-D}\text{ and }
|\tilde S_{q_k}(f)(y_k^j)-\tilde S_{q_k}(f)(T^{q_k}v_k^j)|<Be^{-D}.
\]
Recall that in view of \eqref{eq: residuasmall}, we have $\xi_k<10^{-6}e^{-D}<Be^{-D}$.
Since by the choice of recurrence set we have $(1-\delta)e^{-D}<\Delta_k<e^{-D}$ (recall that $\delta<e^{Be^{-D}}-1<2Be^{-D}$), again by Mean Value Theorem (applied to $\log$) we obtain
\[
|\log(e^{-D}+\Delta_k)-\log(e^{-D}+e^{-D})|<\frac{\delta e^{-D}}{e^{-D}}<2Be^{-D}.
\]
Hence we get 
\[
\begin{split}
2S_{q_k}&(f)(y_k^j)-S_{2q_k}(f)(v_k^j)\\\le
&2\tilde S_{q_k}(f)(y_k^j)-\tilde S_{q_k}(f)(T^{q_k}v_k^j)-\tilde S_{q_k}(f)(v_k^j)
-2\log({\tfrac{10}{7}e^{-D_k^j}})+\log(e^{-D}+\Delta_k)+\log(e^{-D})+4\xi_k\\\le
&2Be^{-D}+2\tilde S_{q_k}(f)(y_k^j)-\tilde S_{q_k}(f)(T^{q_k}v_k^j)-\tilde S_{q_k}(f)(v_k^j)
-2\log({\tfrac{10}{7}e^{-D_k^j}})+\log(e^{-D}+e^{-D})+\log(e^{-D})+4\xi_k\\ \le
&8Be^{-D}-2\log({\tfrac{10}{7}e^{-D-Be^{-D}}})+\log(e^{-D}+e^{-D})+\log(e^{-D})\\=
&10Be^{-D}-\log({\tfrac{100}{49}})+\log(2)\le -\frac{1}{2}\log(\tfrac{50}{49})<0,
\end{split}
\]
where the penultimate inequality follows from the choice of $D$ in Step 0. This shows that 
$ y_k^j$ does not belong to $A_k$, and hence,  by monotonicity (namely Step 3),  concludes the proof of Step 6.

\medskip



\smallskip
\noindent {\it Step 7: Description of the right double tail intervals.}
We now describe the  part of the \emph{double} tail, which corresponds to closest visits from the right (which is shaded by waves in Figure~\ref{fig:tailsets}). Let $\hat v_k\in(\hat{a}, \hat{b})$ be the point  in the  base $(\hat{a}, \hat{b}):= \tilde I_{d+1}^k$ of the tower $\tilde{\mathcal T}_{d+1}^k$ such that
\begin{equation}\label{eq: othersidesymmetry}
S_{2q_k}(f)(\hat v_k)=2t_k+2a_k=S_{2q_k}(f)(v_k)
\end{equation}
and let us define $\hat{D}$ so that
\begin{equation} \label{def:hatD}
\hat D=- \log q_k ( \hat a- T^{q_k}\hat v_k) \quad \Leftrightarrow \quad  \hat a-T^{q_k} \hat v_k=\frac{e^{-\hat D}}{q_k}  .
\end{equation} 
Since we are under the assumption that $v_k=v_k^-$ (see Step 1), and by monotonicity (Step 2) we have that
\begin{equation}\label{eq: D-tilde D relation}
	\hat D \le D\Leftrightarrow e^{-\hat D}\ge e^{- D},
\end{equation}
namely the point $T^{q_k}\hat{v}_k$ (which is chosen to have  distance $e^{-\hat D}/q_k$ from the closest singularity of the run) is \emph{further} from the singularities than $v_k$ (which has distance  $e^{-D}/q_k$).

\smallskip
Consider now the intervals $(\hat a_j,\hat 
b_j):=T^{-j}(\hat a_0,\hat b_0)$ defined in Step 2, which contain the right tail intervals. 
Similarly as before, let $\hat 
v^j_k\in(\hat a_j,\hat b_j)$ be such that
\begin{equation}\label{eq: deftildezkj}
	S_{2q_k}(f)(\hat v^j_k)= 2t_k+2a_k =S_{2q_k}(f)(\hat{v}_k)=S_{2q_k}(f)(v_k),
\end{equation}
and  (analogously to the definition 
\eqref{eq: defDjk} of $D^j_k$) let us then define:
\begin{equation}\label{eq: defhatDjk}
\hat D^{j}_k:=-\log(T^{q_k}\hat v^j_k-\hat a_j)-\log(q_k) \quad \Leftrightarrow \quad 
T^{q_k}\hat v^j_k-\hat a_j=\frac{e^{-\hat D^j_k}}{q_k}  
	. 
\end{equation}
Similarly as in \eqref{eq: trivialDwithD}, we have
\begin{equation}
	\hat D-2B\le \hat D_k^j\le \hat D+2B.
\end{equation}
Moreover, in view of \eqref{eq: othersidesymmetry}, by proceeding similarly as in 
\eqref{eq: DcompareswithD}, we get that
\begin{equation}\label{eq: tildeDcompareswithD}
	|\hat D^j_k-\hat D|\le 5B{e^{-D+2B}}.
\end{equation}
We now claim that we have 
	\begin{equation}\label{eq: intervals on the right}
	\left[\hat{a}_k, \hat v_k^j\right) \cap M_k\subset B_k\ \text{ 
		for every } 
	j\in\{0,\ldots,q_k-1-(\overline{\ell}_k-\underline{\ell}_k)\},\\
\end{equation}
i.e.~also the right double tail consists of intervals.\footnote{{Notice that here we cannot use directly a monotonicity argument, since the Birkhoff sum $S_{2q_k} f $ is not necessarily continuous on $(\hat{a}_k, \hat{b_k})$, since $T^i((\hat{a}_k, \hat{b_k}))$ can hit the cramped discontinuities for some $q_k\leq i\leq 2q_k$. It turns out that one can still use monotonicity arguments, but one has to split the Birkhoff sums $S_{2q_k} f (x)$ in two \emph{halves}, namely $S_{q_k} f (x)$ and $S_{q_k} f (T^{q_k}(x))$, and use different arguments and matchings on each part.}}
Indeed, fix $0\leq j \leq q_k-1-(\overline{\ell}_k-\underline{\ell}_k)$ and consider $x\in \left[\hat{a}_k, \hat v_k^j\right)\subset [\hat{a}_j, \hat{b}_j)$.   
We consider three separate cases (illustrated in Figure~\ref{fig:monotonicity}), according to the relative position of the iterate $T^{q_k}(x)$ with respect to the singularities.\footnote{Remark that $T^{q_k}(x)= x-q_k \Delta_k < x$; if $T^{q_k}(x)>\hat{a}_j$ (Case $1$), then $\mathcal{O}(T^{q_k}(x), q_k)$ is still to the right of the cramped singularities); otherwise, $T^{q_k}(x)$ belongs to the floor at the same height of $(\hat{a}_j,\hat{b}_j)$, but on the other \emph{side} of the singularities, namely $(a_j, b_j)$. In this case  $\mathcal{O}(T^{q_k}(x), q_k)$ is still to the right of the cramped singularities (see Cases 2 and 3).}
\smallskip

\begin{figure}
	\includegraphics[scale=1.5]{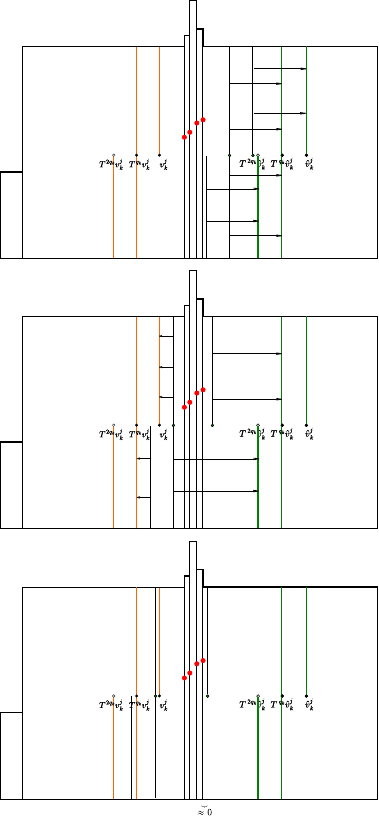}
	\caption{\label{fig:monotonicity} {Illustration of the  three Cases considered in Step 7 of the proof. Vertical (coloured) lines represent orbit segments and arrows indicate which orbit segments are matched and compared using different monotonicity arguments. The top figure illustrates {\it Case 1}, the next (middle) figure illustrates {\it Case 2}. The last figure (bottom) illustrates {\it Case 3}.}}


\end{figure}

\noindent {\it Case 1 (second run also  from the right)}.  Assume first that $T^{q_k}(x)$ still belongs to $(\hat{a}_j, \hat{b}_j)$.  Then one can show that $x\in B_k$ by exploiting the 
monotonicity of $S_{q_k}(f)$ and direct comparison of $S_{q_k}(f)(x)$ with 
$S_{q_k}(f)(\hat v_k^j)$ and $S_{q_k}(f)(T^{q_k}x)$ with 
$S_{q_k}(f)(T^{q_k}\hat v_k^j)$, which gives that 
$$
S_{2q_k}(f)(x)= S_{q_k}(f)(x) +S_{q_k}(f)(T^{q_k}(x)) >    S_{q_k}(f)(\hat v_k^j)+
S_{q_k}(f)(T^{q_k}(\hat v_k^j))= S_{2q_k}(f)(\hat v_k^j)
$$
 (see Figure~\ref{fig:monotonicity} {{top configuration, which shows that the two parts of the Birkhoff sum (in black) can be \emph{slided} to the left to match the other two parts (in green)}.} 
}

\smallskip
{Assume now that $T^{q_k}(v_k^j)\notin  (\hat{a}_j, \hat{b}_j)$. Then, since $x\in M_k$ (so that it cannot belong to the small towers $\mathcal{T}^k_3\cup \cdots \cup \mathcal{T}^k_{d} $) it has to belong to the floor of $\mathcal{T}^k_{2}$ at the same height than $(\hat{a}_j, \hat{b}_j)$, namely $(a_j,b_j)$. We consider two subcases (Case 2 and Case 3).
\smallskip

\noindent {\it Case 2 (second run to the left, but in the double tail)}.
If 
{$\hat{a}_{j}\le x \le \hat v_k^j$ and  $v_k^{j}\le T^{q_k} x\le b_{j}$,} where $v_k^{j}\in ({a}_j,\hat{b}_j)$ is the starting point of the double tail. Then, 
by two different monotonicity arguments, namely since $x< \hat{v}^j_k$ and $S_{q_k}(f)$ is decreasing in $(\hat{a}_j, \hat{b}_j)$, while $T^{q_k}x> {v}^j_k$ and $S_{q_k}(f)$ is increasing in $({a}_j, {b}_j)$ ({refer to Figure~\ref{fig:monotonicity}, middle}), we get: 
\[
\begin{split}
	&S_{q_k}(f)(x)\ge S_{q_k}(f)({\hat v_k^j})\ge\frac{1}{2}(S_{2q_k}(f)(v_k)-
	2a_k), \\ &S_{q_k}(f)(T^{q_k}x)\ge S_{q_k}f(v_k^{j})\ge\frac{1}{2}(S_{2q_k}f(v_k)-
	2a_k).
\end{split}
\]
Hence, adding these two estimates up,} we proved that  $x\in B_k$ in this case as well.

\smallskip
\noindent {\it Case 3 (second run to the left and outside of the double tail)}. If we are neither in Case 1, nor in Case 2, that is $x$ does not satisfy any of the considered conditions, then we claim that either 
$$|x-b_{j}|\le\frac{ 12Be^{-D+2B}\cdot e^{-D}}{q_k}, \quad \text{or} \quad |T^{q_k}x-b_{j}|\le\frac{ 12Be^{-D+2B}\cdot e^{-D}}{q_k}.$$  
\noindent { This means that 
in this case, the orbit has to pass very close to singularities (as illustrated in  the last picture in Figure~\ref{fig:monotonicity}, bottom).  To formally prove the claim, notice that,} by construction we have
$\Delta_k\in C=[c_0,c_1]$, 
by \eqref{eq: DcompareswithD} and \eqref{eq: tildeDcompareswithD}, we have
\begin{equation}\label{eq: allareclose}
	\max\{|\Delta_k-e^{-D}|,|\Delta_k-e^{-D_k^{j}}|\}\le 6Be^{-D+2B}\cdot e^{-D}.
\end{equation}
Assume that $|x-b_{j}|\le 12Be^{-D+2B}\cdot e^{-D}\cdot q_k^{-1}$, the other case is considered analogously. Since $b_{j}$ is a preimage of a discontinuity, the closest visit of the orbit $\mathcal O_T(x,q_k)$  has distance at most 
${(12Be^{-D+2B}\cdot e^{-D})}/{q_k}$. 
Thus, by using the $B$-linear trimmed derivative bounds (see Definition~\ref{def: derivativebound}) 
 and separating the closest visit from the remaining trimmed sum, we get $S_{q_k}(f)(x)-a_k\ge 2D-3B$. Note also, that the closest visit of the orbit $\mathcal O_T(T^{q_k}(x),q_k)$ is at most ${(12Be^{-D+2B}\cdot e^{-D}+\Delta_k)}/{q_k}$. Hence $S_{q_k}(f)(x)-a_k\ge D-3B$. Thus, by the choice of $D$, we have
\begin{equation}\label{eq: superclosevisit}
S_{2q_k}(f)(x)-2a_k\ge 3D-6B>2D+2B,
\end{equation}
which by \eqref{eq: boundonBS} implies that $x\in B_k$ and concludes the proof of the claim and of this Step in this case.

\medskip
{\noindent {\it Step 8: Estimate of  the right single tail intervals.}
We now describe the second part of the single tail, which corresponds to the 
\emph{right}-hand side visits. 	We claim that 
	$$
	(\hat a_j,\hat b_j)\cap A_k\subset \left(\hat a_j,\hat y_k^j \right), \quad \text{for} \ j\in \{ 0, \cdots, q_k-1-(\overline{\ell}-\underline{\ell})\},
	$$
	where $\hat y_k^j\in (\hat a_j,\hat b_j)$ is a point satisfying
	\begin{equation}\label{eq: deftildeykj}
	|\hat y_k^j-\hat a_j|=\Big|\hat v_k^j-\frac{1}{2}\cdot\frac{\Delta_k}{q_k}-\hat a_j\Big|(1+10Be^{-D}).
	\end{equation}
	We will show that $\hat y_k^j$ does \emph{not} belong to $A_k$. By \eqref{eq: residuasmall}, Jensen's inequality and Mean Value Theorem we have 
	\[
	\begin{split}
	2S_{q_k}&(\hat y_k^j)-S_{2q_k}(\hat v_k^j)\\
	&\le 2\tilde S_{q_k}(f)(\hat y_k^j)-\tilde S_{q_k}(f)(\hat v_k^j)-\tilde S_{q_k}(f)(T^{q_k}\hat v_k^j)\\
	&\quad-2\log (e^{-\hat D_k^j}+\frac{1}{2}\cdot {\Delta_k})-2\log(1+10Be^{-D})+\log (e^{-\hat D_k^j})+\log (e^{-\hat D_k^j}+\Delta_k)+4\xi_k\\
	&\le 2\tilde S_{q_k}(f)(\hat y_k^j)-\tilde S_{q_k}(f)(T^{-2q_k}\hat v_k^j)-\tilde S_{q_k}(f)(T^{-q_k}\hat v_k^j)-2\log(1+10Be^{-D})+4\xi_k\\
	&\le 6Be^{-D}-10Be^{-D}\le 0
	\end{split}
	\]
	Hence (since $\hat y_k^j$ does not belong to $A_k$) we have the inclusion 
	$$\bigcup_{j=0}^{q_k-1-(\overline{\ell}_k-\underline{\ell}_k)}(\hat a_j,\hat b_j)\cap A_k\subset \bigcup_{j=0}^{q_k-1-(\overline{\ell}_k-\underline{\ell}_k)}T^{-q_k}\left(\hat a_j,\hat y_k^j \right).$$  By combining this with Step 6, we get that 
	\begin{equation}\label{eq: singletailinclusion}
	A_k\cap\left(\bigcup_{j=0}^{{q_k-1-(\overline\ell_k-\underline\ell_k)}}(a_j,b_j) \cup \bigcup_{j=0}^{q_k-1-(\overline\ell_k-\underline\ell_k)}(\hat a_j,\hat b_j)\right)\subseteq \bigcup_{j=0}^{q_k-1-(\overline{\ell}_k-\underline{\ell}_k)}\left(y_k^{j'},\hat y_k^j\right),
	\end{equation}
where $j'=q_k-1-(\overline\ell_k-\underline\ell_k)-j$.
  }

\medskip

\medskip
\noindent {\it Step 9: Contribution to the tails of orbits with   incomplete runs.} 
Let us now prove  an \emph{upper} bound on the measure of  the set $A_k$ and \emph{lower} bound for the set $B_k$, one needs to consider also points which only make \emph{incomplete} runs.

In addition to the contribution from points whose orbits contain  a full \emph{run} (see Step 2) 
 The initial points of these orbits belong to  the sets 
$$\bigcup_{j=1}^{\overline{\ell}_k-\underline{\ell}_k}T^j(a_0,b_0)\ \text{ and} \  {\bigcup_{j=1}^{\overline{\ell}_k-\underline{\ell}_k} T^j(\hat a_{0},\hat b_{0}).}$$

 Let $j\in\{1,\ldots,\overline{\ell}_k-\underline{\ell}_k\}$ and let $(a,b):=T^j(a_0,b_0)$. In view of \eqref{eq: ynotintail}, we have that $y_k^0=b_0-\frac{\frac{10}{7}e^{-D_k^0}}{q_k}\notin A_k$. We claim that 
$	S_{q_k}(f)(T^{j}y_k^0)-S_{q_k}(f)(y_k^0)<0$, 
	and thus 
	\begin{equation}\label{eq: notvisitingall_1}
		A_k\cap T^j(a_0,b_0)\subset (T^{j}y_k^0,T^{j}b_0).
	\end{equation}
	Indeed, if $r\in\{1,\ldots,d-1\}$ is such that for $i\in\{0,\ldots,r\}$, we have $m_i^{-}(T^{j}y_k^0,q_k)=m_i^{-}(y_k^0,q_k)+\frac{\Delta_k}{q_k}$,  then, since $\Delta_k\in C$ and in view of the trimmed derivative $B$-bounds (see Definition~\ref{def: derivativebound}), 
	 we get 
	\[
	\begin{split}
		S_{q_k}(f)(T^{j}y_k^0)-S_{q_k}(f)(y_k^0)
		&\le B\Delta_k-\sum_{i=1}^{d-1}C_i^-\log(m_p^{-}(T^{j}y_k^0,q_k))+\sum_{i=1}^{d-1}C_i^-\log(m_p^{-}(y_k^0,q_k))\\
		&\le Be^{-D}+\sum_{i=0,\ldots,r}C_i^-\log\left(\frac{\frac{10}{7}e^{-D_k^0}}{\frac{10}{7}e^{-D_k^0}+\Delta_k}\right)\\
		&\le Be^{-D}+C_{min}\cdot\log\left(\frac{\frac{10}{7}e^{-D_k^0}}{\frac{16}{7}e^{-D_k^0}} \right)<0,
	\end{split}
	\]
	where the penultimate inequality follows from the choice of $D$, the definition \eqref{Cmin} of $C_{min}$ and \eqref{eq: allareclose}. 

\smallskip	
{	In a similar fashion, let us now prove that for $j=\{1, \dots ,(\overline{\ell}_k-\underline{\ell}_k)\}$ we have the inclusion
	\begin{equation}\label{eq: notvisitingall_2}
	A_k\cap T^j(\hat a_0,\hat b_0)\subset \left(T^{j}\hat a_0,T^{j}\hat y_k^0\right).
	\end{equation}
To see this, let again $r\in \{1,\ldots,d-1\}$ be such that for $i\in\{1,\ldots,r\}$, we have $m_i^{+}(T^{j}\hat y_k^0,q_k)=m_i^{+}(\hat y_k^0,q_k)+\frac{\Delta_k}{q_k}$. Using \eqref{eq: trivialDwithD}, we have
\[
\begin{split}
	S_{q_k}(f)(T^{j}\hat y_k^0)&-S_{q_k}(f)(\hat y_k^0)\le Be^{-D}-\sum_{i=1}^{d}C_i^+\log(m_p^{+}(T^{j}\hat y_k^0,q_k))+\sum_{i=1}^{d}C_i^+\log(m_p^{+}(\hat y_k^0,q_k))\\
	&\le Be^{-D}+\sum_{i=1,\ldots,r}C_i^+\log\left(\frac{(e^{-\hat D_k^0}+\frac{3}{2}\Delta_k)(1+10Be^{-D})}{(e^{-\hat D_k^0}+\frac{3}{2}\Delta_k)(1+10Be^{-D})+\Delta_k}\right)\\
	&\le Be^{-D}+C_{min}\ \log\left(\frac{e^{-D+2B}+\frac{3}{2}e^{-D}}{e^{-D+2B}+2e^{-D}} \right)\\
	&= Be^{-D}+C_{min}\ \log \left(1-\frac{1}{e^{2B}+4}  \right)\le Be^{-D}-C_{min}\ e^{-2B}<0,
\end{split}
\]
which shows \eqref{eq: notvisitingall_2}.
} 
Finally, we estimate the contribution of incomplete  runs to the double tail is also non-trivial, to show that it is comparable with the contribution to the single tail. More precisely, we show that for $\hat v_k^0$ with $|T^{q_k}\hat v_k^0-\hat a_0|=\frac{e^{-\hat D_k^0}}{q_k}$ and 
$j\in\{1,  \ldots, (\overline{\ell}_k-\underline{\ell}_k)\}$ we have
\begin{equation}\label{eq: notvisitingall_3}
	(T^j\hat a_0,T^j(\hat v_k^0))\cap M_k\subset B_k.
\end{equation}
To prove \eqref{eq: notvisitingall_3}, for a fixed {$j\in\{1,\ldots, (\overline{\ell}_k-\underline{\ell}_k)\}$, we need to show that $S_{2q_k}(f)(T^j\hat v_k^0))\ge S_{2q_k}(f)(\hat v_k^0)$,} since the point $\hat v_k^0$ marks the right-hand side endpoint of the tail $B_k$ (see Step 7). Let $r\in \{1,\ldots,d-1\}$ be such that for $i\in\{1,\ldots,r\}$, we have $m_i^{+}(T^{j}(\hat v_k^0),q_k)=m_i^{+}(T^{2q_k}\hat v_k^0,q_k)-\frac{2\Delta_k}{q_k}$.  Note that the rest of closest visits cancels out i.e. the points which are the closest visit or the second closest visit to a singularity in the orbits $\mathcal O(T^j(T^{-q_k}\hat v_k^0),2q_k)$ and $\mathcal O(T^{-2q_k}\hat v_k^0,2q_k)$ are the same, except the $r$ points pointed above. Thus, by using again the $B$-linear trimmed derivative bounds (see Definition \ref{def: derivativebound}), we get 
\[
\begin{split}
S_{2q_k}(f)(T^j({\hat v_k^0 })&- S_{2q_k}(f)({\hat v_k^0})\\
&\ge -2Be^{-D}-\sum_{i=1}^{d}C_p^+\log(m_p^{-}(T^{j}(\hat v_k^0),q_k))+\sum_{i=1}^{d}C_p^+\log(m_p^{-}(\hat v_k^0,q_k))\\
&\ge -2Be^{-D}+\sum_{i=1,\ldots,r}C_p^+\log\left(\frac{e^{-\hat D_k^0}+\Delta_k}{e^{-\hat D_k^0}-\Delta_k}\right).
\end{split}
\]
Thus, recalling the definition of $C_{min}$ in \eqref{Cmin},  we have
\[
\begin{split}
S_{2q_k}(f)(T^j({\hat v_k^0}))&- S_{2q_k}(f)({\hat v_k^0})\geq \\
&\ge -2Be^{-D}+C_{min}\cdot \log\left(\frac{e^{-D+2B}+e^{-D}}{e^{-D+2B}-e^{-D}} \right).\\
&= -2Be^{-D}+ C_{min}\ \cdot\log \left(1+\frac{2}{e^{2B}-1}  \right) \ge -2Be^{-D}+C_{min} \, \cdot e^{-2B}>0,
\end{split}
\]
which shows \eqref{eq: notvisitingall_3}.

We finish this Step by estimating from above the difference in measures of the single and double tails, on the set where the points realize the incomplete visits. By combining \eqref{eq: notvisitingall_1}, \eqref{eq: notvisitingall_2} and \eqref{eq: notvisitingall_3}, taking into account \eqref{eq: tausareclose}, we get
\begin{equation}\label{eq: tailsincomplete}
	Leb\left((A_k\setminus B_k)\cap \left(\bigcup_{j=1}^{\overline{\ell}_k-\underline{\ell}_k}T^j(a_0,b_0)\quad \cup\quad {\bigcup_{j=1}^{\overline{\ell}_k-\underline{\ell}_k}}T^j(\hat a_0,\hat b_0)  \right)\right)\le (\overline{\ell}_k-\underline{\ell}_k)\cdot \frac{3e^{-D}}{q_k}\le 3\rho e^{-D}.
\end{equation}

\medskip

\noindent {\it Step 10: Putting together all estimates.} 
Let us now combine the estimates from the previous Steps (Step 5 to Step 9), to estimate $Leb(B_k)-Leb(A_k)$. 
We divide the comparison into two parts, one corresponds to  \emph{complete} runs, which were studied in Steps 5 to 8, while the other comes from the \emph{incomplete} runs, with which we dealt in Step 9. First we deal with the latter. As a direct consequence of \eqref{eq: tailsincomplete} we get 
\begin{equation}\label{eq: finalestimate_incomplete}
	\begin{split}
\Bigg|Leb\,&\Bigg(B_k\cap \Bigg(\bigcup_{j=1}^{\overline{\ell}_k-\underline{\ell}_k}T^j(a_0,b_0) \cup {\bigcup_{j=1}^{\overline{\ell}_k-\underline{\ell}_k}}T^j(\hat a_0,\hat b_0)  \Bigg)\Bigg)\\
&-Leb\,\Bigg(A_k\cap \Bigg(\bigcup_{j=1}^{\overline{\ell}_k-\underline{\ell}_k}T^j(a_0,b_0) \cup{\bigcup_{j=1}^{\overline{\ell}_k-\underline{\ell}_k}}T^j(\hat a_0,\hat b_0)  \Bigg)\Bigg)\Bigg|\\
&\le Leb\,\Bigg((A_k\cap B_k)\cap \Bigg(\bigcup_{j=1}^{\overline{\ell}_k-\underline{\ell}_k}T^j(a_0,b_0) \cup{\bigcup_{j=1}^{\overline{\ell}_k-\underline{\ell}_k}}T^j(\hat a_0,\hat b_0)  \Bigg)\Bigg)\le 3\rho e^{-D}.
\end{split}
	\end{equation}
	
Now we go back to the contribution of the {orbits which undergo  complete runs.} In view of \eqref{eq: intervals on the right} and \eqref{eq: singletailinclusion} we get 
\begin{equation}\label{eq: finalestimate_complete_p1}
	\begin{split}
	Leb\,&\Bigg( B_k\cap\Bigg(\bigcup_{j=0}^{{q_k-1-(\overline\ell_k-\underline\ell_k)}}(a_j,b_j) \cup \bigcup_{j=0}^{q_k-1-(\overline\ell_k-\underline\ell_k)}(\hat a_j,\hat b_j)\Bigg)\Bigg)\\&-Leb\,\Bigg( A_k\cap\Bigg(\bigcup_{j=0}^{{q_k-1-(\overline\ell_k-\underline\ell_k)}}(a_j,b_j) \cup \bigcup_{j=0}^{q_k-1-(\overline\ell_k-\underline\ell_k)}(\hat a_j,\hat b_j)\Bigg)\Bigg)\\
	&\ge Leb\Bigg(\bigcup_{j=0}^{{q_k-1-(\overline\ell_k-\underline\ell_k)}}\big(v_k^{j},\hat v_k^j\big)\Bigg)-Leb\Bigg(\bigcup_{j=0}^{{q_k-1-(\overline\ell_k-\underline\ell_k)}}\big(y_k^{j},\hat y_k^j\big)\Bigg).
	\end{split}
\end{equation}
By the definitions of points $v_k^j,\hat v_k^j, y_k^j$ and $\hat y_k^j$ (see \eqref{eq: defzkj}, \eqref{eq: deftildezkj}, \eqref{eq: defykj} and \eqref{eq: deftildeykj}, respectively), we get
\begin{equation}\label{eq: finalestimate_complete_p2}
	\begin{split}
Leb&\Bigg(\bigcup_{j=0}^{{q_k-1-(\overline\ell_k-\underline\ell_k)}}\big(v_k^{j},\hat v_k^j\big)\Bigg)-Leb\Bigg(\bigcup_{j=0}^{{q_k-1-(\overline\ell_k-\underline\ell_k)}}\big(y_k^{j'},\hat y_k^j\big)\Bigg)\\
&=Leb\Bigg(\bigcup_{j=0}^{{q_k-1-(\overline\ell_k-\underline\ell_k)}}\Big(\hat y_k^j,\hat v_k^j+2\frac{\Delta_k}{q_k}\Big)\Bigg)-Leb\Bigg(\bigcup_{j=0}^{{q_k-1-(\overline\ell_k-\underline\ell_k)}}\big( y_k^j, v_k^j\big)\Bigg)
\end{split}
	\end{equation}
In view of definition of $y_k^j$ and $v_k^j$, as well as of \eqref{eq: DcompareswithD}, we get that 
\begin{equation}\label{eq: completebreakdown_left}
v_k^j-y_k^j=\frac{3}{7}\cdot \frac{e^{-D_k^j}}{q_k}= \frac{3}{7}\cdot \frac{e^{-D}}{q_k}\cdot e^{-D_k^j+D}\le 
\frac{3}{7}\cdot \frac{e^{-D}}{q_k}\cdot e^{5Be^{-D+2B}}\le \frac{3}{7}\cdot \frac{e^{-D}}{q_k}(1+6B^{-D+2B}).
\end{equation}
Moreover, in view of definition of $\hat y_k^j$ and $\hat v_k^j$, using again \eqref{eq: DcompareswithD} and \eqref{eq: D-tilde D relation}, we get 
\begin{equation}\label{eq: completebreakdown_right}
	\begin{split}
	\hat v_k^j-\hat y_k^j&=e^{-\hat D_k^j}+\frac{\Delta_k}{q_k}-\big(e^{-\hat D_k^j}+\frac{\Delta_k}{2q_k}\big)(1+10Be^{-D})
	=\frac{\Delta_k}{2q_k}-\frac{10Be^{-D}}{q_k}(e^{-D+2B}+2e^{-D})\\
	\ge &\frac{e^{-D}(1-\delta)}{2q_k}-\frac{11Be^{-D}}{q_k}e^{D+2B}\ge \frac{1}{2}\frac{e^{-D}}{q_k}(1-25Be^{-D+2B})
	\end{split}
\end{equation}
By applying \eqref{eq: completebreakdown_left} and \eqref{eq: completebreakdown_right} to \eqref{eq: finalestimate_complete_p2}, we get
\[
Leb\Bigg(\bigcup_{j=0}^{{q_k-1-(\overline\ell_k-\underline\ell_k)}}\big(v_k^{j},\hat v_k^j\big)\Bigg)-Leb\Bigg(\bigcup_{j=0}^{{q_k-1-(\overline\ell_k-\underline\ell_k)}}\big(y_k^{j},\hat y_k^j\big)\Bigg)
\ge \frac{1-\rho}{14} e^{-D}(1-16Be^{-D+2B})\ge \frac{1-\rho}{16}e^{-D},
\]
which by \eqref{eq: finalestimate_complete_p1}, \eqref{eq: finalestimate_complete_p2}, as well as \eqref{eq: finalestimate_incomplete} gives
\begin{equation}\label{eq: totalfinalestimate}
Leb(B_k)-Leb(A_k)\ge \frac{1-4\rho}{16}e^{-D}\ge \frac{e^{-D}}{20}.
\end{equation}
Since, in view of \eqref{eq: gammaestimate}, we have $\gamma\le 4\cdot 10^{-6}e^{-D}<\frac{1}{2}$,  by \eqref{eq: totalfinalestimate}, we get
\[
Leb(B_k)-Leb(A_k)\ge \gamma.
\]
which finishes the proof of the proposition.
	\end{proof}
\subsection{Distinguishing tails in the general case} \label{sec:distinguishedtailsgeneral}
We now deal with the case when the rescalings under consideration are the $K$ and $L$ rescalings,  with general $K, L$ positive integers.	We assume throughout  that $L>K$.  

\subsubsection{The $K$ and $L$ tail sets.}\label{sec:BKLsets}
Let $(M_k)_k$ be a sequence of good $L$-fold matchings, given by $(c,\delta)$-BT rigidity sets. 
We let, as in the previous sections:
\begin{itemize}
\item[-] $B$ be the constant  given by the linear control of trimmed derivatives (see Definition \ref{def: derivativebound});
\item[-]  $a_k$ be the centering constants given by $a_k:= S_{q_k}(f)(z_k)$ where $z_k$ is the reference point of the matching sets,
\item[-] $X_k=[0,1]\backslash M_k$ be the bad set for matchings. 
\item[-] $\Delta_k=q_k\, \tilde\lambda^k_1$ be the rescaled displacement,
\item[-] $t_k:= \frac{K }{L} S_{L q_k}(f)(v_k)-K a_k$, for $k\in \N$.
\end{itemize}
where $v_k$ marks the beginning of $B_k^L$ on the set $\tilde I_2^k$.
To deduce from the criterion on disjointness that $K$-th and $L$-th acceleration of our flow are disjoint, we need to consider the sets
\begin{align}\label{def:BKL}
B_k^L&:=\{x\in I\setminus X_k\mid S_{Lq_k}(f)(x)-La_k\ge \frac{L}{K}\, t_k:= S_{Lq_k}(f)(v_k)-La_k\},\\ \nonumber
B_k^K&:=\{x\in I\setminus X_k\mid S_{Kq_k}(f)(x)-Ka_k\ge t_k:= \frac{K}{L}S_{Lq_k}(f)(v_k)- Ka_k\}
\end{align}
and prove that asymptotically the measures of $B_k^L$ and $B_k^K$ are different. 
  
\smallskip
  The main estimate which allows to distinguish the tail sets is the following:
  {
  \begin{prop}[Distinguishing tails in the general case]\label{prop:generaltail} 
 { Given any $L\in \mathbb{N}$, there exists $\rho>0$ such that, for every  $f\in SymLog^2(T)$  and  for any $B>0$,}  
 there exists 
  $ \epsilon>0$ and $C=[c_0,c_1]\subset \mathbb{R}_+$, so that for any  
  sequence $(M_k)_k$ of {$(\gamma,B)$}-good $L$-fold $(z_k,q_k)-$matchings given a $(C,\epsilon, \rho)$-BT rigid presentation with
  ${\gamma< (L+1)\rho c_1+\epsilon }$ 
 we have that 
$$Leb(B_k^L)-Leb(B^K_k)\geq \gamma \quad \forall k\in \N,$$
where $B_k^K$ and $B_k^L$ are respectively the $L$ and $K$-tail sets defined above. Furthermore, 
 the sequence $(t_k)_k$  in the definition of the tail sets, namely $t_k= \frac{K }{L} S_{L q_k}(f)(v_k)-K S_{q_k}(f)(z_k)$, $k\in \mathbb{N}$,  is bounded.  
	\end{prop} }
\noindent	The proof of Proposition~\ref{prop:generaltail} is given in \S~\ref{sec:prooftailestimate}. We first prove to auxiliary results (in \S~\ref{sec:convexity} and  \S~\ref{sec:auxiliary} respectively). 
\subsubsection{A convexity inequality.}\label{sec:convexity}  The key estimate which is used to show that the tails can be distinguished is based on convexity. We first state and prove the following Lemma, which will provide the key step to distinguish the tails. 

\begin{lemma}[A convexity estimate]\label{lemma:convexity}
For any non-negative number $R\ge 0$ and for every two positive integers $L\geq K>0$, it holds that 
\begin{equation}\label{eq: mainsimplification}
\sum_{i=1}^K  \log\left(R+\frac{L-K}{2}+i\right)>  \frac{K}{L}\sum_{i=1}^L \log(R+i).
\end{equation}
\end{lemma}
\noindent The rest of this subsection is taken by the proof of this Lemma.

\smallskip
Note that the inequality \eqref{eq: mainsimplification} does not depend on  dynamics, it is just a functional inequality.  If $K=1$ then it is just standard Jensen's inequality, so we assume that $K\geq 2$. To prove 
that it is always satisfied for any $L\geq K>0$, 
we need the following version of Jensen's Inequality (see e.g.~Lemma 1 in \cite{KPW}). 
\begin{lemma}[Jensen inequality]\label{prop: Jensen}
	Let $a<b$ be two real numbers and consider two convex combinations
	\[
	\sum_{i=1}^mp_ix_i=\sum_{j=1}^nq_jy_j\text{ with }1=\sum_{i=1}^mp_i=\sum_{j=1}^nq_j
	\]
	for some $n,m\in\N$ and $x_1,\ldots,x_m\in [a,b]$, while $y_1,\ldots,y_n\notin [a,b]$.
	Then for any strictly concave function $f$ on $\R$ we have
	\[
	\sum_{i=1}^mp_if(x_i)> \sum_{j=1}^nq_jf(y_j).
	\]
\end{lemma}

\begin{proof}[Proof of Lemma~\ref{lemma:convexity}]
To prove that \eqref{eq: mainsimplification} is satisfied, consider two cases. 
\smallskip
 
\noindent {\it Case 1 (same parity)}
First assume that $K$ and $L$ are of the same parity. 
We have
\[
\begin{split}
\sum_{i=1}^K  &\log\left(R+\frac{L-K}{2}+i\right)-  \frac{K}{L}\sum_{i=1}^L\log(R+i)\\
&=\left(\frac{1}{K}-\frac{1}{L}\right)\sum_{i=1}^{K}\log\left(R+\frac{L-K}{2}+i\right)-\frac{1}{L}\left(\sum_{i=1}^{\frac{L-K}{2}}\log(R+i) +\sum_{i=\frac{L+K}{2}+1}^{L}\log(R+i)  \right)\\
&=\frac{L-K}{L}\left(\frac{1}{K}\sum_{i=1}^{K}\log\left(R+\frac{L-K}{2}+i\right) - \frac{1}{L-K}
\left(\sum_{i=1}^{\frac{L-K}{2}}\log(R+i) +\sum_{i=\frac{L+K}{2}+1}^{L}\log(R+i)  \right)   \right).
\end{split}
\]
Then, by applying Lemma \ref{prop: Jensen} to the following data:
\[
\begin{split}
&p_1=\ldots=p_{K}=\frac{1}{K};\quad x_i=R+\frac{L-K}{2}+i\text{ for }i=1,\ldots,K;\quad
q_1=\ldots=q_{L-K}=\frac{1}{L-K};\\
&y_j=R+j\text{ for }j=1,\ldots,\frac{L-K}{2}\ \text{ and }\  y_j=R+K+j\ \text{ for }\ j=\frac{L-K}{2}+1,\ldots,L-K,
\end{split}
\]
we get\eqref{eq: mainsimplification} in this case.

\smallskip

\noindent {\it Case 2 (opposite parity)} If $K$ and $L$ are not of the same parity, then we start by applying classical Jensen's Inequality to each of the logs separately and obtain
\[
\begin{split}
	\frac{1}{K} \sum_{i=1}^K &\log\left(R+\frac{L-K}{2}+i\right)
	> \sum_{i=0}^{K-1} \frac{1}{2K} \left(   \log\left(R+\frac{L-K+1}{2}+i\right) +  \log\left(R+\frac{L-K+1}{2}+i+1\right) \right)\\
	&= \frac{1}{2K}\log\left(R+\frac{L-K+1}{2}\right)+\frac{1}{K}\sum_{i=1}^{K-1}\log\left(R+\frac{L-K+1}{2}+i\right)+
	\frac{1}{2K}\log\left(R+\frac{L-K+1}{2}+K\right).
	\end{split}
\]
Then we get that 
\[
\begin{split}
	\sum_{i=1}^K  &\log\left(R+\frac{L-K}{2}+i\right)-  \frac{K}{L}\sum_{i=1}^L\log(R+i)\\
	&\ge\left(\frac{1}{2K}-\frac{1}{L}\right)\left(\log\left(R+\frac{L-K+1}{2}\right)+\log\left(R+\frac{L-K+1}{2}+K\right)\right)\\&+
	\left(\frac{1}{K}-\frac{1}{L}\right)\left(\log\left(R+\frac{L-K+1}{2}+1\right)+\ldots+\log\left(R+\frac{L-K+1}{2}+K-1\right)\right)\\
	&-\frac{1}{L}\left(\log(R+1)+\ldots+\log\left(R+\frac{L-K+1}{2}-1\right)+\log\left(R+\frac{L-K+1}{2}+K+1\right)+\ldots +\log(R+L)  \right)
	\end{split}
\]
{If $2K<L$, the by applying Lemma \ref{prop: Jensen} to the following data:
\[
\begin{split}
	&p_1=p_{K+1}=\frac{L-2K}{2K(L-K-1)}\text{ and }p_2=\ldots=p_{K}=\frac{L-K}{K(L-K-1)};\\ &x_i=R+\frac{L-K-1}{2}+i\text{ for }i=1,\ldots,K+1;\\
	&q_1=\ldots=q_{L-K-1}=\frac{1}{L-K-1};\\
	&y_j=R+j\text{ for }j=1,\ldots,\frac{L-K-1}{2}\ \text{ and }\  y_j=R+K+j+1\ \text{ for }\ j=\frac{L-K-1}{2}+1,\ldots,L-K-1,
\end{split}
\]
we get \eqref{eq: mainsimplification}.
If on the other hand $2K\ge L$, then we apply Lemma \ref{prop: Jensen} to
\[
\begin{split}
	&p_1=\ldots=p_{K-1}=\frac{1}{K-1};\quad x_i=R+\frac{L-K+1}{2}+i\text{ for }i=1,\ldots,K-1;\\
	&q_1=q_{L-K+1}=\frac{2K-L}{(L-K)(K-1)}\ \text{and}\ q_2=\ldots=q_{L-K}=\frac{K}{(L-K)(K-1)};\\
	&y_j=R+j\text{ for }j=1,\ldots,\frac{L-K+1}{2}\ \text{ and }\  y_j=R+K+j\ \text{ for }\ j=\frac{L-K+1}{2},\ldots,L-K,
\end{split}
\]
and obtain \eqref{eq: mainsimplification} in the remaining case}.
\end{proof}

\subsubsection{An auxiliary tail estimates.}\label{sec:auxiliary}
To deduce the estimate of measures of both  tail sets (which will be used in the proof of Proposition~\ref{prop:generaltail} in the next \S~\ref{sec:prooftailestimate})  
at the same time, we first prove the following estimate. 

For $N\in\N$ and $H>0$,  for every $k\in\N$ let $w_k\in(a,b)$ be such that $b-w_k=\frac{e^{-H}}{q_k}$ and   consider the following set:
\[
E^{H,N}_k:=\{x\in M_k \mid S_{Nq_k}(f)(x)-Na_k\ge S_{Nq_k}(f)(w_k)-N a_k\}.
\]
{Let $\hat w_k\in(\hat a,\hat b)$ be such that $S_{Nq_k}(f)(\hat w_k)=S_{Nq_k}(f)(w_k)$ and let $\hat H>0$ be defined as $T^{(N-1)q_k}\hat w_k-\hat a=\frac{e^{-\hat H}}{q_k}$.}
\begin{lemma}\label{lem: generalformoftails}
	If  $H$ is such that	 $3NB\leq D-3NB<H,\hat H\leq D$, then 
	\begin{equation}
		|Leb(E^{H,N}_{k})-(e^{-H}+{e^{-\hat H}}+(N-1)\Delta_k)|<(N+5)\rho e^{-D}.
	\end{equation}
\end{lemma}
\smallskip
\noindent The proof, that takes the rest of this subsection, largely follows the computations already done in the $K=1,L=2$ case. 

\begin{proof}
	First notice that, due to the choice of interval for $H$, as in Step 3 of the proof of Proposition~\ref{prop:distinguishedtailsL2}, all orbits under consideration lie in the monotonicity intervals of $S_{q_n}(f)$, by increasing $D$ if necessary, depending on $N$. Hence, in the reminder of the proof, we always assume that we are in a proper monotonicity interval. 
	
	For every $j\in\{0,\ldots,q_k-1\}$, define $w_k^j$, $\hat w_k^j$, $H_k^j$ and $\hat H_k^j$ analogously as in \eqref{eq: defzkj}, \eqref{eq: deftildezkj}, \eqref{eq: defDjk} and \eqref{eq: defhatDjk}.
		By proceeding similarly as in the proof of \eqref{eq: DcompareswithD} (and \eqref{eq: tildeDcompareswithD}) we get
		\begin{equation}\label{eq: HcompareswithH}
			|H_k^j-H|\le Be^{-D+2B}\quad\text{and}\quad|\hat H_k^j-\hat H|\le Be^{-D+2B}\quad\text{for every}\quad j\in\{0,\ldots,q_k-1\}.
		\end{equation}
		We will show that $E^H_k$ is essentially a union of intervals and the above expression shows that their left- and right-handside endpoints are very comparable.
		
		Due to the fact that $M_k$ is a good matching set, we know that for every $j$, the points $w_k^{j}$ and $\hat w_k^j$ are endpoints of an interval of continuity of $f$ of length $\frac{e^{-H_k^j}}{q_k}+\frac{e^{-\hat H_k^j}}{q_k}+\delta_k$, where { $\delta_k:=|\tilde I_3^k|+ \ldots+ |\tilde I_d^k|$}. We will show that $E_k^H$ consists essentially of intervals. 
		
		More precisely, let $x\in[w_k^{j},\hat w_k^j]\cap(I\setminus X_k)$. Then either $x\in[w_k^{j},b_{j}]$ or $x\in[\hat a_j, \hat w_k^j]$. If $x\in[w_k^{j},b_{j}]\cap (I\setminus X_k)$, then by monotonicity of $S_{q_k}(f)$, we have
		\[
		S_{Nq_k}(f)(x)\ge S_{Nq_k}(f)(w_k^{j})
		\]
		and thus $x\in E_k^{H,N}$. Similarly, if $x\in[\hat a_j+\frac{(N-1)\Delta_k}{q_k},\hat w_k^j]\cap M_k$, then again by monotonicity argument we have
		\[
		S_{Nq_k}(f)(x)\ge S_{Nq_k}(f)(\hat w_k^j)
		\] 
		and thus again $x\in E_k^{H,N}$. Let then $x\in[\hat a_j,\hat a_j+\tfrac{(N-1)\Delta_k}{q_k}]\cap M_k$, and assume that there exists $M\in\{1,\ldots,N-1\}$ such that $T^{Mq_k}(x)\in [w_k^{j},b_j]$ and $T^{(M-1)q_k}(x)\in[\hat a_j, \hat w_k^j]$. Then
		\[
		S_{Nq_k}(f)\left(T^{Mq_k}(x)\right)\ge S_{Nq_k}(f)(w_k^j).
		\]
		However, due to monotonicity of $S_{q_k}(f)$, we have 
		\[
		S_{q_k}(f)\left(T^{Mq_k}(x)\right)\ge S_{q_k}(f)\left(T^{q_k}\circ T^{Mq_k}(x)\right)\ge\ldots
		\ge S_{q_k}(f)\left(T^{(N-M-1)q_k}\circ T^{Mq_k}(x)\right).
		\]
		Thus 
		\begin{equation}\label{eq: partialestimatept1}
			S_{(N-M)q_k}(f)\left(T^{Mq_k}(x)\right)\ge \frac{N-M}{N}S_{Nq_k}(w_k^j).
		\end{equation}
		Analogously
		\begin{equation}\label{eq: partialestimatept2}
			S_{Mq_k}(f)(x)\ge \frac{M}{N}S_{Nq_k}(T^{-Nq_k}\hat w_k^j).
		\end{equation}
		By summing up \eqref{eq: partialestimatept1} and \eqref{eq: partialestimatept2} we get that 
		\[
		S_{Nq_k}(f)(x)\ge S_{Nq_k}(w_k^j)
		\]
		and thus $x\in E^{H,N}_k$.
		
		It remains to consider the case when $x\in[\hat a_j,\hat a_j+\tfrac{(N-1)\Delta_k}{q_k}]\cap M_k$ but there is no $M\in\{0,\ldots,N-1\}$ for which $T^{Mq_k}(x)\in [w_k^{j},b_j]$ and $T^{(M-1)q_k}(x)\in[\hat a_j, \hat w_k^j]$. This may happen for example if $H\approx \hat H\approx D$. Then, since as in \eqref{eq: allareclose}, by \eqref{eq: HcompareswithH} we have 
		\begin{equation}\label{eq: allareclose_general}
			\max\{|\Delta_k-e^{-D}|,|e^{-H}-e^{-H_k^{j}}|,|e^{-\hat H}-e^{-\hat H_k^j}|\}\le 4Be^{-D+2B}\cdot e^{-D},
		\end{equation}
		if $x$ does not satisfy any of the considered conditions, then for some $M\in\{0,\ldots,N-1\}$ it holds that  $|T^{Mq_k}x-b_{j'}|\le 8Be^{-D+2B}\cdot e^{-D}$. Then, analogously as in the proof of \eqref{eq: superclosevisit}, since $H>3NB$, we have
		\[
		S_{Nq_k}(f)(x)-Na_k\ge (N+1)H-2NB> NH+NB,
		\]
		which implies that $x\in E_k^{H,N}$. To sum up we have 
		\begin{equation}\label{eq: E_k^{H,N} with all visits}
			\begin{split}
				E_k^{H,N}&\cap\left(\bigcup_{j=0}^{q_k-1-(\overline{\ell}_k-\underline{\ell}_k)}T^{-j}(a_0,b_0) \cup \bigcup_{j=0}^{(q_k-1-(\overline{\ell}_k-\underline{\ell}_k)}T^{-j}(\hat a_0,\hat b_0)\right)\cap M_k\\
				&= \bigcup_{j=0}^{q_k-(\overline{\ell}_k-\underline{\ell}_k)-1}[w_k^{j},\hat w_k^j]\cap M_k
			\end{split}
		\end{equation}
		
		We have described the set $E_k^{H,N}$ outside the set $\bigcup_{j=1}^{\overline\ell_k-\underline\ell_k}T^j(a_0,b_0)$ and $\bigcup_{j=1}^{-(\overline\ell_k-\underline\ell_k)}T^j(\hat a_0,\hat b_0)$. Let us denote 
		\[
		V_k:=E_k^{H,N}\cap \left(\bigcup_{j=1}^{\overline{\ell}_k-\underline{\ell}_k}T^j(a_0,b_0) \cup \bigcup_{j=1}^{\overline\ell_k-\underline\ell_k}T^j(\hat a_0,\hat b_0)\right)\cap M_k.
		\]
		Then, analogously as in \eqref{eq: notvisitingall_1} and \eqref{eq: notvisitingall_2} we have
		\begin{equation}\label{eq: V_k is contained}
			V_k\subseteq \bigcup_{j=1}^{\overline{\ell}_k-\underline{\ell}_k}(T^j(w_k^0),T^jb_0) \cup \bigcup_{j=1}^{\overline\ell_k-\underline\ell_k}(T^j\hat a_0,T^j(\hat w_k^0)),
		\end{equation}
		{while similarly as in \eqref{eq: notvisitingall_3} we have
		\begin{equation}\label{eq: V_k contains}
		\left(\bigcup_{j=1}^{\overline{\ell}_k-\underline{\ell}_k}(T^j(T^{-q_k}w_k^0),T^jb_0) \cup \bigcup_{j=-(\overline{\ell}_k-\underline{\ell}_k)}^{-1}(T^j\hat a_0,T^j(T^{q_k}\hat w_k^0))\right)\cap M_k\subseteq V_k.
		\end{equation}
		We note that in \eqref{eq: V_k contains}, the left-hand side of the tail did not appear in the formula since the measure of this part was small enough so that it could be considered as an error. However, although \eqref{eq: V_k contains} is more general, the proof goes along the same lines as the proof of \eqref{eq: notvisitingall_3}.}
		
		To summarize, in view of \eqref{eq: E_k^{H,N} with all visits}, using \eqref{eq: tausareclose}, \eqref{eq: allareclose_general} and \eqref{eq: V_k is contained}, we have
		\[
		\begin{split}
			Leb(E_k^{H,N})&= Leb \left(\bigcup_{j=0}^{q_k-(\overline\ell_k-\underline\ell_k)-1}[w_k^{j},\hat w_k^j]\cap M_k\right) + Leb(V_k)\\
			&\le q_k\cdot \frac{1}{q_k}\cdot ((1+4Be^{-D+2B})(e^{-H}+e^{-\hat H})+(N-1)\Delta_k)\\
			&\le (e^{-H}+e^{-\hat H}+(N-1)\Delta_k)+\rho e^{-D},
		\end{split}
		\]
		where the last inequality follows from the fact that $H,\hat H\ge 3NB$.
		
		{On the other hand, using \eqref{eq: tausareclose}, \eqref{eq: allareclose_general}, and \eqref{eq: V_k contains} we have
		\[
	\begin{split}
		Leb(E_k^{H,N})&= Leb \left(\bigcup_{j=0}^{q_k-(\overline\ell_k-\underline\ell_k)-1}[w_k^{j},\hat w_k^j
	]\cap M_k\right) + Leb(V_k)\\
		&\ge q_k\cdot \frac{1}{q_k}\cdot ((1-4Be^{-D+2B})(e^{-H}+e^{-\hat H})+(N-1)\Delta_k)- (\overline{\ell}_k-\underline{\ell}_k)2\frac{\Delta_k}{q_k}-Leb(X_k)  \\
		&\ge (e^{-H}+e^{-\hat H}+(N-1)\Delta_k)-(N+5)\rho e^{-D},
	\end{split}
		\]
	}Thus we get 
		\[
		-(N+5)\rho e^{-D}\le|Leb(E_k^{H,N})-(e^{-H}+e^{-\hat H}+(N-1)\Delta_k)|\le \rho e^{-D},
		\]
		which finishes the proof of the lemma.
		\end{proof}

\subsubsection{The tail estimates in the general case}\label{sec:prooftailestimate}
We will now sketch the proof of the general tail estimates  (i.e.~Proposition~\ref{prop:generaltail}), following the lines of the proof in the special case $K=1$, $L=2$. We highlight the arguments which are specific to the general case and leave the details of other steps to the reader.

\begin{proof}[Sketch of proof of Proposition~\ref{prop:generaltail}]
For any $k\in\N$, consider the $K,L$ tail sets $B_k^K$ and $ B_k^L$ defined in \S~\ref{sec:BKLsets}, see \eqref{def:BKL}. 
As a defining point $v_k$ for $B_k^L$, we take it so that it satisfies $b-v_k=\frac{e^{-D}}{q_k}$. {Then in view of Lemma \ref{lem: generalformoftails}, we have that $Leb(B_k)\approx e^{-D}+e^{-\hat D}+(L-1)e^{-D}$, where $\hat D$ is such that $T^{(L-1)q_k}\hat v_k-\hat a=\frac{e^{-\hat D}}{q_k}$ and $v_k$ marks the beginning of $B_k^L$ on the right-hand side. We assume from now on that $\hat D\le D$. If this inequality is the other way around, we first define $\hat v_k$ as the point whose distance is $\frac{e^{-D}}{q_k}$ from $\hat a$ and then proceed with the proof symmetrically.

Analogously, again by Lemma \ref{lem: generalformoftails}, we have that $Leb(A_k)\approx e^{-H}+e^{-\hat H}+(K-1)e^{-D}$, where $H$ and $\hat H$ are the constants defining $A_k$ from the left and right, respectively.
 Thus we get that 
\begin{equation}\label{eq: approx_difference}
	Leb(B_k)-Leb(A_k)\approx  e^{-D}+e^{-\hat D}+(L-1)e^{-D}-\left( e^{-H}+e^{-\hat H}+(K-1)e^{-D}\right)
\end{equation} }

As in \S~\ref{sec:strategy}, to determine whether a point $x$ belongs or not to $B_k^L$, as a first approximation,   we consider the contributions of the closest visits of the orbits to the closest singularity (which, since $C_f=1$ and the singularities are cramped, gives a good approximation of the contribution of all singularities).
{Taking this into consideration, by using Lemma \ref{lemma:convexity} we compute
\[
\begin{split}
	-\log&\left(e^{-\hat D}+\frac{L-K}{2}e^{-D}\right)-\log\left(e^{-\hat D}+\left(\frac{L-K}{2}+1\right)e^{-D}\right)-\ldots-\log\left(e^{-\hat D}+\left(\frac{L-K}{2}+K-1\right)e^{-D}\right)\\
	&< \frac{K}{L}\left(-\log\left(e^{-\hat D}\right)-\log\left(e^{-\hat D}+e^{-D}\right)-\ldots-\log\left(e^{-\hat D}+(L-1)e^{-D}\right)\right),
\end{split}
\]
which implies that $e^{-\tilde H}<e^{-\hat D}+\frac{L-K}{2}e^{-D}$ and thus 
\begin{equation}\label{eq: approx_difference_hat}
	e^{-\hat D}+(L-1)e^{-D}-\left(e^{-\hat H}+(K-1)e^{-D}\right)\ge \frac{L-K}{2}e^{-D}.
\end{equation}
On the other hand
\[
\begin{split}
	-\log&\left(\left(\frac{L-K}{2}+1\right)e^{-D}\right)-\log\left(\left(\frac{L-K}{2}+2\right)e^{-D}\right)-\ldots-\log\left(\left(\frac{L-K}{2}+K\right)e^{-D}\right)\\
	&< \frac{K}{L}\left(-\log\left(e^{-D}\right)-\log\left(2e^{-D}\right)-\ldots-\log\left(Le^{-D}\right)\right),
\end{split}
\]
again by Lemma~\ref{lemma:convexity}.  One can then show  that there exists a number $\alpha:=\alpha(K,L)$, with $\alpha<\frac{L-K}{2}+1$ (in the case $K=1,\,L=2$ we took $\alpha=\frac{10}{7}$),  such that 
\[
\begin{split}
	-\log&\left(\alpha e^{-D}\right)-\log\left(\left(\alpha+1\right)e^{-D}\right)-\ldots-\log\left(\left(\alpha+K-1\right)e^{-D}\right)\\
	&< \frac{K}{L}\left(-\log\left(e^{-D}\right)-\log\left(2e^{-D}\right)-\ldots-\log\left(Le^{-D}\right)\right),
\end{split}
\]
and at the same time, by picking $\rho:=\rho(K,L)$ sufficiently small, we can guarantee, that 
\begin{equation}\label{eq: approx_difference_nohat}
	e^{-H}\le \alpha e^{-D}.
	\end{equation}
By applying \eqref{eq: approx_difference_hat} and \eqref{eq: approx_difference_nohat} to \eqref{eq: approx_difference}, we get
\[
Leb(B_k^L)-Leb(B_k^K)\approx \left(\frac{L-K}{2}-(\alpha-1)\right)e^{-D},
\]
which is {a \emph{large} multiple of $e^{-D}$} and depends essentially only on $K$ and $L$.} On the other hand, in view of \eqref{eq: tausareclose}, we have that the number 
\[
Leb(X_k)=Leb([0,1)\setminus M_k)\le \gamma,
\]
with $\gamma:=\rho(L+1)e^{-D}$ is a very \emph{small} multiplicity of $e^{-D}$. Hence, since $\rho$ was taken  sufficiently small, we have that
\begin{equation}\label{eq: criterionassumptionsatisfied}
	Leb(B_k^L)-Leb(B_k^K)\ge \gamma.
\end{equation}
\noindent The remaining details of the proof, needed to  formalize the approximations made in the above calculation (as in the proof for $K=1$, $L=2$)
including the choice of $\epsilon$ and $\delta$, go similarly as in the case $K=1$, $L=2$.

\end{proof}

	\subsection{Final arguments}\label{sec:final}
	We can now conclude the proof of the main disjointness result, i.e.~Theorem~\ref{thm:main}.
	{	\begin{proof}[Proof of Theorem~\ref{thm:main}]
Let ${G}$ be the full measure set of IETs with $d=4g-3$ given by Proposition~\ref{prop:Egoodmatchings}, which admit  sequences of $L$-fold good matchings for every $L\in \mathbb{N}$.  
Let 
 $\mathcal{G}$ be the set of locally Hamiltonian flows on a surface of genus $g\ge 2$ in $\mathcal{U}_{min}$ which admit  a special flow representation over a $T\in G$, which has full measure with respect to the measure class on  $\mathcal{U}_{min}$  (see \S~\ref{sec:localHam}).  {By \cite{Ul:wea}  (see also \S~\ref{sec:main}, footnote \ref{ft:wm} at page~\pageref{ft:wm})}  there exists a 
 full measure set  $\mathcal{W}$  of locally Hamiltonian flows in $\mathcal{U}_{min}$ that  are weakly mixing. 
Let $\mathcal{F}:= \mathcal{G}\cap \mathcal{W}$ be the intersection of these two full measure sets.

Consider a locally Hamiltonian flow $\varphi_\R$ in $\mathcal{F}$. We claim that  it has disjoint rational rescalings. 
	Notice first that for every $p\in\R\setminus \{0\}$, the rescaling 
 flows $\varphi^{\kappa}_\R$ and $\varphi^{\kappa'}_\R$ are disjoint if and only if the rescalings  $\varphi^{p\kappa}_\R$ and $\varphi^{p\kappa'}_\R$ are disjoint. Hence, 
{taking into account also  Corollary \ref{cor: just positive} (which allows us to reduce to considering  only positive times)},  
  it suffices to  prove  that  $\varphi^{\kappa}_\R$ and $\varphi^{\kappa'}_\R$  for $\kappa=K$ and $\kappa'=L$ where $K, L \in \N$,  since then, the rational case follows.  
Consider a representation  of $\varphi_\R$ as special flow over $T$ in $G$ under a roof $f\in SymLog^2 (T)$. 
Since $\varphi_\R$  is isomorphic to the special flow $(T_t^f)_{t\in\R}$, one can see that then the rescalings  $\varphi^{L}_\R$ and $\varphi^{K'}_\R$ are isomorphic, respectively, to $(T_{Kt}^f)_{t\in\R}$ and  $(T_{Lt}^f)_{t\in\R}$. Thus we want to show that the latter two special flows are disjoint. 
 
Assume WLOG that $L\geq K$.  We will prove disjointness of $(T_{Kt}^f)_{t\in\R}$ and  $(T_{Lt}^f)_{t\in\R}$ via the disjointness criterion for special flows given by Proposition~\ref{prof:disjointnesssf}. The assumption that $(T_{t}^f)_{t\in\R}$ is weakly mixing holds by construction (since  $\varphi_\R$ in $\mathcal{W}$). 
{For $L$ and the $\rho>0$  given (depending on $L$) by Proposition \ref{prop:generaltail} (which gives the tail estimates) let $B>0$ be the constant
given by 
 Proposition \ref{prop:Egoodmatchings} (which gives the existence of good-matchings). Finally, let 
$ \epsilon>0 $, $ C=[c_0,c_1]\subset \mathbb{R}_+$  be given by the tail estimates  Proposition \ref{prop:generaltail}  applied to the given $L$, $\rho$ and $B$.  } 
     By  Proposition \ref{prop:Egoodmatchings} (which applies since by construction $T\in G$, which is exactly the set of IETs to which it applies), 
      there exists a sequence  $(M_n)_n$ of {$(\gamma,B)$}-good $(z_n,q_n)$ matchings 
      given by a $(C,\epsilon, \rho)$-BT rigid tower presentation (see Definition~\ref{def:givenby}).  
   Thus, by Proposition \ref{prop:exptails_via_match}, since $(M_n)_n$ are $L$ (and hence $K$)-fold $\gamma$-good $({z_n},q_n)$-matching sets for exponential tails, if we center the Birkhoff sums by $a_k=S_{q_k}f({{z_n}})$, the   sequences of measures on $\R$
 $$P_n:=(S_{Kq_n}(f)-Ka_n)_\ast (Leb_{M_n}), \quad Q_n:=(S_{Lq_n}(f)-La_n)_\ast (Leb_{M_n}), \qquad n\in \mathbb{N}$$    
 (where $Leb|_E$ denotes the restriction of the Lebesgue measure on $[0,1]$ to $E\subset [0,1]$)
  have uniform  exponential tails, i.e.~equation \eqref{eq:exptails} holds.

{
}

From the definition of BT-rigid presentations, we have that sequence  $(q_n)_n$   is a sequence of partial rigidity times with corresponding partial rigidity sets $(M_k)_{k\in\N}$ with $Leb(M_k)>1-\gamma$ {(see the proof of Lemma~\ref{cor:E_BTrigidity})}.   Without loss of generality (up to making $M_k$ smaller) we can assume that $\lim_{k\to \infty} Leb(M_k)=1-\gamma$.  
{Since the sequence $(t_k)_k$ given by Proposition~\ref{prop:generaltail} is bounded,   up to taking a subsequence, we can assume that  it converges and let $t:= \lim_{k\to \infty} t_k$.
Then, by the estimates given  by Proposition~\ref{prop:generaltail} on  the difference of measures of the $K$ and $L$ tail sets $B_k^K$ and $B_k^L$, {defined using $(t_k)_k$ as threshold   as in \eqref{def:BKL}}, we have that}
\begin{equation*}
\underline{\lim}_{n\to\infty}\frac{Leb \left(\{x \in M_n \; |S_{Kq_n}f(x)-Ka_n|\geq 
Kt \}\right) - Leb \left(\{x \in M_n: \; |S_{L q_n}f(x)-L 
a_n|\geq L t\}\right)\Big)}{Leb(M_n)}> \frac{{\gamma}}{1-\gamma} 
\end{equation*}
	This shows  that also the assumption \eqref{differenttail} of  Proposition~\ref{prof:disjointnesssf} holds.  
Thus, all the assumptions of the disjointness criterion given by Proposition~\ref{prof:disjointnesssf} hold and by the criterion we conclude, as claimed, that 
 the $K$ and $L$ rescalings 
   $(T_{Kt}^f)_{t\in\R}$ and  $(T_{Lt}^f)_{t\in\R}$ of $(T_t^f)_{t\in\R}$ are disjoint in the sense of Furstenberg.
	\end{proof}	}

\subsection*{Acknowledgements} 
{P.~B.~has received support from the Swiss National Science Foundation, through Grant 
 200021\textunderscore 188617/1 
{as well as the grant OPUS 29 2025/57/B/ST1/00704 from the National Science Centre (Poland).}  C.~U.~acknowledges also the support of the Swiss National Science Foundation through the SNSF Consolidator Grant Number 213663, \emph{New Frontiers of Renormalization}.}
 
\end{document}